\documentclass[10pt,letterpaper]{article}
\usepackage[utf8]{inputenc}
\usepackage{fullpage}
\usepackage{mathtools}
\usepackage{amsmath}
\usepackage{amssymb}
\usepackage{amsthm}
\usepackage{graphicx} 
\usepackage{hyperref}
\usepackage{authblk}
\usepackage{comment}
\usepackage{subcaption}
\usepackage{xcolor}
\usepackage{nccmath}
\usepackage[parfill]{parskip}
\usepackage{multicol}
\usepackage{multirow}

\DeclareUnicodeCharacter{FFFD}{ }

\usepackage{hoang}

\title{Stochastic Approximation for Nonlinear Discrete Stochastic Control: Finite-Sample Bounds}
\author{Hoang Huy Nguyen\thanks{Georgia Institute of Technology, \url{hnguyen455@gatech.edu}}, Siva Theja Maguluri\thanks{Georgia Institute of Technology, \url{siva.theja@gatech.edu}}}
\date{}

\begin{document}

\maketitle

\begin{abstract}
We consider a nonlinear discrete stochastic control system, and our goal is to design a feedback control policy in order to lead the system to a prespecified state. We adopt a stochastic approximation viewpoint of this problem. It is known that by solving the corresponding continuous-time deterministic system, and using the resulting feedback control policy, one ensures almost sure convergence to the prespecified state in the discrete system. In this paper, we adopt such a control mechanism and provide its finite-sample convergence bounds whenever a Lyapunov function is known for the continuous system. In particular, we consider four cases based on whether the Lyapunov function for the continuous system gives exponential or sub-exponential rates and based on whether it is smooth or not. We provide the finite-time bounds in all cases. Our proof relies on constructing a Lyapunov function for the discrete system based on the given Lyapunov function for the continuous system. We do this by appropriately smoothing the given function using the Moreau envelope. We present numerical experiments corresponding to the various cases, which validate the rates we establish. 
\end{abstract}

\section{Introduction}
In this paper, we consider the problem of controlling a nonlinear discrete stochastic system of the form,
\begin{align}
    \label{eqn: control-framework-discrete}
    x_{k+1} = x_k + \alpha_k(F(x_k, u_k) + w_k)
\end{align}
where  $x_k \in \R^d$ is the state vector, $u_k \in \R^m$ is the control, $w_k$ is the noise, $F(\cdot,\cdot)$ is in general a nonlinear mapping, and $\alpha_k$ is a sequence of step-sizes. The goal is to pick the control sequence $u_k$ in order to ensure that the system reaches a pre-specified state $x^*$. We will focus on feed-back control strategies of the form $u_k = \rho(x_k)$ for some mapping $\rho:\R^d \rightarrow \R^m$ to reach the state $x^*$. In this paper, we provide finite-time convergence bounds on the error, $\norm{x_k-x^*}$ depending on the choice of the step-size sequence $\alpha_k$.

Such a problem arises in several control settings such as adaptive regulation \cite{HanFuChenSAApplication}, \cite{hanfuchen-adaptive-regulator} with widespread applications such as disturbance rejection in spacecraft systems \cite{spacecraft-adaptive-regulation} and frequency regulation in power systems \cite{adaptive-regulation-power-system}; or selector control with an application in air-fuel control \cite{PID-controllers-book}. 
A special case of our framework \eqref{eqn: control-framework-discrete} is the unforced setting without a $u_k$, and in this case, one is interested in the convergence rate to the equilibrium $x^*$. A large class of stochastic optimization \cite{beck-book} and reinforcement learning \cite{zaiwei-envelope} algorithms fit in this framework, and our results immediately apply to this case.

Stochastic recursions of the form \eqref{eqn: control-framework-discrete} are studied under the name of Stochastic Approximation \cite{Borkar2008StochasticAA}, and were first introduced by Robins and Monro \cite{Robbins&Monro:1951}. Asymptotic behavior of such recursions is well understood \cite{10.5555/2408520, Kushner-Clark-SA-book, Borkar2008StochasticAA}
in terms of the behavior of the corresponding continuous-time deterministic control system, 
\begin{align}
    \label{eqn: control-framework-continuous}
    \dot{x} = F(x,u) \,\, \forall x \in \mathbb{R}^d
\end{align}
In particular, it is known that \cite{Borkar2008StochasticAA} under appropriate assumptions on the system, noise and choice of step-sizes, the almost-sure asymptotic behavior of the discrete-stochastic system \eqref{eqn: control-framework-discrete} is identical to that of the continuous-deterministic system \eqref{eqn: control-framework-continuous}. Then, in order to lead the system to state $x^*$, one would find the optimal feedback solution $u = \rho(x(t))$ for the continuous system \eqref{eqn: control-framework-continuous} and use the same solution for the discrete system \eqref{eqn: control-framework-discrete}. The objective of this paper is to characterize the finite-time convergence error in such an approach. 

Naturally, the convergence rate of the discrete system depends on that of the continuous system. The convergence behavior as well as the rate of convergence of the continuous system is usually studied using Lyapunov arguments. In this paper, we characterize the convergence rate of the discrete system \eqref{eqn: control-framework-discrete} based on the properties of the Lyapunov function of the continuous system \eqref{eqn: control-framework-continuous}. Suppose that there exists a feedback control policy and the corresponding Lyapunov function $V$ for the continuous system that satisfies 
\begin{align}
\label{eqn: stability-assumption-continuous-time}
\frac{dV}{dt} = \langle \nabla V(x), F(x) \rangle \leq -\gamma V(x)^c \,\, \forall x \in \mathbb{R}^d \text{ for some }\gamma > 0, c \geq 1.
\end{align}
The convergence rate of the discrete system was established in the literature (in the context of optimization \cite{beck-book} and reinforcement learning \cite{nonlinear-sa}) in the special case when $c=1$ (also known as the global exponential stability \cite{expo-stabilizing-control-sreenath}) and $V(\cdot)$ is smooth (i.e., has Lipschitz gradients).
In contrast, in this paper, we consider the case when the continuous system can either have exponential or sub-exponential rates ($c \geq 1$), while also allowing for non-smooth Lyapunov functions. 
The main contributions of the paper are as follows.



\subsection{Main contributions}
\textbf{Results}: We establish finite-time bounds and the following sample complexity (that is, the number of operator sampling required) results for the SA algorithm for the control problem in \eqref{eqn: control-framework-discrete} based on the properties of the Lyapunov function \eqref{eqn: stability-assumption-continuous-time} for the corresponding continuous-time system \eqref{eqn: control-framework-continuous} under appropriate assumptions on noise and growth of the Lyapunov function in all four cases. 
\begin{enumerate}
\item For completeness, we first present the case of exponential stability ($c=1$) and smooth Lyapunov function $V(\cdot)$ which follows from \cite{nonlinear-sa}. We show that one needs $O\br{\frac{1}{\varepsilon}}$ samples to ensure that the mean square error $\E[\norm{x_k-x^*}^2] \leq \varepsilon$. 
\item In the exponential case ($c=1$) with a non-smooth Lyapunov function, we again establish the same sample complexity.
\item For the sub-exponential case ($c>1$) with a smooth Lyapunov function, we show a sample complexity of $O\left( \frac{1}{\varepsilon^{2c-1}} \right)$.
\item Finally, in the non-smooth sub-exponential case, we establish a sample complexity of $O \left( \frac{1}{\varepsilon^{3d_{a,c}-1}} \right)$ where $a$ is the polynomial growth rate of the Lyapunov function and $d_{a,c} = a(c-1)/2+1$. We believe that this is not optimal, and obtaining tighter results is an open problem. 
\end{enumerate}
Our work is a generalization of \cite{nonlinear-sa} and an extension of \cite{hoang-nonlinear-sa-cdc} where these works establish convergence results for the smooth and non-smooth exponential cases whereas we extend the analysis to the smooth and non-smooth subexponential cases. We believe the analysis of subexponential settings may be of independent interest, such as subgeometric Markov chain mixing \cite{polynomial-convergence-markov-chain, Butkovsky_2014_subgeometric_wasserstein, durmus2015subgeometricratesconvergencewasserstein, qu-glynn2023wasserstein-contractive-drift}.

\textbf{Methodology}: 
Our key idea is to use the Lyapunov function from the continuous system to study the convergence rate of the discrete system. This works in the case when the Lyapunov function is smooth, even though the rate one obtains in the discrete system is worse than that of the continuous system due to the presence of noise. However, this approach does not work when the Lyapunov function is not smooth. In other words, in the absence of smoothness, one cannot get a handle on the errors between the continuous and discrete systems. So, in the exponential nonsmooth case, we construct a new Lyapunov function obtained by smoothing out the Lyapunov function of the continuous system. We do this by applying the infimal convolution \cite{beck-smoothing} with respect to the $\ell_2$ norm square function. This is called the Moreau envelope \cite{Moreau1965ProximitED, zaiwei-envelope}.
Finally, in the nonsmooth subexponential case, naive smoothing does not yield convergence, and so, we use a time-varying smoothness parameter to obtain a time-varying Moreau envelope. The use of time-varying Moreau envelopes may be of independent interest.

\textbf{Numerical Experiments}: In addition to the theoretical analysis, we also empirically validate our results in the latter three cases using three examples: Selector control problem \cite{Johansson2003PiecewiseLC} (non-smooth exponentially stable system), Artstein's circle example (non-smooth sub-exponentially stable system) \cite{nonsmooth-lyapunov-caratheodory}, and a synthetic nonlinear system example from \cite{khalil-book} for smooth sub-exponentially stable systems.

\subsection{Relevant literature}
\label{sec:literature}
Finding the equilibrium point of stabilizing control problems is essentially a root-finding problem. From this perspective, the equilibrium point finding problems can be solved through the framework of Stochastic Approximation (SA) algorithms, which were first proposed by \cite{Robbins&Monro:1951}. The asymptotic convergence of SA methods was analyzed using its associated ordinary differential equations (ODE) \cite{10.5555/2408520, Kushner-Clark-SA-book}. More specifically, it was shown in \cite{Borkar2008StochasticAA, Ljung-SA} that under some conditions, the SA algorithm converges almost surely as long as the corresponding ODE is stable. The asymptotic convergence of SA algorithms with Markovian noise has been studied extensively \cite{Borkar2008StochasticAA, neurodynamic}. The convergence of other SA variants such as SA with Markovian noise and multiple time-scale SA was
previously explored in \cite{karmakar-markovian-sa, sa-controlled-markov-noise, thinh-nonlinear-two-time-sa, bhatnagar_borkar_multiscale_1997, bhatnagar_borkar_twoscale_1998}. In contrast to asymptotic convergence guarantees in the literature, here we focus on obtaining finite-time bounds which enable us to provide sample complexity bounds. 

In order to analyze the stability of dynamical systems, it is common to employ control Lyapunov functions. In particular, one shows that the time derivative of the Lyapunov function is upper bounded by some negative definite function to show asymptotic stability (or some negative constant times the Lyapunov function itself to achieve exponential stability) \cite{khalil-book}. In this work, we seek to obtain finite-time analyses for nonlinear systems with a Lyapunov descent condition or some decay conditions. In the context of subgradient optimization, \cite{davis-subgradient-tame} showed that descent properties hold for any function with a Whitney stratifiable graph, covering a wide range of functions in optimization, control and 
Machine Learning. Indeed, such conditions are found in many settings and applications in control \cite{expo-stabilizing-control-sreenath, osinenko-stabilization, exponential-cbf-high-safety}, electrical systems \cite{decentralized-safe-rl-voltage-control}, robotics \cite{lyapunov-robotic-rl}, diffusion processes \cite{wenlong-diffusion}, queuing theory \cite{hoang-erlang-c-mixing}, and reinforcement learning \cite{rl-safety-clf-cbf, rl-primary-frequency-control, sajad-federated-rl}. 
In optimization, the Polyak-Lojasiewicz condition \cite{POLYAK1963864} establishes that the $\ell_2$ norm serves as a Lyapunov function with exponential stability. 
Previously, \cite{nonlinear-sa} established finite-time analysis for nonlinear SA using an exponential dissipativity assumption, which is equivalent to exponential stability condition for the $\ell_2$ norm Lyapunov function. However, due to its reliance on the dissipativity assumption, its results are limited to the case of $\ell_2$ norm Lyapunov function and exponential stability. In this paper, we generalize the results to the case of subexponential nonlinear nonsmooth Lyapunov functions. 

In practice, many control Lyapunov functions are non-smooth \cite{osinenko-stabilization} since the control inputs are usually measured in discrete time rather than having continuous measurements, such as logical systems \cite{boolean-networks-stability} or approximate discrete-time models \cite{8606612}. In addition, discontinuous stabilizing feedback and non-smooth Lyapunov functions are deeply connected \cite{Clarke-discontinous-feedback} as the non-existence of a smooth control Lyapunov function implies the absence of a continuous stabilizing feedback law \cite{osinenko-stabilization}. The theory of Lyapunov stability for non-smooth systems was developed by \cite{nonsmooth-lyapunov}, where the Clarke generalized gradient \cite{clarke1983oan} was used to complement the lack of gradient as we usually have in the smooth setting. The non-smooth Lyapunov analysis of equilibria is present in the differential inclusions literature \cite{clarke1983oan, aubin1984differential, frankowska-lower-semicontinuous}, with applications in robotics \cite{paden-calculus-filippov} and non-smooth Stochastic Approximation \cite{nonsmooth-sa-di}. The work \cite{khalil-book} uses the ODE approach to handle the time derivative condition. A special application of non-smooth SA is the switching SA algorithm which is used in networked systems \cite{switching-sa}. However, these prior works did not have a finite-time guarantee for the non-smooth stochastic systems.

In order to deal with the non-smoothness of the system, one can attempt smoothing methods to yield a smoothened Lyapunov function. There are several approaches to perform smoothing. \cite{cutkosky2023optimalstochasticnonsmoothnonconvex, hazan2023introductiononlineconvexoptimization} consider a convolution-based approach. On the other hand \cite{beck-smoothing, Moreau1965ProximitED} use infimal convolution-based smoothing which has been widely used in non-smooth optimization \cite{dru-weakly-convex}. Specifically, the authors used the norm of the gradient of the Moreau envelope as a measure of progress, which is a proxy for the distance from $0$ to the set of generalized gradients. In the context of SA, \cite{zaiwei-envelope} was the first to use Moreau envelopes to obtain convergence rates for the non-smooth infinity norms, which is common for the analysis of Reinforcement Learning algorithms, and \cite{Chen2025-sa-concentration} utilizes this result to obtain concentration bounds for SA iterates. However, the authors relied on the contractive property for analysis, which can be limiting as it excludes a wide range of potential operators without such property and the systems in control settings may not have such properties. The paper \cite{osinenko-stabilization} used the inf-convolution operator to analyze the stochastic stabilization when a control policy exists. In contrast to these prior works, we consider general nonlinear control problems and study them under various drift/stability conditions. The case when one has exponential stability and nonsmooth Lyapunov function  (Section \ref{ssec: non-smooth exponential}) can be thought of as a natural generalization of the setting in \cite{zaiwei-envelope}, and we show that the smoothing of \cite{zaiwei-envelope} approach works well beyond the contractive setting as long as one has exponential stability. In addition, we study the cases where one has subexponential stability, which was not studied in the prior literature, and we show that a naive adaptation of the smoothing from \cite{zaiwei-envelope} does not work, and so, we independently develop a time-varying smoothing approach, which is also previously used in composite optimization \cite{variable-smoothing-weakly-convex}.


\subsection{Paper organization}
The paper is organized as follows. Section \ref{sec:prelim} presents the necessary preliminaries such as the notations, the system model, and the general assumptions used in this work. Section \ref{sec:main-results} presents the problem formulation, assumptions, and results in all four cases. We present the finite-time convergence bounds and also show that they immediately imply almost sure convergence results. 
Section \ref{sec:proof-outline} presents the proof outline and the key ideas in the proofs. Section \ref{sec:experiments} presents numerical experiments demonstrating the convergence rate in various examples. Finally, complete proofs of the main results are presented in the Appendix.

\section{Preliminaries}
\label{sec:prelim}
\subsection{Notations}
\label{ssec: notations}
To assist readers' comprehension, we declare a few notations that will be used throughout this work. We denote $\norm{\cdot}_p$ as the $\ell_p$ norm, the distance of a point $x$ to a set $\cX$ as $\operatorname{dist}(x_k,\cX)$ and for any function $f: \cX \rightarrow \R$, we define the proximal operator as $\operatorname{prox}$ where
\begin{align}
    \label{eqn: prox-definition}
    \operatorname{prox}_f(x) = \arg \min_{u \in \cX} \br{ f(u) + \frac{\norm{x-u}^2}{2} }.
\end{align}
Note that we also have
\begin{align}
    \label{eqn: scaled-prox}
    \operatorname{prox}_{\lambda f}(x) = \arg \min_{u \in \cX} \br{ \lambda f(u) + \frac{\norm{x-u}^2}{2} } = \arg \min_{u \in \cX} \br{ f(u) + \frac{\norm{x-u}^2}{2\lambda} }.
\end{align}
In addition, for any two functions $f,g: \cX \rightarrow \R$, we define operator $\square$ as the infimal convolution operation where
\begin{align}
    \label{eqn: inf-convolution-definition}
    f \square g (x) = \min_{u \in \cX} \left\lbrace f(u) + g(x-u)\right\rbrace.
\end{align}
From here, we go on to define the Moreau envelope $M_\mu$ of the function $f: \cX \rightarrow \R$ with respect to the smoothness parameter $\mu$ as
\begin{align}
    \label{eqn: moreau-envelope-definition}
    M_\mu^f(x) = f \square \frac{\norm{\cdot}^2}{2\mu} = \min_{u \in \cX} \br{ f(u) + \frac{\norm{x-u}^2}{2\mu} }.
\end{align}
With a slight abuse of notation, we write $M_\mu$ in our manuscript as the Moreau envelope of some function to be defined.

\subsection{System model and general assumptions}
\label{ssec: assumptions}
Consider the following noiseless discrete dynamics:
\begin{align}
    \label{general-update-rule}
    x_{k+1} = x_k + \alpha_k F(x_k,u_k)
\end{align}
where $\F: \mathbb{R}^d \rightarrow \mathbb{R}$ be a possibly nonlinear mapping, $x, u \in \R^d$ are the state vector and the control vector respectively. Suppose that the control $u$ follows a feed-in control rule $u = \rho(x)$ where $\rho:\R^d \rightarrow \R^m$, the core of our interests is finding the solution $x^*$ to the stationary point equation $F(x,\rho(x)) = 0$. In control, this is equivalent to finding a control that stabilizes our system given the state.\\
\\
Now, suppose that we can only obtain the value of $F$ via a noisy oracle $\tilde{F}$ such that for any $x$ it will return $\tilde{F}(x,u) = F(x,u) + w$ where $w$ is the noise (which can be dependent on the state value $x$ and the control $u$). 
\begin{align}
    \label{update-rule}
    x_{k+1} = x_k + \alpha_k(F(x_k,u_k) + w_k)
\end{align}
This happens when the control agent tries to obtain environmental inputs via some device, which can be contaminated by noise or device inaccuracies. Let $\mathcal{F}_k = \{x_0,w_0,...,x_{k-1},w_{k-1},x_k\}$ where $\{w_k\}$ is a martingale difference sequence with some mild conditions on its variance. We have some assumptions on the noise $w_k$ as follows:
\begin{assumption}
\label{assumption: noise}
(Noise assumptions) The noise $w_k$ is unbiased, that is for any $k \in \Z^+$:
\begin{align}
    \label{assumption: unbiased-noise}
    \mathbb{E}[w_k | \mathcal{F}_k] = 0
\end{align}
and the noise is square-integrable. That is for $e > 0$:
\begin{align}
    \label{assumption: bounded-variance}
    \mathbb{E}[\norm{w_k}_e^2 | \mathcal{F}_k] \leq A + B\norm{x_k}_e^2.
\end{align}
\end{assumption}
When $B > 0$ in \eqref{assumption: bounded-variance}, the magnitude of the so-called multiplicative noise can potentially scale with $x_k$ \cite{Chen2025-sa-concentration, lam-nguyen-hogwild}. In addition to these noise assumptions, we also assume that $F$ is Lipschitz:
\begin{assumption}
\label{assumption: F-lipschitz} (Lipschitz assumption of $F$) There exists a positive constant $C$ such that:
\begin{align}
    \label{lipschitz-assumption}
    \norm{F(x,u) - F(y,u)}_e \leq C\norm{x-y}_e
\end{align}
for any $x,y \in \R^d$ and $u \in \R^m$.
\end{assumption}
Unless specified otherwise, we assume $e = 2$ for Assumption \ref{assumption: noise}, \ref{assumption: F-lipschitz}. In the Stochastic Gradient Descent (SGD) algorithm where $F(x) = -\nabla V(x)$, this assumption is the gradient Lipschitz assumption that is vital to ensure the stability of the algorithm. In addition, this assumption is a much more relaxed assumption than the usual contractive operator assumption used in \cite{zaiwei-envelope}. Not having to rely on the contractive property will allow us to apply our results to a wider range of problems rather than just the optimal control problem with time-discounted rewards. From these assumptions, we will proceed to analyze the systems in different settings as follows:

\section{Main Results}
\label{sec:main-results}

\subsection{Stochastic control of smooth exponentially stable systems}
\label{ssec: smooth-exponential}
First, for completeness, we will consider smooth, exponentially stable systems. This setting is widespread in Stochastic Optimization (which is analogous to the smooth, strongly-convex case \cite{gd-polyak-lojasiewicz}), Markov Chain mixing \cite{hoang-erlang-c-mixing} and Reinforcement Learning \cite{nonlinear-sa}. To that end, let us assume that we have the following assumptions of $V$:
\begin{assumption}
\label{assumption: smooth-time-derivative}
(Exponential stability assumption of the Lyapunov function) Assume that the gradient of $V$ exists everywhere, we have that:
 \begin{align}
    \label{eqn: smooth-time-derivative}
    \langle \nabla V(x), F(x) \rangle \leq - \gamma V(x) \forall x \in \R^d
\end{align}
for some constant $\gamma > 0$.
\end{assumption}
In the Stochastic Gradient Descent (SGD) algorithm, we have that $F(x) = -\nabla V(x)$. Thus, the Assumption \ref{assumption: smooth-time-derivative} is equivalent to $\norm{\nabla V(x)}^2 \geq \gamma V(x) \forall x \in \R^d$, which is the Polyak-Lojasiewicz condition. Combining with the gradient Lipschitz assumption, its convergence results are known \cite{gd-polyak-lojasiewicz}. Here, we are considering a more general operator $F$ rather than a gradient of some functions. In addition, in the context of continuous dynamical systems, such conditions imply the global exponential stability of the system, meaning that the system will approach the equilibrium point exponentially fast. This is because the LHS of \eqref{eqn: smooth-time-derivative} is somewhat approximating the time derivative of $V$ since $\dot{V}(x) = \langle \nabla V(x), \dot{x} \rangle \approx \langle \nabla V(x), F(x) \rangle$.

\begin{assumption}
\label{assumption: gradient-lipschitz}
(Lipschitz gradient assumption) There exists a positive constant $L$ such that:
\begin{align}
    V(y) \leq V(x) + \langle \nabla V(x), y-x \rangle + \frac{L\norm{x-y}_e^2}{2} \,\, \forall x,y \in \mathbb{R}^d
\end{align}
\end{assumption}
Unless specified otherwise, we assume $e = 2$ by default. The gradient Lipschitz assumption is commonly found in optimization literature and is commonly used to obtain a first-order upper bound on the objective function. In addition, this assumption also implies that $V$ is upper-bounded by some quadratic function as well. Indeed, from Assumption \ref{assumption: gradient-lipschitz} we have:
\begin{align*}
    V(x) \leq V(x^*) + \frac{L\norm{x-x^*}_e^2}{2} \,\, \forall x \in \mathbb{R}^d
\end{align*}
hence we have $V$ is upper-bounded by some quadratic function. This gives rise to the following assumption:
\begin{assumption}
\label{assumption: quadratic-growth}
(Quadratic growth assumption of $V$) There exists positive constants $C_1,C_2$ such that:
\begin{align}
    C_1 \norm{x-x^*}_e^2 \leq V(x) \leq C_2\norm{x-x^*}_e^2 \,\, \forall x \in \R^d
\end{align}
for some $e > 0$.
\end{assumption}
Unless specified otherwise, we consider $e = 2$ by default. Now, we are ready to state the following results:
\begin{theorem}
\label{finite-time-corollary-0}
Under Assumptions \ref{assumption: noise}, \ref{assumption: F-lipschitz}, \ref{assumption: smooth-time-derivative}, \ref{assumption: gradient-lipschitz}, \ref{assumption: quadratic-growth} and let the step size $\alpha_k = \frac{\alpha}{(k+K)^\xi}$ where $K = \max\left\lbrace 1, \frac{\alpha (4C^2 + 8B)}{\gamma} \right\rbrace$ for $\xi = 1$ and $K = \max \left\lbrace 1, \left( \frac{\alpha (4C^2 + 8B)}{\gamma} \right)^{\frac{1}{\xi}}, \left( \frac{2 \xi}{\alpha \gamma} \right)^{\frac{1}{1-\xi}}\right\rbrace$ for $\xi \in (0,1)$, we have:
\begin{align*}
    &\textbf{For $\xi = 1$}: \\
    &\E[\norm{x_{k}-x^*}^2] \leq \frac{C_2}{C_1} \norm{x_0-x^*}^2 \br{\frac{K}{k+K}}^{\frac{\alpha \gamma}{2}} + 
    \begin{cases}
        O\br{\frac{1}{k^{\frac{\alpha \gamma}{2}}}} &\text{ if } \alpha \in \br{0, \frac{2}{\gamma}}\\
        O\br{\frac{\log k}{k}} &\text{ if } \alpha = \frac{2}{\gamma} \\
        O\br{\frac{1}{k}} &\text{ if } \alpha \in \br{\frac{2}{\gamma}, \infty}
    \end{cases}
    \\
    &\textbf{For $\xi \in (0,1)$}: \\
    &\mathbb{E}[\norm{x_{k}-x^*}^2] \leq \frac{C_2}{C_1} \norm{x_0-x^*}^2 \exp\left[-\frac{\alpha \gamma}{2(1-\xi)}((k+K)^{1-\xi}-K^{1-\xi}) \right] + \frac{4\alpha L (A + 2B\norm{x^*}^2)}{\gamma (k+K)^{\xi}} \qquad\qquad \\
    &\textbf{For $\xi = 0$}: \\
    &\mathbb{E}[\norm{x_{k}-x^*}^2] \leq \frac{C_2}{C_1} \norm{x_0-x^*}^2 \left( 1-\frac{\alpha \gamma}{2} \right)^k + \frac{2L (A + 2B\norm{x^*}^2)\alpha}{\gamma}
\end{align*}
\end{theorem}
A typical function $V$ satisfying Assumption \ref{assumption: smooth-time-derivative}, \ref{assumption: gradient-lipschitz}, \ref{assumption: quadratic-growth} is the quadratic function $x^TAx$ where $A$ is a positive definite matrix. The exact statement and the proof of this result follow from \cite{nonlinear-sa} and the Subsection \ref{ssec: non-smooth exponential} in this work so we omit the details here. The results in Theorem \ref{finite-time-corollary-0} show that the best asymptotic complexity is $O\br{\frac{1}{k}}$ when $\xi = 1, \alpha \geq \frac{2}{\gamma}$, which matches the optimal complexity of the SGD algorithm for the strongly-convex setting. Note that we present this result for the sake of completeness. In the subsequent subsections, we will consider more difficult settings that would be the core of our work. 

\subsection{Stochastic control of non-smooth exponentially stable systems}
\label{ssec: non-smooth exponential}
In many control applications, oftentimes we find that having smoothness is a luxury. For example, the states are usually measured in discrete times rather than being continuously measured since most device is not capable of doing such a task. Thus, this necessitates an analysis of the non-smooth setting, which is far more widespread in real-world applications.

In absence of a smooth $V$, even many assumptions in the smooth exponentially stable case will cease to be meaningful. For example, the Assumption \ref{assumption: smooth-time-derivative} does not hold since the gradient does not exist everywhere in the non-smooth setting. Thus, we define an analogous Assumption \ref{assumption: smooth-time-derivative} using Clarke generalized gradient \cite{clarke1983oan, Clarke1998-nonsmooth-analysis-control-theory} as following: for a locally Lipschitz function $V$, define the generalized gradient of $V$ at $x$ by:
\begin{align}
    \label{eqn: clark-gradient-definition}
    \partial V(x) = \overline{co}\{\lim \nabla V(x_i) | x_i \rightarrow x, x_i \in \Omega_V \}
\end{align}
where $\overline{co}$ denotes the closed convex hull and $\Omega_V$ is the set of points where the gradient of $V$ exists. The time derivative Assumption \ref{assumption: smooth-time-derivative} can be written as:
\begin{assumption}
\label{assumption: clarke-time-derivative}
(Exponential stability assumption of the Lyapunov function with the Clarke generalized gradient) Let $x \in \R^d$, we have that the following holds:
\begin{align}
    \label{clarke-time-derivative}
    \langle g_x, F(x) \rangle \leq -\gamma V(x) \forall x \in \mathbb{R}^d, g_x \in \partial V(x)
\end{align}
\end{assumption}
In addition, we also assume that the value of $V$ is bounded by some polynomial whose input is the distance to $x^*$, that is:
\begin{assumption}
\label{assumption: polynomial-growth} (Polynomial growth assumption of the Lyapunov function) There exists positive constants $C_{1,a},C_{2,a}$ such that:
\begin{align}
    C_{1,a} \norm{x-x^*}^a \leq V(x) \leq C_{2,a} \norm{x-x^*}^a \,\, \forall x \in \R^d.
\end{align}
\end{assumption}
This assumption controls the magnitude of $V$ given its distance to $x^*$ and allows us to obtain a bound on $\norm{x-x^*}$ using $V$. Moreover, since $\min_{x \in \R^d} V(x) = V(x^*) = 0$, the left inequality can be interpreted as the sharpness condition in the optimization literature \cite{Davis2018-subgradient-sharp-weakly, ding2024sharpnesswellconditioningnonsmoothconvex}. Next, we have the norm of the generalized gradient of $V$ at $x$ is also bounded by some polynomial as well:\begin{assumption}
\label{assumption: generalized-gradient-bound}
(Gradient growth assumption) Let $x \in \R^d$, for any $g_x \in \partial V(x)$, we have that:
\begin{align}
    \label{generalized-gradient-bound-assumption}
    \norm{g_x} \leq G \norm{x-x^*}^{a-1}
\end{align}
\end{assumption}
For $a = 2$, this condition implies a linear upper bound on the generalized gradient. The linear growth assumption of the gradient can be found in several works in non-smooth optimization \cite{yu-constant-stepsize} and control \cite{khalil-book}. In the smooth case, this assumption automatically holds since the Assumption \ref{assumption: gradient-lipschitz} implies linear growth of the gradient.\\
\\
Now that we have established the necessary prerequisites, we have the following finite-time bounds:
\begin{theorem}
\label{finite-time-corollary-1}
Under Assumptions \ref{assumption: noise},  \ref{assumption: F-lipschitz}, \ref{assumption: clarke-time-derivative}, \ref{assumption: polynomial-growth}, \ref{assumption: generalized-gradient-bound}, and with the step size $\alpha_k = \frac{\alpha}{\br{k+K}^{\xi}}$ where $K = \max\left\lbrace 1, \frac{\alpha (4C^2 + 8B)}{\gamma} \right\rbrace$ for $\xi = 1$ and $K = \max \left\lbrace 1, \left( \frac{\alpha (4C^2 + 8B)}{\gamma} \right)^{\frac{1}{\xi}}, \left( \frac{2 \xi}{\alpha \gamma} \right)^{\frac{1}{1-\xi}}\right\rbrace$ for $\xi \in (0,1)$, we have the finite-time bound of the SA algorithm is:
\begin{align*}
&\textbf{For $\xi = 1$}:\\
&\mathbb{E}[\norm{x_{k}-x^*}^2] \leq C_{2,a}^{\frac{2}{a}}\br{\frac{1}{C_{1,a}^{\frac{2}{a}}} + 2\mu} \norm{x_0-x^*}^2 \left( \frac{K}{k+K} \right)^{\frac{\alpha \gamma_M}{2}} + 
    \begin{cases}
        O\br{\frac{1}{k^{\frac{\alpha \gamma_M}{2}}}} &\text{ if } \alpha \in \br{0, \frac{2}{\gamma_M}}\\
        O\br{\frac{\log k}{k}} &\text{ if } \alpha = \frac{2}{\gamma_M} \\
        O\br{\frac{1}{k}} &\text{ if } \alpha \in \br{\frac{2}{\gamma_M}, \infty}
    \end{cases} \qquad\qquad\qquad \\
&\textbf{For $\xi \in (0,1)$}:\\
&\mathbb{E}[\norm{x_{k}-x^*}^2] \leq C_{2,a}^{\frac{2}{a}}\br{\frac{1}{C_{1,a}^{\frac{2}{a}}} + 2\mu} \norm{x_0-x^*}^2 \exp\left[-\frac{\alpha \gamma_M}{2(1-\xi)}((k+K)^{1-\xi}-K^{1-\xi}) \right] \\
&\qquad\qquad\qquad + \frac{4\alpha}{\gamma_M (k+K)^{\xi}} \frac{A + 2B\norm{x^*}^2}{\mu}\br{\frac{1}{C_{1,a}^{\frac{2}{a}}} + 2\mu}. \\
&\textbf{For $\xi = 0$}: \\
&\mathbb{E}[\norm{x_{k}-x^*}^2] \leq C_{2,a}^{\frac{2}{a}}\br{\frac{1}{C_{1,a}^{\frac{2}{a}}} + 2\mu} \norm{x_0-x^*}^2 \left( 1-\frac{\alpha \gamma_M}{2} \right)^k + \frac{2(A + 2B\norm{x^*}^2)\alpha}{\mu \gamma_M}\br{\frac{1}{C_{1,a}^{\frac{2}{a}}} + 2\mu}.
\end{align*}
\end{theorem}

\textcolor{black}{Typical Lyapunov functions satisfying Assumptions \ref{assumption: clarke-time-derivative}, \ref{assumption: polynomial-growth}, \ref{assumption: generalized-gradient-bound} are $\norm{x-x^*}_\infty^2$ or piecewise quadratic functions which can be found in Reinforcement Learning \cite{zaiwei-envelope} or selector control \cite{Johansson2003PiecewiseLC}. Regarding the main result, we have the following discussions.}

\textbf{Discussion of step sizes}: When $\xi = 0$, we do not have convergence but the algorithm converges exponentially fast to some ball, which is in the interests of practitioners. When $\xi \in (0,1]$, we have mean square convergence to the point $x^*$, and the optimal convergence rate of the exponentially stable setting is $O(1/k)$ when $\xi = 1$. Notably, the convergence rate is dependent on the initial step size in this case. On the other hand, when $\xi \in (0,1)$, the convergence rate is independent of the step size, and hence, it is robust. \textcolor{black}{To address the parameter dependence issue when $\xi = 1$, we can employ Polyak-Ruppert averaging $x_n^{PR} = \frac{1}{n} \sum_{k=1}^n x_k$ and assume the following assumptions.
\begin{assumption}
    \label{assumption: ode-vector-field}
    (ODE vector field assumption, Assumption A3 in \cite{kontoyiannisborkar2024odemethodasymptoticstatistics}) The scaled vector field $F_\infty(\theta) = \lim_{r \rightarrow \infty} \frac{F(r\theta)}{r}$ exists for any $\theta \in \R^d$.
\end{assumption}
Now, we would assume an extra condition on the noise, which can be loosely interpreted as the ergodicity condition of the noise.
\begin{assumption}
    \label{assumption: dv3}
    (Assumption DV3 in \cite{kontoyiannisborkar2024odemethodasymptoticstatistics}) For $b > 0$, functions $V: X \rightarrow \R_+, W: X \rightarrow [1,\infty)$ and a small function $s: X \rightarrow [0,1]$ (defined as in \cite{kontoyiannisborkar2024odemethodasymptoticstatistics}) such that
    \begin{align}
        \E\sqbr{e^{V(\theta_{k+1})} \big| \theta_k = x} \leq e^{V(x) - W(x) + b s(x)}.
    \end{align}
\end{assumption}
Then, we obtain the following Central Limit Theorem (CLT) for the nonlinear Stochastic Approximation algorithm from \cite{kontoyiannisborkar2024odemethodasymptoticstatistics}.
\begin{theorem}
    \label{thm: clt-exponential}
    Under step size $\alpha_k = \alpha (k+1)^{-\xi}$ for $\xi \in (1/2, 1]$, Assumptions \ref{assumption: F-lipschitz}, \ref{assumption: clarke-time-derivative}, \ref{assumption: polynomial-growth}, \ref{assumption: ode-vector-field}, \ref{assumption: dv3} such that $\gamma > 4\max\{C, \norm{F(0)}\} \sqbr{a\log 2 + \frac{C_{2,a}}{C_{1,a}}}$, there exists a covariance matrix $\Sigma^{PR}$ so that
    \begin{align}
        \lim_{n \rightarrow \infty} n \E\sqbr{(x_n^{PR}-x^*)(x_n^{PR}-x^*)^T} = \Sigma^{PR}.
    \end{align}
\end{theorem}
The proof of this theorem follows from Theorem 5 of \cite{kontoyiannisborkar2024odemethodasymptoticstatistics}, and we verify the assumptions of this theorem in Appendix \ref{ssec: clt-proof}. This asymptotic result implies that under the Polyak-Ruppert averaging scheme and for a sufficiently large $n$, we have $\E\sqbr{(x_n^{PR}-x^*)(x_n^{PR}-x^*)^T}$ converges with rate $O(1/n)$, matching the optimal convergence rate of the exponentially stable setting.}

\textbf{Connection to contractive operators}: Our result in Theorem \ref{finite-time-corollary-1} can be used to obtain finite-time guarantees for contractive operators in \cite{zaiwei-envelope}. It is known that contractive operators with parameter $\gamma \in (0,1)$ satisfy Assumption \ref{assumption: clarke-time-derivative} with parameter $1-\gamma$. Hence, one can apply Theorem \ref{finite-time-corollary-1} to obtain finite-time guarantees for contractive operators.

Additionally, we can easily extend the finite-time bounds to obtain almost sure convergence results.
\begin{corollary}
    \label{almost-sure-convergence-1}
    Let $\{x_k\}_{k \geq 0}$ be the sequence of iterates generated by the update rule \eqref{eqn: control-framework-discrete}, then when the stepsize sequence $\{\alpha_k\}_{k \geq 0}$ satisfies $\sum_{k = 1}^{\infty}\alpha_k = +\infty$ and $\sum_{k = 1}^{\infty}\alpha_k^2 < +\infty$, we have $x_k$ converges to $x^*$ almost surely.
\end{corollary}
Detailed proof of Corollary \ref{almost-sure-convergence-1} is in Appendix \ref{sssec:almost-sure-convergence-1-proof}. Since we only have restrictions on the step sizes, this corollary is also applicable in the smooth, exponentially stable case. The proof of this corollary is deferred to Subsection \ref{sssec:almost-sure-convergence-1-proof} and its proof outline is deferred to Subsection \ref{ssec: proof-outline-almost-sure}.

\subsection{Stochastic control of smooth sub-exponentially stable systems}
\label{ssec: smooth-sub-exponential}
Going beyond exponential stability, we consider a generalization of the smooth exponentially stable setting in Subsection \ref{ssec: smooth-exponential} where we assume a weaker stability condition as follows:
\begin{assumption}
\label{assumption: asymptotic-stability}
(Sub-exponential stability assumption of the Lyapunov function) Assume that the gradient of $V$ exists everywhere, we have the following holds:
 \begin{align}
    \label{eqn: asymptotic-stability-rewrite}
    \langle \nabla V(x), F(x) \rangle \leq -\gamma V(x)^c \,\, \forall x \in \R^d
\end{align}
where $c \geq 1, \gamma > 0$.
\end{assumption}
For $c = 1$, we have that this is the Assumption \ref{assumption: smooth-time-derivative}. In control theory, the Assumption \ref{assumption: asymptotic-stability} implies the asymptotic stability of the system \cite{khalil-book}, where the rate of the ODE is obtained using comparison functions, and with $c > 1$ we would obtain a weaker convergence (which we will call sub-exponential convergence) rather than an exponential convergence. The subgeometric drift condition in the form of \eqref{eqn: asymptotic-stability-rewrite} appears in many applications, one of which is subgeometric Markov chain mixing \cite{polynomial-convergence-markov-chain, Butkovsky_2014_subgeometric_wasserstein, durmus2015subgeometricratesconvergencewasserstein, qu-glynn2023wasserstein-contractive-drift}

In addition, we assume that the Lyapunov function has the gradient Lipschitz property (Assumption \ref{assumption: gradient-lipschitz}). In addition, we assume that the iterates of $\{x_k\}_{k \geq 0}$ is bounded:
\begin{assumption}
\label{assumption: bounded-iterates}
(Bounded iterates) There exists a positive constant $r$ such that $\norm{x_k} \leq r, \forall k \in \Z^+$.
\end{assumption}

The bounded iterate assumption is a common assumption in Stochastic Approximation literature \cite{Borkar2008StochasticAA}, \cite{bhandari2018finite}. Note that we don't require this assumption to obtain the convergence guarantees, but rather it is required here in order for Assumptions \ref{assumption: F-lipschitz}, \ref{assumption: asymptotic-stability} to hold concurrently. Now, to ensure Assumption \ref{assumption: bounded-iterates} holds for our updates, one can use some projection method to ensure that the iterates are in some ball radius $r$. In particular, we apply the following projection routine:
\begin{enumerate}
    \item Choose $r > 8\br{\frac{C_2}{C_1}}^{\frac{1}{\alpha}}\norm{x^*}$, a sufficiently small step size $\alpha_k\, \forall k \in \Z^+$ and run the update step \eqref{eqn: control-framework-discrete}
    \item If $\norm{x_k+\alpha_k (F(x_k)+w_k)} > r$, project $x_k$ to the ball radius $\frac{1}{8}\br{\frac{C_1}{C_2}}^{\frac{1}{\alpha}}r$.
\end{enumerate}
Essentially, we perform the projection if there is a positive probability that the iterate $x_k$ might violate the condition $\norm{x_k} \leq r$. As we will show below, this only happens when $\norm{x_k}$ is already sufficiently large. This projection routine ensures Assumption \ref{assumption: bounded-iterates} and our analysis holds for all iterations $k$. For a more detailed analysis, we refer the readers to Appendix \ref{ssec:projection-justification}. From these Assumptions, we have:
\begin{theorem}
\label{finite-time-corollary-2}
Under Assumptions \ref{assumption: noise}, \ref{assumption: F-lipschitz}, \ref{assumption: gradient-lipschitz}, \ref{assumption: quadratic-growth}, \ref{assumption: asymptotic-stability}, \ref{assumption: bounded-iterates} and the step size $\alpha_k = \frac{\alpha}{(k+K)^\xi}$ where for $\xi = 1: K = \max\left\lbrace 1, \frac{\alpha (4C^2 + 8B)}{\gamma} \right\rbrace$ and for $\xi \in (0,1)$, $K = \max \left\lbrace 1, \left( \frac{\alpha (4C^2 + 8B)}{\gamma} \right)^{\frac{1}{\xi}}, \left( \frac{2 \xi}{\alpha \gamma} \right)^{\frac{1}{1-\xi}}\right\rbrace$  and $0 \leq \xi \leq \frac{c}{2c-1}$, we have the finite-time bound of the SA algorithm is:
\begin{align*}
    &\textbf{For $\xi = \frac{c}{2c-1}$}: \\
    &\mathbb{E}\sqbr{\norm{x_k-x^*}^2} \leq \frac{C_2}{C_1} \norm{x_0-x^*}^2 \left( \frac{K}{k+K} \right)^{\phi} +
    \begin{cases}
        O\br{\frac{1}{k^{\phi}}} &\text{ if } \phi \in \left(0, \frac{1}{2c-1} \right) \\
        O\br{\frac{\log k}{k^{\frac{1}{2c-1}}}} &\text{ if } \phi = \frac{1}{2c-1} \\
        O\br{\frac{1}{k^{\frac{1}{2c-1}}}} &\text{ if } \phi \in \left(\frac{1}{2c-1}, \infty\right)
    \end{cases} \qquad\qquad\qquad\qquad \\
    &\textbf{For $\xi \in \left(0,\frac{c}{2c-1}\right)$}: \\
    &\mathbb{E}\sqbr{\norm{x_k-x^*}^2} \leq \frac{C_2}{C_1} \norm{x_0-x^*}^2 \exp\left[-\frac{\phi ((k+K)^{1-\frac{(2c-1)\xi}{c}}-K^{1-\frac{(2c-1)\xi}{c}})}{1-\frac{(2c-1)\xi}{c}} \right] +\frac{2\alpha^2}{C_1 \phi (k+K)^{\frac{\xi}{c}}} \qquad\qquad\qquad\qquad \\
    &\textbf{For $\xi = 0$}: \\
    &\mathbb{E}\sqbr{\norm{x_k-x^*}^2} \leq \frac{C_2}{C_1} \norm{x_0-x^*}^2 \left(1 - \phi\right)^k + \frac{L \alpha^2(cA + 2B\norm{x^*}^2)}{\mu C_1 \phi}
\end{align*}
where $\phi = \alpha^{2-1/c} A^{1-1/c} \gamma^{1/c}, \tau = \br{\frac{1}{(2c-1)A^{1-1/c}\gamma^{1/c}}}^{\frac{c}{2c-1}}$.
\end{theorem}
This result suggests that we will have weaker convergence rates for larger $c$ values. Furthermore, notice that our analysis matches with \ref{finite-time-corollary-1} for $c = 1$. We leave detailed proof for this result in Appendix \ref{sssec:finite-time-corollary-2-proof}. In addition, we also obtain the almost sure convergence below for which we leave the proof in Appendix \ref{sssec:almost-sure-convergence-2-proof} and its proof outline to Subsection \ref{ssec: proof-outline-almost-sure}.
\begin{corollary}
    \label{almost-sure-convergence-2}
    Let $\{x_k\}_{k \geq 0}$ be the sequence of iterates generated by the update rule \eqref{eqn: control-framework-discrete}, then when the stepsize sequence $\{\alpha_k\}_{k \geq 0}$ satisfies $\sum_{k = 1}^{\infty}\alpha_k^{2-\frac{1}{c}} = +\infty$ and $\sum_{k = 1}^{\infty}\alpha_k^2 < +\infty$, we have $x_k$ converges to $x^*$
    almost surely.
\end{corollary}

\subsection{Stochastic control of non-smooth sub-exponentially stable systems}
\label{ssec: non-smooth-subexponential}
In this section, we will complement the results of the previous section by considering general non-smooth systems that satisfy sub-exponential stability. Since the smoothness is absent, we define an analogous assumption to Assumption \ref{assumption: asymptotic-stability} using the notion of the Clarke generalized gradient defined in Subsection \ref{ssec: non-smooth exponential} as follows:
\begin{assumption}
\label{assumption: clarke-asymptotic-stability}
(Sub-exponential stability assumption of the Lyapunov function with the Clarke generalized gradient) Assume that the gradient of $V$ exists everywhere, we have that:
 \begin{align}
    \langle g_x, F(x) \rangle \leq -\gamma V(x)^c \,\, \forall x \in \mathbb{R}^d, g_x \in \partial V(x) \forall x \in \R^d
\end{align}
for some constant $\gamma > 0$ and $c \geq 1$.
\end{assumption}
Now, consider $R = V^{\frac{2}{a}}$ and let $M_\mu = R \square \frac{\norm{x}^2}{2\mu}$, we obtain the following bound on the time-derivative of the Moreau envelope:
\begin{lemma}
    \label{lemma: moreau-negative-drift-2}
    Under Assumptions \ref{assumption: noise}, \ref{assumption: polynomial-growth}, \ref{assumption: generalized-gradient-bound} and \ref{assumption: clarke-asymptotic-stability} and let $R = V^{\frac{2}{a}}$ and $M_{\mu} = R \square \frac{\norm{x}^2}{2\mu}$, we have:
    \begin{align}
        \langle \nabla M_{\mu}(x), F(x) \rangle \leq -\frac{2\gamma\left(\frac{C_1}{C_2}\right)^{c-1+\frac{2}{a}}}{a\left(1 + \frac{2\mu G_2 C_{1,a}^{\frac{2}{a}-1}}{a}\right)^2} R(x)^{\frac{a(c-1)}{2}+1} + \frac{2 \mu C G_2 C_{2,a}^{\frac{2}{a}-1}}{a} R(x) \forall x \in \mathbb{R}^d
    \end{align}
\end{lemma}
Notice that we don't have a negative semi-definite RHS as in the non-smooth exponentially stable case. Thus, naively applying the Moreau envelope here will not give convergence even for diminishing step sizes. Indeed, with constant $\mu$, there is an additional positive term $\frac{2\mu CG}{a}R(x)$ which cannot be canceled out by the other non-negative term for sufficiently small $x$, meaning that we would not be able to obtain convergence for arbitrary accuracy. \\
\\
Thus, this motivates us to instead select $M_{\mu_k}(x_k)$ as our Lyapunov function, which has a time-varying smooth parameter $\mu_k$ so that we can drive $\mu \rightarrow 0$ to obtain convergence. However, it must be noted that a careful choice of $\mu_k$ is required in order to achieve competitive rates. Choosing a $\mu_k$ that converges to $0$ too slowly would harm the rate at which $E_k$ converges while if $\mu_k$ converges too fast then the noise term would grow too quickly and it would harm the convergence rate regardless. With a specific choice of the parameter pairs $(\alpha_k, \mu_k)$, we obtain the finite-time bounds for non-smooth sub-exponentially stable systems as follows:
\begin{theorem}
    \label{finite-time-corollary-3}
    Under Assumptions \ref{assumption: noise},  \ref{assumption: F-lipschitz}, \ref{assumption: polynomial-growth}, \ref{assumption: generalized-gradient-bound}, \ref{assumption: bounded-iterates} and \ref{assumption: clarke-asymptotic-stability} and with the step size $\alpha_k = \frac{\alpha}{(k+K)^{\xi}}$ and the choice of parameter $\mu_k = \frac{\mu}{(k+K)^{0.5\xi}}$ where $d_{a,c} = a(c-1)/2+1, \xi \leq \frac{2d}{3}, \Phi = 0.5 d_{a,c} \alpha^{2-\frac{1}{d_{a,c}}} \nu_4^{\frac{1}{d_{a,c}}} \left(\frac{\br{\mu_0^2 G_M^2 + 1} N_C}{2\mu}\right)^{1-\frac{1}{d_{a,c}}}, \omega_* = \frac{d_{a,c} \alpha^2 \br{\mu_0^2 G_M^2 + 1} N_C}{2\mu}, d = \frac{3d_{a,c}}{3d_{a,c}-1}, c \geq 1$, we have:
    \begin{flalign*}
    &\textbf{For $\xi = \frac{2d}{3}$}: \\
    &E_k \leq E_0\br{\frac{C_{2,a}^{\frac{2}{a}}}{C_{1,a}^{\frac{2}{a}}} + 2\mu_0 C_{2,a}^{\frac{2}{a}}} \left(\frac{K}{k+K}\right)^{\Phi} + \begin{cases}
        O\br{\frac{1}{k^{\Phi}}} &\text{ if } \Phi \in (0,d-1)\\
        O\br{\frac{\log k}{k^{d-1}}} &\text{ if } \Phi = d-1 \\
        O\br{\frac{1}{k^{d-1}}} &\text{ if } \Phi > d-1
    \end{cases}\\
    &\textbf{For $\xi \in \left(0, \frac{2d}{3}\right)$}: \\
    &E_k \leq E_0\br{\frac{C_{2,a}^{\frac{2}{a}}}{C_{1,a}^{\frac{2}{a}}} + 2\mu_0 C_{2,a}^{\frac{2}{a}}} \exp{\left[ -\frac{\Phi((k+K)^{1-\frac{(3d_{a,c}-1)\xi}{2d_{a,c}}}-K^{1-\frac{(3d_{a,c}-1)\xi}{2d_{a,c}}}))}{1-\frac{(3d_{a,c}-1)\xi}{2d_{a,c}}} \right]} + \frac{2\alpha^{1.5} \omega_* \br{\frac{1}{C_{1,a}^{\frac{2}{a}}} + 2\mu_0}}{\Phi (k+K)^{\frac{\xi}{2d_{a,c}}}} \qquad\qquad\qquad \\
    &\textbf{For $\xi = 0$}: \\
    &E_k \leq E_0\br{\frac{C_{2,a}^{\frac{2}{a}}}{C_{1,a}^{\frac{2}{a}}} + 2\mu_0 C_{2,a}^{\frac{2}{a}}}(1-\Phi)^k + 
    \frac{\omega_*\br{\frac{1}{C_{1,a}^{\frac{2}{a}}} + 2\mu_0}}{\Phi}
    \end{flalign*}
    Here, we denote $E_k = \mathbb{E}\sqbr{\norm{x_k-x^*}^2} \forall k \in \Z^+$ and $E_0 = \norm{x_0-x^*}^2$.
\end{theorem}
We defer the detailed proof of this result and its prerequisite lemmas to Appendix \ref{sssec:finite-time-corollary-3-proof}. Similarly to other cases, we can 
obtain the almost sure convergence result via the Supermartingale Convergence Theorem. We defer the proof of the following corollary to Appendix \ref{sssec:almost-sure-convergence-3-proof} and its proof outline to Subsection \ref{ssec: proof-outline-almost-sure}:
\begin{corollary}
    \label{almost-sure-convergence-3}
    Let $\{x_k\}_{k \geq 0}$ be the sequence of iterates generated by the update rule \eqref{eqn: control-framework-discrete}, then when the stepsize sequence $\{\alpha_k\}_{k \geq 0}, \{\mu_k\}_{k \geq 0}$ satisfies $\alpha_k = \mu_k^2 = \frac{\alpha}{(k+K)^{\xi}}$ where $\xi \in \left(\frac{2}{3}, \frac{2d}{3} \right], c \geq 1$, we have $x_k$ converges to $x^*$
    almost surely. 
\end{corollary}


\section{Proof outlines}
\label{sec:proof-outline}
We outline the proofs of the main results in this section. All the main proofs in the paper follow a general three-step procedure. We elucidate these steps below.\\
\\
\textbf{Step 1 (Constructing a smooth Lyapunov function)}: First, we engineer a smooth Lyapunov function such that a negative drift condition (corresponds to the stability assumptions described in each setting) is satisfied. If the existing Lyapunov function $V$ already has the gradient Lipschitz property then we proceed to the next step. Otherwise, we construct a smooth approximation of the original Lyapunov function by using infimal convolution to obtain a smooth function $M_\mu = V \square \frac{\norm{\cdot}^2}{2\mu}$. Then, we show that our smooth Lyapunov function also has a negative drift property similar to the original Lyapunov function. For instance, we have to show that an equivalence of the Assumption \ref{assumption: smooth-time-derivative} for the non-smooth case. This negative drift condition corresponds to the time-derivative condition of the system in the ODE representation. \\
\\
To show the existence of the negative drift for the Moreau envelope, we observe that as $\mu \rightarrow 0$ then $\operatorname{prox}_{\mu R}(x) \rightarrow x, R(x)-M_\mu(x) \rightarrow 0$ and thus $\nabla M_\mu(x) \in \partial R(\operatorname{prox}_{\mu R}(x)) \rightarrow \partial R(x)$, which means that we will have an approximation of the time derivative of $R$ at $x$ for sufficiently small $\mu$. This comes at the expense of the smoothness constant of the Moreau envelope, that is $1/\mu$.
\\
\\
\textbf{Step 2 (Obtaining the one-iterate bound)}: Now, consider the discrete-time system, this is where the smoothness of the Lyapunov function became desirable. From the smoothness of the Lyapunov function and the negative drift, we apply the smooth inequality and the negative drift condition to obtain a one-iterate bound in the form of:
    \begin{align*}
        \E\left[M_{\mu_{k+1}}(x_{k+1}) | \cF_k \right] \leq \br{1- O\br{ \alpha_k^{c'}}}M_{\mu_k}(x_k) + O\br{\alpha_k^2}
    \end{align*}
    for some constant $c' > 0$. In Subsection \ref{ssec: proof-outline-step-2}, we give a detailed outline for obtaining the one-iterate bound for two different cases: the time-invariant Lyapunov function case and the time-varying Lyapunov function case.\\
\\
\textbf{Step 3 (Obtaining the finite-time bounds)}: Lastly, from the one-iterate bound above we get a corresponding finite-time bound. This is done by choosing a proper step size $\alpha_k$ and the parameter $\mu_k$ (for the non-smooth cases). Note that if our chosen $\alpha_k$ is too large then the algorithm might diverge while if $\alpha_k$ is too small then the algorithm would converge too slowly. In addition, for the non-smooth settings, choosing the right smoothness parameter is also vital for achieving a competitive convergence rate since if $\mu_k$ is too small then this will cause the smoothness parameter of the Lyapunov function to be too large and hinder the rate of convergence. On the other hand, if $\mu_k$ is not sufficiently small then the Moreau envelope would not be able to approximate the original Lyapunov function well enough. Thus, in order to complete the proofs, a detailed analysis with a precise choice of parameters is required.
\\
\\
These are the overarching steps to achieve finite-time bounds for the stochastic approximation algorithm. However, in some settings, our proof requires more involved techniques such as using a time-varying Lyapunov function with a time-varying smoothness parameter. For more details, please refer to the full proofs in the Appendix. With this framework, we will go on to outline the proofs of specific lemmas and theorems below:

\subsection{Constructing a smooth Lyapunov function}
\label{ssec: proof-outline-step-1}
Now, we will present a proof outline. In cases of a non-smooth Lyapunov function (subsections \ref{ssec: non-smooth exponential}, \ref{ssec: non-smooth-subexponential}), we need to take the Moreau envelope of the Lyapunov function so that we can upper-bound the Lyapunov function using the smooth inequality. While there are many ways of smoothing \cite{beck-smoothing}, the Moreau envelope is most desirable here since it can preserve the negative drift condition with a proper choice of parameters. We rescale the Lyapunov function to some other Lyapunov function that has a quadratic growth $R = V^{\frac{2}{a}}$ and then take the Moreau envelope as $M_\mu = R \square \frac{\norm{\cdot}^2}{2\mu}$. Note that when the Lyapunov function is already smooth, we basically choose $M_0$ as the Lyapunov function.

To show the existence of the negative drift for the Moreau envelope, we observe that as $\mu \rightarrow 0$ then $\operatorname{prox}_{\mu R}(x) \rightarrow x, R(x)-M_\mu(x) \rightarrow 0$ and thus $\nabla M_\mu(x) \in \partial R(\operatorname{prox}_{\mu R}(x)) \rightarrow \partial R(x)$, which means that we will have an approximation of the time derivative of $R$ at $x$ for sufficiently small $\mu$. This comes at the expense of the smoothness constant of the Moreau envelope, that is $1/\mu$.

Denote $u = \operatorname{prox}_{\mu R}(x)$, notice that $\nabla M_\mu(x) = g_u \in \partial R(\operatorname{prox}_{\mu R}(x)) = \partial R(u)$ and $\nabla M_\mu(x) = \frac{x-u}{\mu}$ by the Envelope Theorem. \textcolor{black}{Since $R, V$ are polynomially bounded functions from Assumption \ref{assumption: polynomial-growth}, we have $R, V$ are Lipschitz for some ball radius $\delta$. From Rademacher's Theorem \cite{Evans2015-measure-theory-book}, we have $R, V$ are differentiable almost everywhere in this ball. Let $g_x^V \in \partial V(x)$, we have $\nabla R(p) = \frac{2}{a} V(p)^{\frac{2}{a}-1} \nabla V(p)$ for all $p \in \Omega_R$. Thus, for $g_u \in \partial R(u)$, we can represent $g_u =  \frac{x-u}{\mu} = \sum_i w_i v_i$ where $v_i = \lim_{u_i \rightarrow u} \nabla R(u_i) = \frac{2}{a} V(u)^{\frac{2}{a}-1} \lim_{u_i \rightarrow u} \nabla R(u_i)$ are the limit points in $\partial R(u)$ and $\sum_i w_i = 1$ where $w_i$ are non-negative weights. From here, we have
\begin{align}
    \nonumber
    \langle \nabla M_\mu(x), F(x) \rangle &= \left\langle \sum_i w_i v_i, F(u) \right\rangle + \left\langle \frac{x-u}{\mu}, F(x)-F(u) \right\rangle \\
    \nonumber
    &= \left\langle \sum_i w_i \lim_{u_i \rightarrow u} \nabla R(u_i), F(u) \right\rangle + \left\langle \frac{x-u}{\mu}, F(x)-F(u) \right\rangle \\
    \nonumber
    &= \underbrace{\frac{2}{a}V(u)^{\frac{2}{a}-1} \left\langle \sum_i w_i \lim_{u_i \rightarrow u} \nabla V(u_i), F(u) \right\rangle}_{\text{chain rule}} + \left\langle \frac{x-u}{\mu}, F(x)-F(u) \right\rangle \\
    \nonumber
    &\overset{(a)}{=} \frac{2}{a}V(u)^{\frac{2}{a}-1} \langle g_u^V, F(u) \rangle + \left\langle \frac{x-u}{\mu}, F(x)-F(u) \right\rangle \\
    \nonumber
    &\overset{(b)}{\leq} -\frac{2\gamma}{a} V(u)^{\frac{2}{a}-1} V(u)^c + \left\langle \frac{x-u}{\mu}, F(x)-F(u) \right\rangle \\
    \nonumber
    &= -\frac{2\gamma}{a} R(u)^{\frac{a(c-1)}{2}+1} + \left\langle \frac{x-u}{\mu}, F(x)-F(u) \right\rangle \text{ since } R = V^{\frac{2}{a}} \\
    \label{eqn: chain-rule-fully-written}
    &\leq -\frac{2\gamma}{a} R(u)^{\frac{a(c-1)}{2}+1} + \frac{C\norm{x-u}^2}{\mu} \text{ from Assumption \ref{assumption: F-lipschitz}}.
\end{align}
Here, we have (a) follows from $g_u^V \in \partial V(u)$ and the definition of Clarke generalized gradient, (b) follows from Assumption \ref{assumption: clarke-time-derivative}.} From here, we will lower-bound $V(u)$ and upper-bound the quantity $\norm{\frac{x-u}{\mu}}$ so that the RHS will be upper-bounded by some terms dependent on $M_\mu(x)$. The resulting bound will be:
\begin{align*}
    \langle \nabla M_\mu(x), F(x) \rangle \leq -D_1 M_\mu(x)^{\frac{a(c-1)}{2}+1} + \mu D_2 M_\mu(x) \forall x \in \R^d
\end{align*}
for $\mu$ is some bounded positive real, $D_1,D_2$ are some positive constants and $c \geq 1$. Note that if $c = 1$ then choosing a sufficiently small $\mu$ would be adequate to obtain a negative drift. However, if $c > 1$ then there exists $r > 0$ such that the RHS is positive for any $\norm{x} < r$. Thus, this will motivate us to use a time-varying $\mu$ parameter in such a scenario. We reserve the proof for this lemma in the Appendix \ref{sssec:finite-time-corollary-1-proof} and \ref{sssec:finite-time-corollary-3-proof}.



\subsection{Obtaining the one-iterate bound}
\label{ssec: proof-outline-step-2}
Our next step is to obtain a one-iterate bound using the negative drift condition. This is the essential step to obtain the finite-time bound.

\subsubsection{Time-invariant Lyapunov function}
\label{sssec: proof-outline-time-invariant}
To establish the finite-time bounds for our main results, note that from smoothness and the negative drift assumption of our Lyapunov function, we have the one-iterate bound for some constants $\nu_1, \nu_2, \nu_3, A^*, B^*$:
\begin{align*}
    \E\left[M_\mu(x_{k+1}) | \cF_k \right] \leq \br{1 + \alpha_k \mu \nu_1 + \frac{\alpha_k^2 \nu_2}{\mu}}M_\mu(x_k) \underbrace{- \alpha_k \nu_3 M_\mu(x_k)^c}_{\text{The negative drift term}}
    + \underbrace{\frac{\alpha_k^2}{\mu}\br{A^*+B^*\norm{x^*}^2}}_{\text{The variance term}}.
\end{align*}
Our goal is to somehow reduce this to a simpler form:
\begin{align*}
    \E\left[M_\mu(x_{k+1}) | \cF_k \right] \leq \br{1- \nu \alpha_k^{c'}}M_\mu(x_k) + \frac{\alpha_k^2}{\mu}\br{A^*+B^*\norm{x^*}^2}
\end{align*}
for some constant $\nu, c' > 0$. First, we need to absorb the negative drift term into the $M_\mu(x_k)$ term. If $c = 1$ (the exponential stable case) then we are done, otherwise, we bound them with:
\begin{align*}
    \alpha_k \nu_3 M_\mu(x_k)^c + (c-1)\frac{\alpha_k^2}{\mu} \geq c \nu \alpha_k^{c'}M_\mu(x_k)
\end{align*}
for some constant $\nu, c' > 0$, which yield:
\begin{align*}
    \E\left[M_\mu(x_{k+1}) | \cF_k \right] \leq \br{1 + \alpha_k \mu \nu_1 + \frac{\alpha_k^2 \nu_3}{\mu}}M_\mu(x_k) - \nu \alpha_k^{c'}M_\mu(x_k)
    + \frac{\alpha_k^2}{\mu}\br{cA^*+B^*\norm{x^*}^2}.
\end{align*}
In both cases, we can choose $\alpha, K$ so that this bound holds for all $k$ and for $c > 1$:
\begin{align*}
    1 + \alpha_k \mu_k \nu_1 + \frac{\alpha_k^2 \nu_2}{\mu_k} - c \nu \alpha_k^{c'} \leq 1 - \nu \alpha_k^{c'}.
\end{align*}
And thus, the one-iterate bound can be reduced to:
\begin{align*}
    \E\left[M_\mu(x_{k+1}) | \cF_k \right] \leq \br{1- \nu \alpha_k^{c'}}M_\mu(x_k) + \frac{\alpha_k^2}{\mu}\br{A^*+B^*\norm{x^*}^2}.
\end{align*}

\subsubsection{Time-varying Lyapunov function}
\label{sssec: proof-outline-time-variant}
For a time-varying Lyapunov $M_{\mu_k}(x_k)$ (with a changing smoothness parameter $\mu_k$), the situation is slightly different. Not only we have to care about the time-derivative of $M$ w.r.t to $x$, we have to quantify the change of $M$ in terms of $\mu$ as well. That being said, the latter will incur an additional $(\mu_k-\mu_{k+1})\nu_2 M_{\mu_k}(x_k)$ term, which gives:
\begin{align*}
    \E\left[M_{\mu_{k+1}}(x_{k+1}) | \cF_k \right] \leq \br{1 + \alpha_k \mu_k \nu_1 + \underbrace{(\mu_k-\mu_{k+1})\nu_2}_{\text{change of } M \text{ w.r.t }\mu} + \frac{\alpha_k^2 \nu_3}{\mu_k}}M_{\mu_k}(x_k) \\
    \underbrace{- \alpha_k \nu_4 M_{\mu_k}(x_k)^{\frac{a(c-1)}{2}+1}}_{\text{The negative drift term}}
    + \underbrace{\frac{\alpha_k^2}{\mu_k}\br{A^*+B^*\norm{x^*}^2}}_{\text{The variance term}}.
\end{align*}
Apply a similar routine to absorb the negative drift term, we obtain the bound:
\begin{align*}
    \E\left[M_{\mu_{k+1}}(x_{k+1}) | \cF_k \right] \leq \br{1 + \alpha_k \mu_k \nu_1 + (\mu_k-\mu_{k+1})\nu_2 + \frac{\alpha_k^2 \nu_3}{\mu_k}}M_{\mu_k}(x_k) \\
    -\nu \alpha_k^{c'}M_{\mu_k}(x_k)
    + \frac{\alpha_k^2}{\mu_k}\br{cA^*+B^*\norm{x^*}^2}.
\end{align*}
In the nonsmooth sub-exponential stable case, our chosen parameters are $\alpha_k = \frac{\alpha}{\br{k+K}^{\xi}}, \mu_k = \frac{\mu}{\br{k+K}^{\upsilon}}$ where $0 \leq \xi \leq 1$ and $\nu = \frac{\xi}{2}$ (in this case, we have $c' = \frac{3d_{a,c}-1}{2d_{a,c}}$) which gives $\alpha_k \approx \mu_k^2 \Rightarrow \alpha_k \mu_k \approx \frac{\alpha_k^2}{\mu_k} = O\br{k^{-\frac{3\xi}{2}}} \gg \mu_k-\mu_{k+1} = O\br{k^{-\frac{\xi}{2}-1}}$ since $\xi \leq \frac{2d}{3} \leq 1$ for $c \geq 1$. On the other hand, we have $c' = \frac{(3d_{a,c}-1)}{2d_{a,c}} \Rightarrow \alpha_k^{c'} = O\br{k^{-\frac{(3d_{a,c}-1)\xi}{2d_{a,c}}}} \ll O\br{k^{-\frac{3\xi}{2}}}$. Thus, we can choose $\alpha, K$ so that this bound holds for all $k$ and for $c > 1$:
\begin{align*}
    1 + \alpha_k \mu_k \nu_1 + (\mu_k-\mu_{k+1})\nu_2 + \frac{\alpha_k^2 \nu_3}{\mu_k} - c \nu \alpha_k^{c'} \leq 1 - \nu \alpha_k^{c'}.
\end{align*}
And thus, the one-iterate bound can be reduced to:
\begin{align*}
    \E\left[M_{\mu_{k+1}}(x_{k+1}) | \cF_k \right] \leq \br{1- \nu \alpha_k^{c'}}M_{\mu_k}(x_k) + \frac{\alpha_k^2}{\mu_k}\br{A^*+B^*\norm{x^*}^2}
\end{align*}

\subsection{Obtaining the finite-time bounds}
\label{ssec: proof-outline-step-3}
Expand this one-iterate bound and we obtain:
\begin{align*}
    \E\left[M_{\mu_{k+1}}(x_{k+1}) | \cF_k \right] \leq \underbrace{\prod_{i=1}^k \br{1-\nu \alpha_k^{c'}}}_{T_1} M_{\mu_0}(x_0) + \br{A^*+B^*\norm{x^*}^2} \underbrace{\sum_{i = 0}^{k} \frac{\alpha_i^2 \prod_{j = i+1}^{k} \left( 1 - \nu \alpha_j^{c'} \right)}{\mu_i}}_{T_2}.
\end{align*}
The $T_1$ term represents how fast the RHS vanishes with each $1-\nu \alpha_k^{c'}$ term is the diminishing factor at each iteration while the $T_2$ term represents the impact of noise in our bound.\\
\\
While it is not possible to solve this recursion for any choice of step size, it is possible to obtain finite-time bounds with some specific choice of step size. Now, we proceed to bound the terms $T_1$ and $T_2$. First, the $T_1$ term can be bounded as the following:
\begin{align*}
    T_1 = \prod_{i=1}^k \br{1-\nu \alpha_k^{c'}} \leq e^{-\nu \sum_{i=1}^k \alpha_k^{c'}} \leq 
    \begin{cases}
        \br{\frac{K}{k+K}}^{\nu\alpha} \text{ if } \xi = \frac{1}{c'}\\
        e^{-\frac{\nu}{1-\xi}\br{\br{k+K}^{1-\xi}-K^{1-\xi}}} \text{ if } \xi \in \br{0,\frac{1}{c'}}
    \end{cases}.
\end{align*}
Secondly, if $\xi = \frac{1}{c'}$, the $T_2$ term can be bounded as follows:
\begin{align*}
    T_2 \leq \frac{4\alpha^2}{\br{k+K}^{\alpha \nu}} \sum_{i=1}^k \frac{1}{\br{k+K+1}^{2-\alpha \nu}} \forall k \in \Z^+.
\end{align*}
Then, we perform some casework on the value of $\alpha \nu$ based on its relation to $2$ to obtain the finite-time bound. For $\xi \in \br{0,\frac{1}{c'}}$, we can prove via induction that:
\begin{align*}
    T_2 \leq \frac{2\alpha}{\nu} \frac{1}{\br{k+K+1}^{\xi c'}} \forall k \in \Z^+.
\end{align*}
Lastly, for completeness, we consider $\xi = 0$. In this case, $T_1$ will diminish exponentially fast while $T_2$ does not converge to $0$. This means that the algorithm will converge to a ball exponentially fast but will not converge to the optimal point.\\
\\
Hence, our proof is complete.

\subsection{Proof outline of the almost sure convergence results}
\label{ssec: proof-outline-almost-sure}
In this subsection, we will present the proof outline for Corollaries \eqref{almost-sure-convergence-1}, \eqref{almost-sure-convergence-2}, \eqref{almost-sure-convergence-3}. First, we recall the following theorem in \cite{neurodynamic}:
\begin{theorem}
(Supermartingale Convergence Theorem) Let $\{Y_t\}_{t \geq 0}, \{X_t\}_{t \geq 0}, \{Z_t\}_{t \geq 0}$ be the sequences of random variables and let $\cF_t$ be the sets of random variables such that $\cF_t \subset \cF_{t+1}$. Suppose that:
\begin{enumerate}
    \item The random variables $Y_t, X_t, Z_t$ are nonnegative and are functions of random variables in $\cF_t$.
    \item For each $t$, we have $\E[Y_{t+1} | \cF_t] \leq Y_t-X_t+Z_t$.
    \item The sequence $\{Z_t\}_{t \geq 0}$ is summable.
\end{enumerate}
Then we have $\{X_t\}_{t \geq 0}$ is summable and $\{Y_t\}_{t \geq 0}$ converges to a nonnegative random variable $Y$ with probability $1$.
\end{theorem}
Thus, to apply the Supermartingale Convergence Theorem, we will have to choose proper $Y_t, Z_t, X_t$ and then decompose the one-iterate accordingly. With this in mind, we choose $Y_t$ to be the value of the Lyapunov function at step $t$, $X_t$ to be the negative drift and $Z_t$ to be the variance term. Since we chose $\alpha_k = \frac{\alpha}{\br{k+K}^{\xi}}$, we can impose conditions on $\xi$ so that $\{Z_t\}_{t \geq 0}$ is summable and $X_t, Y_t$ can be driven to $0$. Hence, we have the value of our Lyapunov function converges to $0$ almost surely.

\section{Numerical experiments}
\label{sec:experiments}
In this section, we will demonstrate the performance of our algorithm in three settings: non-smooth exponentially stable systems (selector control), smooth sub-exponentially stable systems (Artstein's circle), and non-smooth sub-exponentially stable systems (a toy example from \cite{khalil-book}). In each experiment, we measure the slope of the linear regressor of the Lyapunov function value to empirically validate our convergence complexities. We omit the smooth exponentially stable setting since it was done in \cite{nonlinear-sa} via the Q-learning experiment. Note that the following settings cannot be studied using the analysis from prior works such as \cite{zaiwei-envelope, nonlinear-sa}. Indeed, the dissipativity assumption in \cite{nonlinear-sa} would not be applicable to the selector control example in Subsection \ref{ssec: selector-control}, and the contractive operator assumption would not be able to study subexponentially stable settings in Subsections \ref{ssec: toy-example} and \ref{ssec: nonsmooth-subexponential-experimental-setting}.

In all of our experimental settings below, the noise sequence $\{w_k\}_{k \geq 0}$ sampled as an i.i.d. Gaussian noise sequence with mean $0$ and variance $1$.



\subsection{Selector control}
\label{ssec: selector-control}
In this subsection, we consider the setting corresponding to the nonsmooth exponential case which is selector control, which corresponds to our results in Subsection \ref{ssec: non-smooth exponential}. In practical applications of control systems, oftentimes we will have discrete measurements which make our control systems non-smooth. Furthermore, external factors such as numerical precision, hardware and environmental issues can yield noisy measurements. One such example is the selector control (Example 4.4 in \cite{Johansson2003PiecewiseLC}), which is commonly utilized for control in boilers, power systems, and nuclear reactors \cite{PID-controllers-book}. Consider the system as shown in Figure \ref{fig:selector_control} with the dynamic $\dot{x} = Ax + Bu$, we can write the closed-loop dynamic as:
\begin{align*}
    \dot{x} &= F(x,u) = Ax + Bu \\
    &= Ax + B\min \{k_1^Tx, k_2^Tx\} \\
    &= (A + Bk_1^T)x + B\min \{0, (k_2-k_1)^Tx \} \\
    &= (A + Bk_1^T)x + B\min \{0, k^Tx\}
\end{align*}
where $k = k_2-k_1$. From the closed-loop dynamic, we have a corresponding nonlinear stochastic discrete-time update with the noise sequence $w_k$ and the step size sequence $\alpha_k$:
\begin{align*}
    x_{k+1} = x_k + \alpha_k \br{Ax_k + B\min \{k_1^Tx_k, k_2^Tx_k\} + w_k}.
\end{align*}
Let $A_1 = A + Bk_1^T, A_2 = A + Bk_2^T$ and consider the system:
\begin{align*}
A_1 = 
\begin{bmatrix}
    -5 & -4 \\
    -1 & -2
\end{bmatrix},
B = \begin{bmatrix}
    -3 \\
    -21
\end{bmatrix},
k = \begin{bmatrix}
    1 \\
    0
\end{bmatrix}.
\end{align*}
The system is piecewise linear hence it is straightforward to see that this system satisfies Assumption \ref{assumption: F-lipschitz}. Note that there is no global quadratic Lyapunov function for this system, the dissipativity assumption in \cite{nonlinear-sa} (which uses $\norm{x-x^*}^2$ as the Lyapunov function) would not be applicable. Hence we have to choose a piecewise quadratic Lyapunov, which is non-smooth, in order to take into account of the hybrid nature of the system:
\[
V =
\begin{cases}
    x^TPx &\text{ if } k^Tx \leq 0 \\
    x^T(P +  \eta kk^T)x &\text{ otherwise }
\end{cases}
\]
where $\eta > 0$. Choose 
\[
P = 
\begin{bmatrix}
    1 & 0 \\
    0 & 3
\end{bmatrix}, \eta = 9
\]
one can check that $P, P + \eta kk^T$ are symmetric positive definite matrix that satisfies:
\begin{align*}
    (A + Bk_1^T)P + P(A + Bk_1^T)^T < 0 \\
    (A + Bk_2^T)(P + \eta kk^T) + (P +  \eta kk^T)(A + Bk_2^T)^T < 0
\end{align*}
which implies that there exists $\gamma', \gamma > 0$ such that $\dot{V} \leq -\gamma' \norm{x}^2 \Rightarrow \dot{V} \leq -\gamma V$. Thus, the system satisfies Assumptions \ref{assumption: clarke-time-derivative}. Furthermore, since the Lyapunov function is piecewise quadratic, the Clarke generalized gradient is linearly bounded. Hence, the Assumptions \ref{assumption: polynomial-growth} and \ref{assumption: generalized-gradient-bound} are satisfied for $a=2$. Thus, we can apply Theorem \ref{finite-time-corollary-1} and expect that the system will converge at the rate of $O\br{\frac{1}{k}}$. Applying the SA algorithm, we obtain results in Figure \ref{fig:selector_control_lyapunov}:
\begin{figure}[h]
\centering
\begin{subfigure}[t]{0.49\textwidth}
    \includegraphics[width = 8cm]{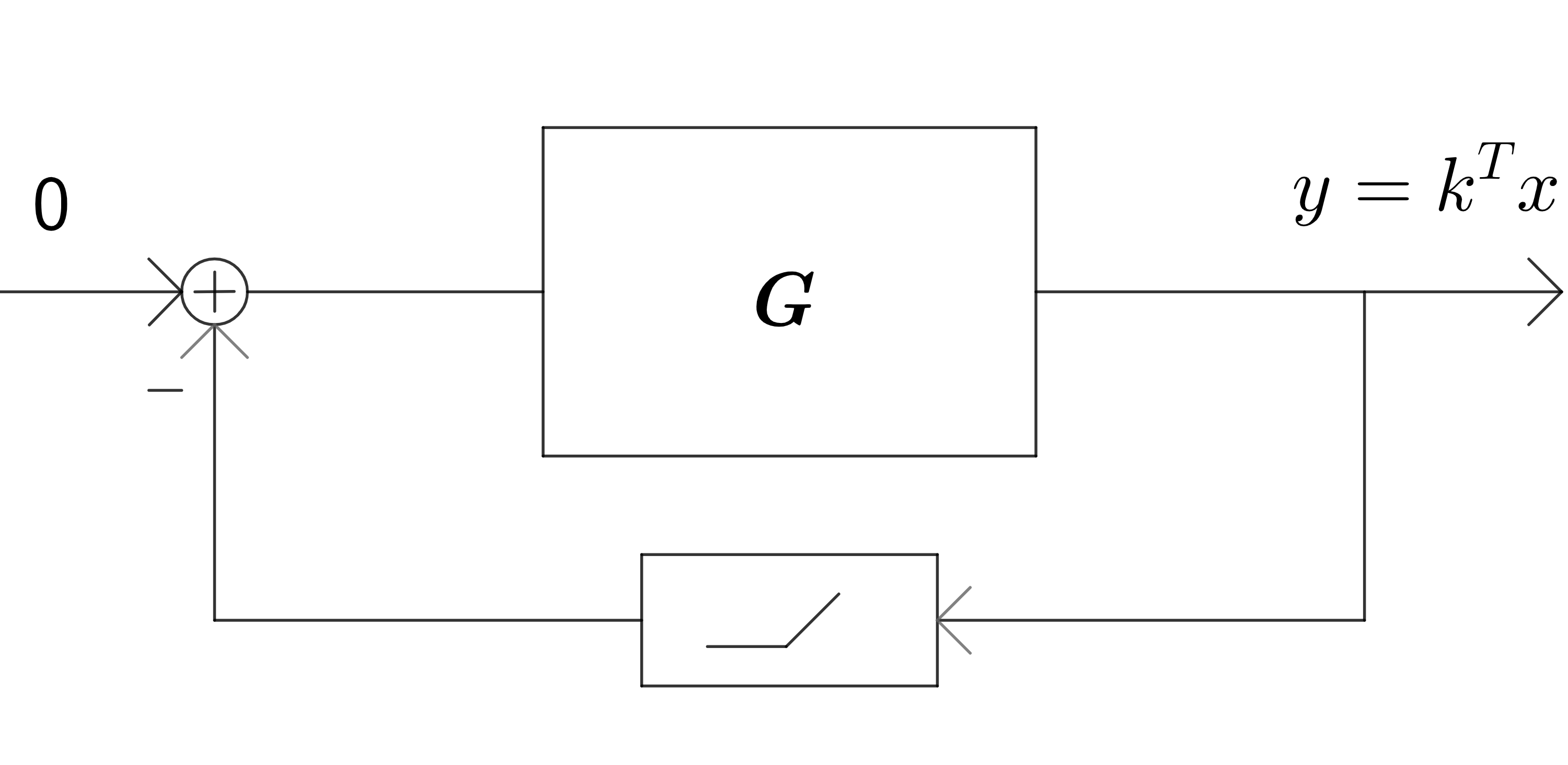}
    \caption{Selector control electrical grid illustrations in the feedback control form}
    \label{fig:selector_control}
\end{subfigure}\hfill
\begin{subfigure}[t]{0.49\textwidth}
    \includegraphics[width = 8cm]{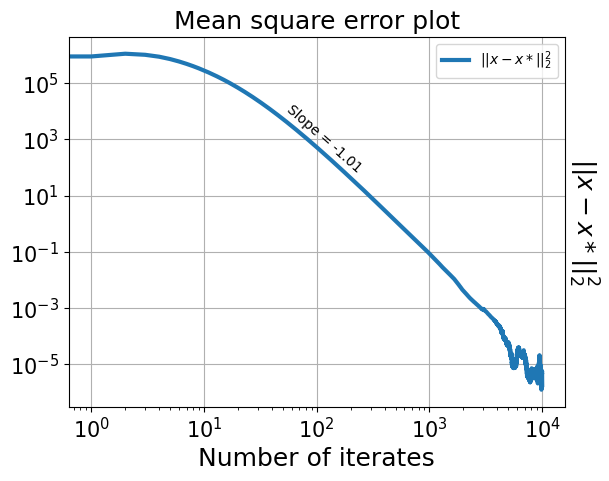}
    \caption{Mean square error plot of the selector control model when $\xi = 1$}
    \label{fig:selector_control_lyapunov}
\end{subfigure}
\caption{An example of a non-smooth exponentially stable system with the elector control example.}
\end{figure}

Note that the slope value $-1.01$ of the linear regressor of the log plot in Figure \ref{fig:selector_control_lyapunov} roughly matches the complexity $O\br{\frac{1}{k}}$ for $\xi = 1$. This empirical result seems to suggest that our complexity bound is tight.


\subsection{Nonlinear smooth sub-exponentially stable systems}
\label{ssec: toy-example}
Now, we consider the following example of a nonlinear system in \cite{khalil-book} which represents an example of a smooth sub-exponentially stable system presented in Subsection \ref{ssec: smooth-sub-exponential}. Let the dynamic of the nonlinear system with the state vector $x \in \R^2$ be:
\begin{align*}
\dot{x} =  
\begin{bmatrix}
    \dot{x}_1 \\
    \dot{x}
\end{bmatrix}
= 
\begin{bmatrix}
    -x_2-x_1^3 \\
    x_1-x_2^3
\end{bmatrix}.
\end{align*}
From the dynamic, we derive the corresponding stochastic nonlinear discrete-time dynamic as follows
\begin{align*}
\begin{bmatrix}
    x_{k+1,1} \\
    x_{k+1,2}
\end{bmatrix}
=
\begin{bmatrix}
    x_{k,1} \\
    x_{k,2}
\end{bmatrix}
+
\alpha_k
\br{
\begin{bmatrix}
    -x_{k,2}-x_{k,1}^3 \\
    x_{k,1}-x_{k,2}^3
\end{bmatrix}
+
\begin{bmatrix}
    w_{k,1} \\
    w_{k,2}
\end{bmatrix}
}.
\end{align*}
Under Assumption \ref{assumption: bounded-iterates}, this system also satisfies the Lipschitz Assumption \ref{assumption: F-lipschitz}. Choose the Lyapunov function $V(x_1,x_2) = \frac{x_1^2+x_2^2}{2}$, we have $V$ is gradient Lipschitz and the time derivative of the Lyapunov function becomes:
\begin{align*}
    \dot{V} &= x_1\dot{x}_1 + x_2\dot{x} = x_1(-x_2-x_1^3) + x_2(x_1-x_2^3) \\
    &= -(x_1^4+x_2^4) \leq -\frac{(x_1^2+x_2^2)^2}{2} = -2V^2
\end{align*}
Thus, the system satisfies Assumption \ref{assumption: quadratic-growth} and Assumption \ref{assumption: asymptotic-stability} for $c = 2$ and $\gamma = 2$. Given that these Assumptions are satisfied, we can apply Theorem \ref{finite-time-corollary-2} and expect that the system will converge at the rate of $O\br{\frac{1}{k^{\frac{\xi}{c}}}}$. We obtain the following plots in Figure \ref{fig:smooth-lyapunov}.
\begin{figure}[h]
\centering
\begin{subfigure}[t]{0.5\textwidth}
    \includegraphics[width = 8cm]{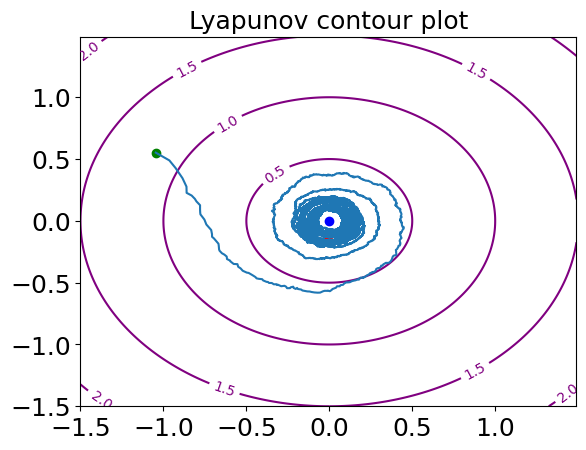}
    \caption{Trajectory of the system}
    \label{fig:smooth-contour}
\end{subfigure}\hfill
\begin{subfigure}[t]{0.5\textwidth}
    \includegraphics[width = 8cm]{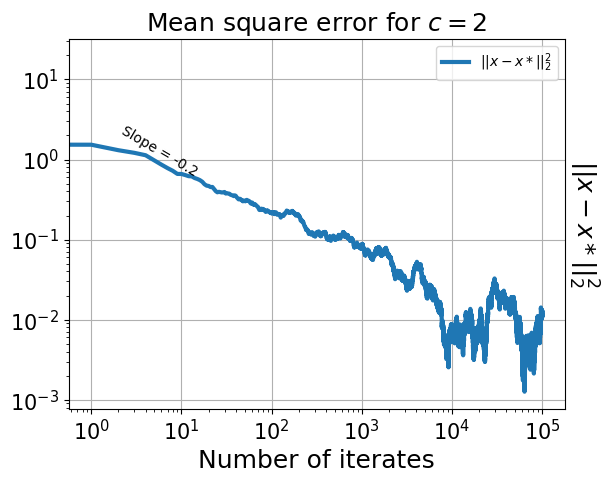}
    \caption{Mean square error plot of the problem when $\xi = 0.4$.}
    \label{fig:smooth-lyapunov}
\end{subfigure}
\caption{Smooth nonlinear system example from \cite{khalil-book}}
\end{figure}
Note that the slope value of the log plot in Figure \ref{fig:smooth-lyapunov} roughly matches the complexity $O\br{\frac{1}{k^{\frac{\xi}{c}}}}$ for $\xi = 0.4, c = 2$, which is $O\br{\frac{1}{k^{0.2}}}$. Thus, the empirical performance of our algorithm also validates the theoretical bounds.

\subsection{Nonlinear non-smooth sub-exponentially stable systems}
\label{ssec: nonsmooth-subexponential-experimental-setting}
In this subsection, we consider the Artstein's circle example which is a common example for the discontinuous Caratheodory systems. This example is notable since it cannot be stabilized by any continuous feedback \cite{nonsmooth-lyapunov-caratheodory}. The system has $x = (x_1,x_2)$ as the state and $\R^2$ as the state space. In addition, its dynamics can be described as follows:
\begin{align*}
    \begin{bmatrix}
    \dot{x}_1 \\
    \dot{x}
\end{bmatrix}
= 
\begin{bmatrix}
    (x_1^2-x_2^2)u(x_1,x_2) \\
    2x_1x_2u(x_1,x_2)
\end{bmatrix}
\end{align*}
where $u(x_1,x_2)$ is the discontinuous feedback control such that $u(x_1,x_2) = -1$ if $x_1 \geq 0$ and $u(x_1,x_2) = 1$ otherwise \cite{nonsmooth-lyapunov-caratheodory}. From here, we can derive the following corresponding stochastic nonlinear discrete-time system:
\begin{align*}
    \begin{bmatrix}
    x_{k+1,1} \\
    x_{k+1,2}
\end{bmatrix}
=
\begin{bmatrix}
    x_{k,1} \\
    x_{k,2}
\end{bmatrix}
+
\alpha_k
\br{
\begin{bmatrix}
    (x_{k,1}^2-x_{k,2}^2)u(x_{k,1},x_{k,2}) \\
    2x_{k,1}x_{k,2} u(x_{k,1},x_{k,2})
\end{bmatrix}
+
\begin{bmatrix}
    w_{k,1} \\
    w_{k,2}
\end{bmatrix}
}.
\end{align*}
While the system is nonlinear, we can choose $D$ sufficiently large such that the norm ball radius $D$ contains $x^*$ and project the iterates onto this norm ball to ensure that they are properly contained inside the norm ball. When the iterates are bounded, the Assumption \ref{assumption: F-lipschitz} is satisfied. Choose the Lyapunov function $V(x_1,x_2) = \sqrt{4x_1^2+3x_2^2}-|x_1|$, one can verify that $\dot{V} \leq -\frac{V^2}{15}$ and thus $V$ satisfies Assumption \ref{assumption: polynomial-growth} for $a = 1, C_{1,1} = \sqrt{3}-1, C_{2,1} = 2$ and Assumption \ref{assumption: clarke-asymptotic-stability} for $c = 2$ and $\gamma = \frac{1}{15}$, implying sub-exponential stability (one can refer to Appendix \ref{ssec: negative-drift-verification} for a detailed proof of the negative drift). This gives $d_{1,2} = 1.5$ and we obtain the following results in Figure \ref{fig:caratheodory}.
\begin{figure}[hbt!]
\centering
\begin{subfigure}[t]{0.5\textwidth}
    \includegraphics[width = 8cm]{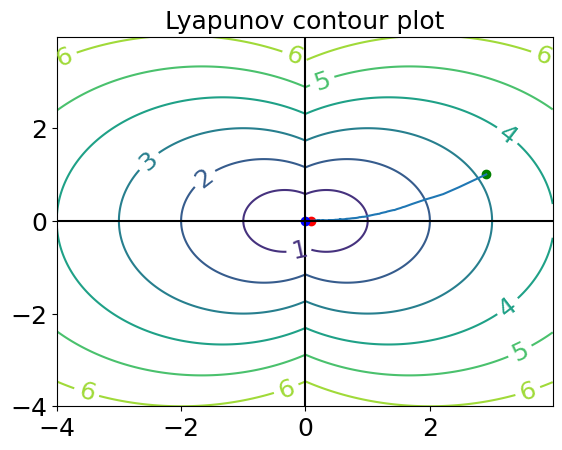}
    \caption{Trajectory of the system}
    \label{fig:caratheodory-path}
\end{subfigure}\hfill
\begin{subfigure}[t]{0.5\textwidth}
    \includegraphics[width = 8cm]{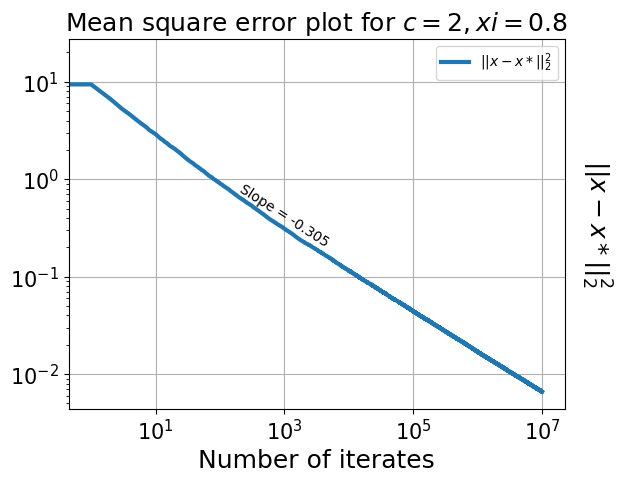}
    \caption{Lyapunov of selector control model when $\xi = 0.8$.}
    \label{fig:caratheodory-lyapunov}
\end{subfigure}
\caption{Non-smooth Caratheodory system with Artstein's circle example}
\label{fig:caratheodory}
\end{figure}
\newpage
While the plots suggest that the algorithm was able to converge, note that the slope value $-0.305$ of the linear regressor of the log plot in Figure \ref{fig:caratheodory-lyapunov} does not match the complexity $O\br{\frac{1}{k^{\frac{\xi}{2d_{a,c}}}}}$ for $\xi = 0.8, a = 1, c = 2$, which is $O\br{k^{-\frac{4}{15}}}$. Given the gap between the theoretical guarantee and the empirical performance, we conjecture that the convergence rate of $O\br{\frac{1}{k^{\frac{\xi}{2d_{a,c}}}}}$ for $\xi \in \br{0,\frac{2d}{3}}$ and $O\br{\frac{1}{k^{d-1}}}$ for $\xi = \frac{2d}{3}$ is yet to be optimal.

\section{Conclusion and Future works}
\label{sec: conclusion}

In this work, we consider stochastic approximation in discrete time, and given a Lyapunov drift of the corresponding ODE, we obtain finite-time convergence bounds for the discrete time stochastic system. The results in this paper immediately lead to the following future research directions. 
\begin{itemize}
    \item Under the exponential drift condition, we believe that the rates of convergence we obtain are tight. However, in the subexponential case, the results we have for the smooth and non-smooth cases do not match. We believe that the rates we obtain in the smooth case are tight, but those in the non-smooth case are not. This is an artifact of our proof approach. Obtaining tight rates in this non-smooth subexponential case is an immediate future direction. 
    \item One limitation of the algorithms proposed in this paper is that the choice of step size that leads to the right rate is heavily dependent on the problem's parameter. A future research direction is to develop robust methods that do not need such fine-grained step size choices. One promising direction to explore is the use of iterate averaging a.k.a. Polyak-Ruppert averaging \cite{POLYAK1963864}. However, such averaging leads to a set of two coupled iterates, which is known as two time-scale Stochastic Approximation that is notoriously hard to analyze. While recent works \cite{kaledin-wai2020finitetimeanalysislinear, shaanhaque2023tightfinitetimeboundstwotimescale} established bounds on two time-scale linear SA and we have established a Central Limit Theorem, obtaining such results in the full generality of the settings studied in this paper is a challenging future research direction.

    \item Finally, our negative drift conditions are general with many potential applications in Optimization \cite{thinh-decentralized-sa, hoang-i-pds}, Markov chain mixing \cite{polynomial-convergence-markov-chain, Butkovsky_2014_subgeometric_wasserstein, durmus2015subgeometricratesconvergencewasserstein, qu-glynn2023wasserstein-contractive-drift, hoang-erlang-c-mixing} or Generative Modeling \cite{raginsky2017nonconvexsgld, vempala-wibisono2022rapid-isoperimetry-ula, Durmus2014QuantitativeBoundsLangevin, augustchen2024langevindynamicsunifiedlyapunovperspective}. Therefore, a potentially fruitful future direction is to apply the analysis where it is appropriate, that is, whenever the system in question admits some form of a stability condition.
\end{itemize}

\section{Acknowledgements}
This work was partially supported by NSF grants EPCN-2144316 and CPS-2240982. The author H.H.N. also kindly thanks Sajad Khodadadian, Sushil Varma, and the anonymous referees for their suggestions to make the paper more clear and readable.

\bibliographystyle{IEEEtran}
\bibliography{refs}

\section{Appendix}
\subsection{Proof of useful lemmas}
\label{ssec:supplementary-lemmas}
 Recall that in the proof outline for the non-smooth systems, we need to show that the existing properties of the Lyapunov function $V$ such as Assumptions \ref{assumption: polynomial-growth}, \ref{assumption: generalized-gradient-bound}, and \ref{assumption: clarke-asymptotic-stability} are all preserved after we take the infimal convolution of $V$. Thus, in this subsection, we will provide the proofs of some supplementary lemmas of the rescaled Lyapunov function $R$ and the properties of the Moreau envelope. For any polynomially bounded Lyapunov function $V$, a rescale step to make the function quadratically bounded is necessary in order to ensure the resulting Moreau envelope is smooth and also quadratically bounded. Therefore, we will show that Assumptions \ref{assumption: polynomial-growth}, \ref{assumption: generalized-gradient-bound}, and \ref{assumption: clarke-asymptotic-stability} still more or less hold for $R = V^{\frac{2}{a}}$.

\begin{lemma}
\label{lemma: envelope-properties}
Under Assumptions \ref{assumption: polynomial-growth}, \ref{assumption: generalized-gradient-bound}, and \ref{assumption: clarke-asymptotic-stability}. Let $R = V^{\frac{2}{a}}$ and $M_\mu = R \square \frac{1}{2\mu}\norm{x}^2$, we have that:
\begin{itemize}
    \item $C_{1,a}^{\frac{2}{a}}\norm{x-x^*}^2 \leq R(x) \leq C_{2,a}^{\frac{2}{a}}\norm{x-x^*}^2$
    \item $M$ is convex and $\frac{1}{\mu}$-smooth
    \item $\exists 0 \leq a < b: (1+a)M_\mu(x) \leq R(x) \leq (1+b)M_\mu(x)$
    \item If $R$ is a norm then $M$ is also a norm
\end{itemize}
\end{lemma}
\begin{proof}
The first property follows trivially from Assumption \ref{assumption: polynomial-growth}. The second and fourth properties follow from Lemma 2.1 in \cite{zaiwei-envelope}. For the third property, we have:
\begin{align}
    M_\mu(x) = \min_{u \in \mathbb{R}^d} \left\lbrace R(u) + \frac{\norm{x-u}^2}{2\mu} \right\rbrace \leq R(x)
\end{align}
which equality holds when $u = x$. For the upper bound, note that from the first property that
\begin{align}
    \nonumber
    M_\mu(x) &= \min_{u \in \mathbb{R}^d} \left\lbrace R(u) + \frac{\norm{x-u}^2}{2\mu} \right\rbrace \\
    \nonumber
    &\geq \min_{u \in \mathbb{R}^d} \left\lbrace C_{1,a}^{\frac{2}{a}}\norm{u-x^*}^2 + \frac{(\norm{x-x^*}-\norm{u-x^*})^2}{2\mu} \right\rbrace \\
    \nonumber
    &\geq \frac{\norm{x-x^*}^2}{C_{1,a}^{-\frac{2}{a}} + 2\mu} \geq \frac{R(x)}{C_{2,a}^{\frac{2}{a}}\br{C_{1,a}^{-\frac{2}{a}} + 2\mu}}\\
    \Leftrightarrow R(x) &\leq C_{2,a}^{\frac{2}{a}}\br{C_{1,a}^{-\frac{2}{a}} + 2\mu} M_\mu(x)
\end{align}
\end{proof}
Lemma \ref{lemma: envelope-properties} allows us to quantify how well $M$ approximates $R$ and its smoothness in terms of $\mu$. While smaller $\mu$ gives a better approximation of $R$, it also scales inversely with the smoothness parameter.

\begin{lemma}
    \label{lemma: moreau-gradient-bound}
    Under Assumptions \ref{assumption: polynomial-growth}, \ref{assumption: generalized-gradient-bound} and let $R = V^{\frac{2}{a}}$ and $M_\mu = R \square \frac{1}{2\mu}\norm{x}^2$, we have $\nabla M_\mu(x) = \frac{x-u}{\mu}$ where $u = \operatorname{prox}_{\mu R}(x)$ and for $G_M = \frac{2}{a} G C_{2,a}^{\frac{2}{a}-1}$:
    \begin{align}
        \norm{\nabla M_\mu(x)} = \norm{\frac{x-u}{\mu}} \leq G_M \norm{u-x^*}
    \end{align}
\end{lemma}
\begin{proof}
    From the Envelope Theorem \cite{beck-smoothing}, we have that $\nabla M_\mu(x) = \frac{x-u}{\mu}$. From Assumption \ref{assumption: generalized-gradient-bound}, we have $\nabla M_\mu(x) \in \partial R(u)$ due to the Fundamental Theorem of Calculus. Furthermore, notice that from Assumptions \ref{assumption: polynomial-growth} and \ref{assumption: generalized-gradient-bound}, we obtain that:
    \begin{align*}
    \norm{\frac{\partial V}{\partial x}} &\leq G \norm{x-x^*}^{a-1} \\
    \Rightarrow \norm{\frac{\partial R}{\partial x}} = \frac{2}{a} V^{\frac{2}{a}-1} \norm{\frac{\partial V}{\partial x}} &\leq \frac{2}{a} G C_{2,a}^{\frac{2}{a}-1} \norm{x-x^*} \forall x \in \mathbb{R}^d
    \end{align*}
    and since $\nabla M_\mu(x) \in \partial R (\operatorname{prox}_{\mu R}(x))$ and $u = \operatorname{prox}_{\mu R}(x)$, we have:
    \begin{align*}
        \norm{\nabla M_\mu(x)} = \norm{\frac{x-u}{\mu}} \leq \frac{2}{a} G C_{2,a}^{\frac{2}{a}-1} \norm{u-x^*} = G_M \norm{u-x^*} \forall x \in \R^d
    \end{align*}
\end{proof}
Conveniently, Lemma \ref{lemma: moreau-gradient-bound} allows us to obtain a lower bound of $\norm{u-x^*}$ in terms of $\norm{x-x^*}$ and $\mu$. This comes from the intuition that $\operatorname{prox}_{\mu R}(x)$ and $x$ cannot be too far from each other and as $\mu \rightarrow 0$, $\operatorname{prox}_{\mu R}(x)$ should approach $x$, and so $\norm{u-x^*}$ cannot be too far away from $\norm{x-x^*}$.
\begin{corollary}
\label{u-lower-bound-general}
Let $u = \operatorname{prox}_{\mu R}(x)$ and $G_M = \frac{2}{a} G C_{2,a}^{\frac{2}{a}-1}$, we have:
\begin{align*}
    \norm{u-x^*} &\geq \frac{\norm{x-x^*}}{1 + \mu G_M} \forall x \in \mathbb{R}^d, \mu > 0
\end{align*}
\end{corollary}

\begin{proof}
From Lemma \ref{lemma: moreau-gradient-bound}, we have $\norm{\nabla M_\mu(x)} = \norm{\frac{x-u}{\mu}} \leq G_M \norm{u-x^*} \forall x \in \mathbb{R}^d$, thus by triangle inequality:
\begin{align*}
    \frac{\norm{x-x^*}-\norm{u-x^*}}{\mu} \leq \norm{\frac{x-u}{\mu}} \leq G_M \norm{u-x^*} \\
    \Rightarrow \frac{\norm{x-x^*}}{1 + \mu G_M} \leq \norm{u-x^*}
\end{align*}
\end{proof}


In the next lemma, we will show that if we have a negative drift condition on $V$ and let $R = V^{\frac{2}{a}}$, we also have the negative drift condition holds for $R$ as well by chain rule.

\begin{lemma}
    \label{lemma: rescaled-time-derivative-condition}
    (Negative drift of $R$) Under Assumptions \ref{assumption: polynomial-growth}, \ref{assumption: clarke-asymptotic-stability} and let $R = V^{\frac{2}{a}}, g_x^R \in \partial R(x), g_x^V \in \partial V(x)$, we have that:
    \begin{align}
        \langle g_x^R, F(x) \rangle \leq -\frac{2\gamma}{a} R(x)^{\frac{a(c-1)}{2}+1}
    \end{align}
\end{lemma}
\begin{proof}
\textcolor{black}{First, since $R, V$ are polynomially bounded functions from Assumption \ref{assumption: polynomial-growth}, we have $R, V$ are Lipschitz for some ball radius $\delta$. From Rademacher's Theorem, we have $R, V$ are differentiable almost everywhere in this ball. Observe that $\nabla R(p) = \frac{2}{a} V(p)^{\frac{2}{a}-1} \nabla V(p)$ for all $p \in \Omega_R$. And so, by applying chain rule as in \eqref{eqn: chain-rule-fully-written}, we obtain $g_x^R = \frac{2}{a} V^{\frac{2}{a} - 1} g_x^V$ for some $g_x^V \in \partial V(x)$ and from the definition of the Clarke generalized gradient \eqref{eqn: clark-gradient-definition}, we have:}
\begin{align*}
    \langle g_x^R, F(x) \rangle &= \frac{2}{a}V^{\frac{2}{a}-1} \langle g_x^V, F(x) \rangle \\
    &\leq -\frac{2}{a}V(x)^{\frac{2}{a}-1} \gamma V(x)^c \text{ from Assumption \ref{assumption: clarke-asymptotic-stability}}\\
    &= -\frac{2\gamma}{a} R(x)^{\frac{a(c-1)}{2}+1}
\end{align*}
Hence proved.
\end{proof}
Lemma \ref{lemma: rescaled-time-derivative-condition} shows that if $V$ has a negative drift then $R = V^{\frac{2}{a}}$ also has a negative drift. Lemma \ref{lemma: moreau-gradient-bound} and Lemma \ref{lemma: rescaled-time-derivative-condition} allow us to reduce the problem to when $a = 2$ where $a$ is the constant in Assumption \ref{assumption: polynomial-growth}. Note that since $R, V$ are polynomially bounded functions, we have that these functions are locally Lipschitz, which allows us to use the Rademacher Theorem to show that the set of non-differentiable points has measure $0$. Next, the following lemma provides a bound on the distance between $M_\mu$ and $M_{\mu'}$ for $0 < \mu < \mu'$, which will be crucial to analyze the convergence bounds in the non-smooth sub-exponential stable setting.

\begin{lemma}
\label{mu-derivative-bound}
(Change of smoothness parameter bound) For $0 < \mu < \mu'$ and $G_M =  \frac{2}{a} G C_{2,a}^{\frac{2}{a}-1}$, we have:
\begin{align}
    M_{\mu}(x)-M_{\mu'}(x) \leq (\mu'-\mu)G_M^2\norm{x-x^*}^2 \forall x \in \mathbb{R}^d
\end{align}
\end{lemma}
\begin{proof}
    From the Gradient Theorem for line integrals, we have $\forall x \in \mathbb{R}^d$:
    \begin{align*}
        M_{\mu}(x)-M_{\mu'}(x) &= \int_{\mu'}^{\mu} \frac{\partial M_\mu(x)}{\partial \mu} d\tau\\ &=  \int_{\mu'}^{\mu} -\norm{\nabla M_{\tau}(x)}^2 d \tau = \int_{\mu}^{\mu'} \norm{\nabla M_{\tau}(x)}^2 d \tau \\
        &\leq \int_{\mu}^{\mu'} \left(\frac{2G C_{2,a}^{\frac{2}{a}-1}}{a}\right)^2 \norm{x-x^*}^2 d \tau = (\mu'-\mu)\left(\frac{2G C_{2,a}^{\frac{2}{a}-1}}{a}\right)^2\norm{x-x^*}^2
    \end{align*}
    where $\frac{\partial M_\mu(x)}{\partial \mu} =  -\norm{\nabla M_{\mu}(x)}^2$ from Remark 3.32 \cite{Attouch1984VariationalCF} and the last inequality follows from Lemma \ref{lemma: moreau-gradient-bound}. Hence proved.
\end{proof}



\subsection{Proof of the main results}
Here, we will provide detailed proofs of the settings we included in the main results. We will first present the proofs for the non-smooth exponentially stable setting and then proceed to the smooth exponentially stable setting since the proof of the latter largely follows the arguments in the former setting once we have constructed a smooth Lyapunov function using the Moreau envelope. Later on, we will provide proof for the smooth sub-exponentially stable setting in the Subsection \ref{sssec:finite-time-corollary-2-proof} and the non-smooth sub-exponentially stable setting in the Subsection \ref{sssec:finite-time-corollary-3-proof}.

\subsubsection{Nonsmooth exponentially stable case}
\label{sssec:finite-time-corollary-1-proof}
In this subsection, we will give a detailed proof of the finite-time convergence of non-smooth exponentially stable systems. Consider $R = V^{\frac{2}{a}}$ and let $M_\mu = R \square \frac{\norm{x}^2}{2\mu}, G_M = \frac{2}{a} G C_{2,a}^{\frac{2}{a}-1}$, we will show that $M$ will also have a negative drift:
\begin{lemma}
\label{lemma: moreau-negative-drift}
There exists a constant $0 < \gamma_M < \frac{2\gamma}{a}\left(\frac{C_{1,a}}{C_{2,a}}\right)^{\frac{2}{a}}$ such that for sufficiently small $\mu > 0$:
\begin{align*}
    \langle \nabla M_\mu(x), F(x) \rangle \leq - \gamma_M M_\mu(x) \forall x \in \R^d
\end{align*}
\end{lemma} 
\begin{proof}
Recall that $u = \operatorname{prox}_{\mu R}(x)$. Since $\nabla M_\mu(x) = \frac{x-u}{\mu} = g_u \in \partial R(u)$, we have:
\begin{align*}
    \langle \nabla M_\mu(x), F(x) \rangle &= \langle g_u, F(u) \rangle + \langle g_u, F(x)-F(u) \rangle \\
    &\leq -\frac{2\gamma}{a} R(u) + C \norm{x-u}\norm{\nabla M_\mu(x)} \\
    &\leq -\frac{2\gamma C_{1,a}^{\frac{2}{a}}}{a\left(1 + \mu G_M \right)^2} \norm{x-x^*}^2 + \mu C G_M^{2} \norm{u-x^*}^2 \\
    &\leq \left[-\frac{2\gamma C_{1,a}^{\frac{2}{a}}}{a\left(1 + \mu G_M\right)^2} + \mu CG_M^2\right] \norm{x-x^*}^2
\end{align*}
where the first inequality is from Lemma \ref{lemma: rescaled-time-derivative-condition}, Cauchy-Schwarz inequality and Assumption \ref{assumption: F-lipschitz}, the second one is from Assumption \ref{assumption: polynomial-growth}, Lemma \ref{lemma: moreau-gradient-bound} and Corollary \ref{u-lower-bound-general} and the last one follows from the non-expansiveness of the proximal operator. Since $0 < \gamma_M < \frac{2\gamma}{a}\left(\frac{C_{1,a}}{C_{2,a}}\right)^{\frac{2}{a}}$, we can choose $\mu$ small enough such that $-\frac{2\gamma C_{1,a}^{\frac{2}{a}}}{a\left(1 + \mu G_M\right)^2} + \mu C G_M^{2} \leq -\gamma_M C_{2,a}^{\frac{2}{a}}$. From here, we are done since:
\begin{align*}
    \left[-\frac{2\gamma C_{1,a}^{\frac{2}{a}}}{a\left(1 + \mu G_M\right)^2} + \mu CG_M^{2}\right] \norm{x-x^*}^2 \leq -\gamma_M C_{2,a}^{\frac{2}{a}} \norm{x-x^*}^2 \leq -\gamma_M R(x)
\end{align*}
and by the definition of Moreau envelope, note that $M_\mu(x) \leq R(x) \forall x \in \R^d$. Thus $\forall x \in \mathbb{R}^d$, we have:
\begin{align*}
    \langle \nabla M_\mu(x), F(x) \rangle \leq \left[-\frac{2\gamma C_{1,a}^{\frac{2}{a}}}{a\left(1 + \mu G_M\right)^2} + \mu CG_M^{2}\right] \norm{x-x^*}^2 \leq -\gamma_M M_\mu(x)
\end{align*}
Hence proved.
\end{proof}
Now, the existence of the negative drift would allow us to obtain a contraction on $M$ with some noise via some first-order bounds (such as the smoothness inequality). From here and given that $\mu$ is chosen to be sufficiently small, we first obtain the one-iterate bound as follows:
\begin{prop}
\label{prop: one-iterate-bound-exponential}
Assume that we have chosen a sufficiently small $\mu$ such that $\exists \gamma_M \in \left(0, \frac{2\gamma}{a}\br{\frac{C_{1,a}}{C_{2,a}}}^{\frac{2}{a}} \right)$ such that $\gamma_M = -\left[-\frac{2\gamma C_{1,a}^{\frac{2}{a}}}{a\left(1 + \mu G_M\right)^2} + \mu CG_M^{2}\right]C_{2,a}^{-\frac{2}{a}}$. We have:
\begin{align*}
    \mathbb{E}[M_\mu(x_{k}) | \mathcal{F}_k] \leq \left(1 - \alpha_{k-1} \gamma_M + \frac{\alpha_{k-1}^2(C^2 + 2B)}{2\mu^2 + \frac{\mu}{C_{1,a}^{\frac{2}{a}}}} \right)M_\mu(x_{k-1}) + \frac{\alpha_{k-1}^2(A+2B \norm{x^*}^2)}{\mu}
\end{align*}
\end{prop}
\begin{proof}
Consider $R = V^{\frac{2}{a}}$, we have:
\begin{align*}
   C_{1,a}^{\frac{2}{a}}\norm{x-x^*}^2 \leq R(x) \leq C_{2,a}^{\frac{2}{a}}\norm{x-x^*}^2 \forall x \in \R^d
\end{align*}
and
\begin{align*}
    \dot{R} &= \langle r_x, F(x) \rangle \forall r_x \in \partial R(x)
\end{align*}
where $\partial R(x) = \overline{co}\{\lim \nabla R(x) | x_i \rightarrow x, x_i \not \in \Omega_V \}$ is the generalized gradient of $R$. Note that:
\begin{align*}
    \nabla R(x) = \frac{2}{a} \nabla V(x) V(x)^{\frac{2}{a}-1} \forall x \not \in \Omega_V
\end{align*}
Hence:
\begin{align*}
    \langle r_x, F(x) \rangle &= \frac{2V(x)^{\frac{2}{a}-1}}{a} \langle g_x, F(x) \rangle \forall r_x \in \partial R(x), g_x \in \partial V(x)\\
    &\leq \frac{2V(x)^{\frac{2}{a}-1}}{a} \times (-\gamma V(x)) \\
    &= -\frac{2\gamma}{a} V(x)^{\frac{2}{a}} = -\frac{2\gamma}{a} R(x) \forall x = -\gamma_R R(x) \in \mathbb{R}^d
\end{align*}
where $\gamma_R = \frac{2\gamma}{a}$ and the first inequality follows from \ref{assumption: smooth-time-derivative}. Furthermore:
\begin{align*}
    \norm{\frac{\partial V}{\partial x}} &\leq G \norm{x-x^*}^{a-1} \\
    \Rightarrow \frac{2}{a}\norm{V^{\frac{2}{a}-1}}\norm{\frac{\partial V}{\partial x}} &\leq \frac{2}{a} \times C_{2,a}^{\frac{2}{a}-1}\norm{x-x^*}^{a \times \frac{2-a}{a}}G_2 \norm{x-x^*}^{a-1} \\ 
    \Rightarrow \norm{\frac{\partial R}{\partial x}} &\leq \frac{2G C_{2,a}^{\frac{2}{a}-1}}{a}\norm{x-x^*} = G_M \norm{x-x^*}
\end{align*}
Thus, we have $R$ satisfies Assumptions \ref{assumption: clarke-time-derivative}, \ref{assumption: polynomial-growth} and \ref{assumption: generalized-gradient-bound} for $a = 2$. Let $M_\mu = R \square \frac{\norm{x}^2}{2\mu}$, from the $\frac{1}{\mu}$-smoothness of $M$, we obtain:
\begin{align}
    \nonumber
    M_\mu(x_{k}) &\leq M_\mu(x_{k-1}) + \langle \nabla M_\mu(x_{k-1}), x_{k}-x_{k-1} \rangle + \frac{1}{2\mu}\norm{x_{k}-x_{k-1}}^2
\end{align}
Taking expectations on both sides, we have:
\begin{align*}
    E_{k} = \mathbb{E}[M_\mu(x_{k}) | \mathcal{F}_{k-1}] \leq M_\mu(x_{k-1}) + \alpha_{k-1} \langle \nabla M_\mu(x_{k-1}), F(x_{k-1}) \rangle + \frac{\alpha_{k-1}^2 \mathbb{E}[\norm{F(x_{k-1})+w_{k-1}}^2 | \mathcal{F}_{k-1}]}{2\mu}
\end{align*}
Thus, from Lemma \ref{lemma: moreau-negative-drift}, we have that:
\begin{align*}
    \langle \nabla M_\mu(x), F(x) \rangle \leq -\gamma_M M_\mu(x) \forall x \in \mathbb{R}^d
\end{align*}
with $\gamma_M = \sqbr{\frac{2\gamma C_{1,a}^{\frac{2}{a}}}{a\br{1 + \mu G_M}^2} - \mu CG_M^{2}}C_{2,a}^{-\frac{2}{a}}$. Denote $E_{k} = \mathbb{E}[M_\mu(x_{k}) | \mathcal{F}_{k-1}]$, we have:
\begin{align*}
    E_{k} &\leq M_\mu(x_{k-1}) + \alpha_{k-1} \langle \nabla M_\mu(x_{k-1}), F(x_{k-1}) \rangle + \frac{\alpha_{k-1}^2 \mathbb{E}[\norm{F(x_{k-1})+w_{k-1}}^2 | \mathcal{F}_{k-1}]}{2\mu}
\end{align*}
which gives:
\begin{align*}
    E_{k} &\leq (1-\alpha_{k-1} \gamma_M)M_\mu(x_{k-1}) + \frac{\alpha_{k-1}^2}{2\mu} \mathbb{E} [\norm{F(x_{k-1})+w_{k-1}}^2 | \mathcal{F}_{k-1}] \\
    &= (1-\alpha_{k-1} \gamma_M)M_\mu(x_{k-1}) + \frac{\alpha_{k-1}^2}{2\mu}\mathbb{E}[\norm{F(x_{k-1})-F(x^*)+w_{k-1}}^2 | \mathcal{F}_{k-1}].
\end{align*}
The RHS can be further bounded as:
\begin{align*}
    &\overset{(a)}{\leq} (1-\alpha_{k-1} \gamma_M)M_\mu(x_{k-1}) + \frac{\alpha_{k-1}^2}{2\mu}\mathbb{E}[(\norm{F(x_{k-1})-F(x^*)}+\norm{w_{k-1}})^2 | \mathcal{F}_{k-1}] \\
    &\overset{(b)}{\leq} (1-\alpha_{k-1} \gamma_M)M_\mu(x_{k-1}) + \frac{\alpha_{k-1}^2}{2\mu}\mathbb{E}[(C\norm{x_{k-1}-x^*}+\norm{w_{k-1}})^2 | \mathcal{F}_{k-1}] \\
    &\overset{(c)}{\leq} (1-\alpha_{k-1} \gamma_M)M_\mu(x_{k-1}) + \frac{\alpha_{k-1}^2 (C^2\norm{x_{k-1}-x^*}^2 + \mathbb{E}[\norm{w_{k-1}}^2 | \mathcal{F}_{k-1}])}{\mu} \\
    &\overset{(d)}{\leq} (1-\alpha_{k-1} \gamma_M)M_\mu(x_{k-1}) + \frac{\alpha_k^2 (C^2\norm{x_{k-1}-x^*}^2 + A + B\norm{x_{k-1}}^2)}{\mu} \\
    &\overset{(e)}{\leq} (1-\alpha_{k-1} \gamma_M)M_\mu(x_{k-1}) + \frac{\alpha_k^2 (C^2\norm{x_{k-1}-x^*}^2 + A + 2B\norm{x_{k-1}-x^*}^2 + 2B\norm{x^*}^2)}{\mu} \\
    &= (1-\alpha_{k-1} \gamma_M)M_\mu(x_{k-1}) + \frac{\alpha_{k-1}^2(C^2 + 2B)\norm{x_{k-1}-x^*}^2}{\mu} + \frac{\alpha_{k-1}^2(A+2B \norm{x^*}^2)}{\mu} \\
    &\overset{(f)}{\leq} (1-\alpha_{k-1} \gamma_M)M_\mu(x_{k-1}) + \frac{\alpha_{k-1}^2(C^2 + 2B)M_\mu(x_{k-1})}{\mu\left(\frac{1}{C_{1,a}^{\frac{2}{a}}} + 2\mu \right)} + \frac{\alpha_{k-1}^2(A+2B \norm{x^*}^2)}{\mu} \\
    &= \left(1 - \alpha_{k-1} \gamma_M + \frac{\alpha_{k-1}^2(C^2 + 2B)}{2\mu^2 + \frac{\mu}{C_{1,a}^{\frac{2}{a}}}} \right)M_\mu(x_{k-1}) + \frac{\alpha_{k-1}^2(A+2B \norm{x^*}^2)}{\mu}
`\end{align*}
The first inequality (a) follows from triangle inequality, the second inequality (b) follows from Assumption \ref{assumption: F-lipschitz}, the third inequality (c) follows from Cauchy-Schwarz, the fourth inequality (d) follows from the Assumption \ref{assumption: noise} and the last inequality (f) follows from \ref{lemma: envelope-properties}. Hence, we are done.
\end{proof}
From the one-iterate bound Proposition \ref{prop: one-iterate-bound-exponential}, we can expand all the terms by applying the result repeatedly. Then, with appropriate choice of step size of the form $\alpha_k = \frac{\alpha}{(k+K)^{\xi}}$ where $\alpha > 0, \xi \in [0,1]$ where $K$ is chosen such that $\alpha_k \leq \frac{\gamma_M\left(2\mu^2 + \frac{\mu}{C_{1,a}^{\frac{2}{a}}}\right)}{2(C^2 + 2B)} \forall k \in \mathbb{Z}^+$. 
\begin{prop}
\label{main-theorem-1}
Under Assumptions \ref{assumption: clarke-time-derivative}, \ref{assumption: polynomial-growth} and \ref{assumption: generalized-gradient-bound} and let
\begin{align*}
    \gamma_M = -\left[-\frac{2\gamma C_{1,a}^{\frac{2}{a}}}{a\left(1 + \mu G_M\right)^2} + \mu CG_M^{2}\right]C_{2,a}^{-\frac{2}{a}}, \alpha_0 \leq \frac{2(C^2 + 2B)}{\gamma_M\left(2\mu^2 + \frac{\mu}{C_{1,a}^{\frac{2}{a}}}\right)}
\end{align*}
then we have:
\begin{align*}
    \mathbb{E}[M_\mu(x_{k}) | \mathcal{F}_k] \leq C_{2,a}^{\frac{2}{a}} \norm{x_0-x^*}^2 \prod_{j = 0}^{k-1} \left(1 - \frac{\alpha_j \gamma_M}{2}\right) + \left[ \sum_{i = 0}^{k-1} \alpha_i^2 \prod_{j = i+1}^{k-1} \left( 1 - \frac{\alpha_j \gamma_M}{2} \right) \right]\frac{A + 2B\norm{x^*}^2}{\mu}
\end{align*}
\end{prop}
\begin{proof}
By our choice of $\alpha_k$ and from the Lemma \ref{prop: one-iterate-bound-exponential}, we have:
\begin{align*}
    \mathbb{E}[M_\mu(x_{k}) | \mathcal{F}_{k-1}] &\leq \left(1 - \frac{\alpha_{k-1} \gamma_M}{2} \right) M_\mu(x_{k-1}) + \frac{\alpha_{k-1}^2(A+2B \norm{x^*}^2)}{\mu} \\
    \Rightarrow \mathbb{E}[M_\mu(x_{k}) | \mathcal{F}_{k-1}] &\leq \prod_{j = 0}^{k-1} \left(1 - \frac{\alpha_j \gamma_M}{2}\right) M_\mu(x_0) + \left[ \sum_{i = 0}^{k-1} \alpha_i^2 \prod_{j = i+1}^{k-1} \left( 1 - \frac{\alpha_j \gamma_M}{2} \right) \right]\frac{A + 2B\norm{x^*}^2}{\mu} \\
    &\leq C_{2,a}^{\frac{2}{a}} \norm{x_0-x^*}^2 \prod_{j = 0}^{k-1} \left(1 - \frac{\alpha_j \gamma_M}{2}\right) + \left[ \sum_{i = 0}^{k-1} \alpha_i^2 \prod_{j = i+1}^{k-1} \left( 1 - \frac{\alpha_j \gamma_M}{2} \right) \right]\frac{A + 2B\norm{x^*}^2}{\mu}
\end{align*}
\end{proof}
From this one-iterate bound, we expand all the terms and obtain the finite-time bounds depending on the chosen step sizes. The detailed description and proof of the finite-time bounds are presented below:

\begin{theorem}
\label{finite-time-corollary-1-full}
Under Assumptions \ref{assumption: noise}, \ref{assumption: F-lipschitz}, \ref{assumption: quadratic-growth}, \ref{assumption: clarke-time-derivative} and \ref{assumption: generalized-gradient-bound}, the rate of convergence for the SA algorithm is:\\
\textbf{For $\xi = 1$}:
\begin{align*}
\mathbb{E}[\norm{x_{k}-x^*}^2] &\leq C_{2,a}^{\frac{2}{a}} \br{\frac{1}{C_{1,a}^{\frac{2}{a}}} + 2\mu} \norm{x_0-x^*}^2 \left( \frac{K}{k+K} \right)^{\frac{\alpha \gamma_M}{2}} \\
&+ \br{\frac{1}{C_{1,a}^{\frac{2}{a}}} + 2\mu} \times
    \begin{cases}
        \frac{8\alpha^2}{2-\alpha \gamma_M} \frac{1}{(k+K)^{\frac{\alpha \gamma_M}{2}}} \frac{A + 2B\norm{x^*}^2}{\mu} &\text{ if } \alpha \in \br{0, \frac{2}{\gamma_M}}\\
        \frac{4 \alpha^2\log(k+K)}{k+K} \frac{A + 2B\norm{x^*}^2}{\mu} &\text{ if } \alpha = \frac{2}{\gamma_M} \\
        \frac{8e\alpha^2}{\alpha \gamma_M - 2} \frac{1}{k+K} \frac{A + 2B\norm{x^*}^2}{\mu} &\text{ if } \alpha \in \br{\frac{2}{\gamma_M}, \infty}
    \end{cases} \qquad\qquad\qquad\qquad\qquad\qquad\qquad\qquad
\end{align*}
\textbf{For $\xi \in (0,1)$}:
\begin{align*}
    \mathbb{E}[\norm{x_{k}-x^*}^2] &\leq C_{2,a}^{\frac{2}{a}} \br{\frac{1}{C_{1,a}^{\frac{2}{a}}} + 2\mu} \norm{x_0-x^*}^2 \exp\left[-\frac{\alpha \gamma_M}{2(1-\xi)}((k+K)^{1-\xi}-K^{1-\xi}) \right] \qquad\qquad\qquad\qquad\qquad\qquad\qquad\qquad \\
    &+ \frac{4\alpha\br{\frac{1}{C_{1,a}^{\frac{2}{a}}} + 2\mu}}{\gamma_M (k+K)^{\xi}} \frac{A + 2B\norm{x^*}^2}{\mu} 
\end{align*}
\textbf{For $\xi = 0$}:
\begin{align*}
    \mathbb{E}[\norm{x_{k}-x^*}^2] \leq C_{2,a}^{\frac{2}{a}} \br{\frac{1}{C_{1,a}^{\frac{2}{a}}} + 2\mu} \norm{x_0-x^*}^2 \left( 1-\frac{\alpha \gamma_M}{2} \right)^k + \br{\frac{1}{C_{1,a}^{\frac{2}{a}}} + 2\mu}\frac{2(A + 2B\norm{x^*}^2)\alpha}{\mu \gamma_M} \qquad\qquad\qquad\qquad
\end{align*}
where $K = \max\left\lbrace 1, \frac{\alpha (4C^2 + 8B)}{\gamma} \right\rbrace$ for $\xi = 1$ and $K = \max \left\lbrace 1, \left( \frac{\alpha (4C^2 + 8B)}{\gamma} \right)^{\frac{1}{\xi}}, \left( \frac{2 \xi}{\alpha \gamma} \right)^{\frac{1}{1-\xi}}\right\rbrace$ for $\xi \in (0,1)$.
\end{theorem}
\begin{proof}
From Proposition \ref{main-theorem-1}, we have:
\begin{align*}
    \mathbb{E}[M_\mu(x_{k}) | \mathcal{F}_{k-1}] \leq C_{2,a}^{\frac{2}{a}} \norm{x_0-x^*}^2 \prod_{j = 0}^{k-1} \left(1 - \frac{\alpha_j \gamma_M}{2}\right) + \left[ \sum_{i = 0}^{k-1} \alpha_i^2 \prod_{j = i+1}^{k-1} \left( 1 - \frac{\alpha_j \gamma_M}{2} \right) \right]\frac{A + 2B\norm{x^*}^2}{\mu}.
\end{align*}
Note that $M_\mu(x) \geq \frac{\norm{x-x^*}^2}{C_{1,a}^{-\frac{2}{a}}+2\mu} \forall x \in \R^d$. Hence:
\begin{align}
    \label{mse-bound-1}
    \nonumber
    \mathbb{E}[\norm{x_{k}-x^*}^2] &\leq C_{2,a}^{\frac{2}{a}}\br{\frac{1}{C_{1,a}^{\frac{2}{a}}}+2\mu} \norm{x_0-x^*}^2 \prod_{j = 1}^{k-1} \left(1 - \frac{\alpha_j \gamma_M}{2}\right) \\
    &+ \left[ \sum_{i = 0}^{k-1} \alpha_i^2 \prod_{j = i+1}^{k-1} \left( 1 - \frac{\alpha_j \gamma_M}{2} \right) \right]\frac{\br{A + 2B\norm{x^*}^2}\br{\frac{1}{C_{1,a}^{\frac{2}{a}}}+2\mu}}{\mu}
\end{align}
Note that:
\begin{align}
    \nonumber
    C_{2,a}^{\frac{2}{a}}\norm{x_0-x^*}^2 \prod_{i = 0}^{k-1} \left(1 - \frac{\alpha_i \gamma_M}{2} \right) &\leq C_{2,a}^{\frac{2}{a}}\norm{x_0-x^*}^2 e^{-\frac{\gamma_M}{2} \sum_{i = 0}^k \alpha_i} \\
    &= C_{2,a}^{\frac{2}{a}}\norm{x_0-x^*}^2 e^{-\frac{\alpha \gamma_M}{2} \sum_{i = 0}^{k-1} \frac{1}{(i+K)^{\xi}}}
\end{align}
and since $\int_a^{b+1} h(x) dx \leq \sum_{n = a}^b h(n) \leq \int_{a-1}^b h(x) dx$ for any non-increasing function $h$, we have for $\xi = 1$:
\begin{align}
    C_{2,a}^{\frac{2}{a}} \norm{x_0-x^*}^2 \prod_{i = 0}^{k-1} \left(1 - \frac{\alpha_i \gamma_M}{2} \right) \leq C_{2,a}^{\frac{2}{a}} \norm{x_0-x^*}^2 \left( \frac{K}{k+K} \right)^{\frac{\alpha \gamma_M}{2}}.
\end{align}
and for $\xi \in (0,1)$:
\begin{align}
    C_{2,a}^{\frac{2}{a}} \norm{x_0-x^*}^2 \prod_{i = 0}^{k-1} \left(1 - \frac{\alpha_i \gamma_M}{2} \right) \leq C_{2,a}^{\frac{2}{a}} \norm{x_0-x^*}^2 \exp\left[-\frac{\alpha \gamma_M}{2(1-\xi)}((k+K)^{1-\xi}-K^{1-\xi}) \right].
\end{align}
Next, we need to bound the second term. For $\xi = 1$, we have:
\begin{align*}
    \sum_{i = 0}^{k-1} \alpha_i^2 \prod_{j = i+1}^{k-1} \left( 1 - \frac{\alpha_j \gamma_M}{2} \right) &= \alpha^2 \sum_{i = 0}^{k-1} \frac{1}{(i+K)^2} \prod_{j = i+1}^{k-1} \left( 1 - \frac{\alpha_j \gamma_M}{2} \right) \\
    &\leq  \alpha^2 \sum_{i = 0}^{k-1} \frac{1}{(i+K)^2} \exp \left( -\frac{\alpha \gamma_M}{2} \sum_{j = i+1}^{k-1} \frac{1}{j + K} \right) \\
    &\leq \alpha^2 \sum_{i = 0}^{k-1} \frac{1}{(i+K)^2} \exp \left( -\frac{\alpha \gamma_M}{2} \int_{i+1}^{k} \frac{1}{j + K} dj \right) \\
    &\leq \alpha^2 \sum_{i = 0}^{k-1} \frac{1}{(i+K)^2} \left( \frac{i+1+K}{k+K} \right)^{\frac{\alpha \gamma_M}{2}}  \\
    &\leq \frac{4\alpha^2}{(k+K)^{\frac{\alpha \gamma_M}{2}}} \sum_{i = 0}^{k-1} \frac{1}{(i+1+K)^{2-\frac{\alpha \gamma_M}{2}}}
\end{align*}
\end{proof}
This proof follows from \cite{zaiwei-envelope}. Now, it is sufficient to do casework for $\alpha \gamma_M$. We have:
\begin{itemize}
    \item $\alpha \gamma_M \in (0,2) \Rightarrow \sum_{i = 0}^{k-1} \frac{1}{(i+1+K)^{2-\frac{\alpha \gamma_M}{2}}} \leq \frac{2}{2 - \alpha \gamma_M}$
    \item $\alpha \gamma_M = 2 \Rightarrow \sum_{i = 0}^{k-1} \frac{1}{(i+1+K)^{2-\frac{\alpha \gamma_M}{2}}} \leq \log (k+K)$
    \item $\alpha \gamma_M \in (2,4) \Rightarrow \sum_{i = 0}^{k-1} \frac{1}{(i+1+K)^{2-\frac{\alpha \gamma_M}{2}}} \leq \frac{2(k+K)^{\alpha \gamma_M -2}}{\alpha \gamma_M - 2}$
    \item $\alpha \gamma_M = 4 \Rightarrow \sum_{i = 0}^{k-1} \frac{1}{(i+1+K)^{2-\frac{\alpha \gamma_M}{2}}} = k$
    \item $\alpha \gamma_M \in (4,+\infty) \Rightarrow \sum_{i = 0}^{k-1} \frac{1}{(i+1+K)^{2-\frac{\alpha \gamma_M}{2}}} \leq \frac{e (k+K)^{\frac{\alpha \gamma_M - 2}{2}}}{\frac{\alpha \gamma_M}{2}-1}$
\end{itemize}
Thus for $\xi = 1$, we have:
\begin{itemize}
    \item $\alpha \gamma_M \in (0,2) \Rightarrow \frac{4\alpha^2}{(k+K)^{\frac{\alpha \gamma_M}{2}}} \sum_{i = 0}^{k-1} \frac{1}{(i+1+K)^{2-\frac{\alpha \gamma_M}{2}}} \leq \frac{8\alpha^2}{2-\alpha \gamma_M} \frac{1}{(k+K)^{\frac{\alpha \gamma_M}{2}}}$
    \item $\alpha \gamma_M = 2 \Rightarrow \frac{4\alpha^2}{(k+K)^{\frac{\alpha \gamma_M}{2}}} \sum_{i = 0}^{k-1} \frac{1}{(i+1+K)^{2-\frac{\alpha \gamma_M}{2}}} \leq \frac{4 \alpha^2\log(k+K)}{k+K}$
    \item $\alpha \gamma_M \in (2,+\infty) \Rightarrow \frac{4\alpha^2}{(k+K)^{\frac{\alpha \gamma_M}{2}}} \sum_{i = 0}^{k-1} \frac{1}{(i+1+K)^{2-\frac{\alpha \gamma_M}{2}}} \leq \frac{8e\alpha^2}{\alpha \gamma_M - 2} \frac{1}{k+K}$
\end{itemize}
For $\xi \in (0,1)$, we replicate the analysis in \cite{zaiwei-envelope}. Consider the sequence $\{u_k\}_{k \geq 0}:u_0 = 0, u_{k} = \left( 1-\frac{\alpha_{k-1} \gamma_M}{2} \right)u_{k-1} + \alpha_{k-1}^2$, we have $u_{k} = \sum_{i = 0}^{k-1} \alpha_i^2 \prod_{j = i+1}^{k-1} \left( 1 - \frac{\alpha_{k-1} \gamma_M}{2} \right)$. By induction, when $k \geq \max \left(0,\left( \frac{4\xi}{\alpha \gamma_M} \right)^{\frac{1}{1-\xi}} - K\right)$, we have:
\begin{align*}
    u_{k} \leq \frac{4\alpha}{\gamma_M} \frac{1}{(k+K)^{\xi}}
\end{align*}
and thus $\sum_{i = 0}^{k-1} \alpha_i^2 \prod_{j = i+1}^{k-1} \left( 1 - \frac{\alpha_j \gamma_M}{2} \right) \leq \frac{4\alpha}{\gamma_M (k+K)^{\xi}}$. Thus, for $\xi \in (0,1)$, we have the bound:
\begin{align*}
    \mathbb{E}[\norm{x_{k}-x^*}^2] &\leq C_{2,a}^{\frac{2}{a}} \br{\frac{1}{C_{1,a}^{\frac{2}{a}}} + 2\mu} \norm{x_0-x^*}^2 \exp\left[-\frac{\alpha \gamma_M}{2(1-\xi)}((k+K)^{1-\xi}-K^{1-\xi}) \right] \\
    &+ \frac{4\alpha\br{\frac{1}{C_{1,a}^{\frac{2}{a}}} + 2\mu}}{\gamma_M (k+K)^{\xi}} \frac{A + 2B\norm{x^*}^2}{\mu}
\end{align*}

\subsubsection{Smooth exponentially stable case}
\label{sssec:finite-time-corollary-0-proof}
In this subsection, we will present a proof outline for Theorem \ref{finite-time-corollary-0}. Since we already have a Lyapunov function $V$ whose gradient is Lipschitz, we do not have to construct a Moreau envelope to smoothen the Lyapunov function but instead directly use the existing Assumption \ref{assumption: smooth-time-derivative} to obtain the finite time bounds. We present the full description of the Theorem below:
\begin{theorem}
Under Assumptions \ref{assumption: noise},  \ref{assumption: F-lipschitz}, \ref{assumption: smooth-time-derivative}, \ref{assumption: gradient-lipschitz}, \ref{assumption: quadratic-growth} and the step size $\alpha_k = \frac{\alpha}{(k+K)^\xi}$ where $K = \max\left\lbrace 1, \frac{\alpha (4C^2 + 8B)}{\gamma} \right\rbrace$ for $\xi = 1$ and $K = \max \left\lbrace 1, \left( \frac{\alpha (4C^2 + 8B)}{\gamma} \right)^{\frac{1}{\xi}}, \left( \frac{2 \xi}{\alpha \gamma} \right)^{\frac{1}{1-\xi}}\right\rbrace$ for $\xi \in (0,1)$, we have:
\begin{align*}
    &\textbf{For $\xi = 1$}: \\
    &\mathbb{E}\sqbr{\norm{x_{k}-x^*}^2} \leq \frac{C_2}{C_1} \norm{x_0-x^*}^2 \br{\frac{K}{k+K}}^{\frac{\alpha \gamma}{2}} + \frac{L(A + 2B\norm{x^*}^2)}{C_1} \times
    \begin{cases}
        \frac{8\alpha^2}{2-\alpha \gamma} \frac{1}{(k+K)^{\frac{\alpha \gamma}{2}}} &\text{ if } \alpha \in \br{0, \frac{2}{\gamma}}\\
        \frac{4 \alpha^2 \log(k+K)}{k+K} &\text{ if } \alpha = \frac{2}{\gamma} \\
        \frac{8e\alpha^2}{\alpha \gamma - 2} \frac{1}{k+K} &\text{ if } \alpha \in \br{\frac{2}{\gamma}, \infty}
    \end{cases}
    \\
    &\textbf{For $\xi \in (0,1)$}: \\
    &\mathbb{E}\sqbr{\norm{x_{k}-x^*}^2} \leq \frac{C_2}{C_1} \norm{x_0-x^*}^2 \exp\left[-\frac{\alpha \gamma}{2(1-\xi)}((k+K)^{1-\xi}-K^{1-\xi}) \right] + \frac{4\alpha L(A + 2B\norm{x^*}^2)}{\gamma C_1 (k+K)^{\xi}} \qquad\qquad\qquad \\
    &\textbf{For $\xi = 0$}: \\
    &\mathbb{E}\sqbr{\norm{x_{k}-x^*}^2} \leq \frac{C_2}{C_1} \norm{x_0-x^*}^2 \left( 1-\frac{\alpha \gamma}{2} \right)^k + \frac{2\alpha L(A + 2B\norm{x^*}^2)}{\gamma C_1}
\end{align*}
\end{theorem}
We omit the proof of this Theorem since it is similar to the proof of Theorem \ref{finite-time-corollary-1} and the only difference is that we do not have to use a Moreau envelope to obtain a smooth Lyapunov function. Hence, the analysis is fairly straightforward.

\subsubsection{Smooth sub-exponentially stable case}
\label{sssec:finite-time-corollary-2-proof}
With the presence of smoothness, we are given the first-order upper bound and the negative drift of the Lyapunov function for free. However, for $c > 1$, applying the smooth inequality and the negative drift condition of the Lyapunov function would give:
\begin{align}
    \E\sqbr{V(x_{k}) | \cF_{k-1}} \leq V(x_{k-1})\left(1 + \frac{\alpha_k^2 L(C^2 + 2B)}{C_1}\right) - \alpha_{k-1} \gamma V(x_{k-1})^c + \alpha_k^2L(2B\norm{x^*}^2 + A)
\end{align}
In that case, we need to further bound $-(c-1)\alpha_{k-1}^2A - \alpha_{k-1} \gamma V(x_{k-1})^c \leq -c\alpha_{k-1}^{2-\frac{1}{c}}A^{1-\frac{1}{c}} \gamma^{\frac{1}{c}} V(x_{k-1})$ using AM-GM to absorb this term into the contraction term. Thus, now we have a similar one-iterate bound albeit with different step size dependence. Following this line of idea, we obtain this generalized result that would later be used to obtain finite-time bounds:
\begin{prop}
\label{main-theorem-2}
Let $V: \mathbb{R}^d \rightarrow \mathbb{R}$ be a $L$-smooth Lyapunov function satisfying Assumptions \ref{assumption: polynomial-growth} and \ref{assumption: asymptotic-stability}. We have:
\begin{align*}
    \mathbb{E}[V(x_{k}) | \mathcal{F}_{k-1}] &\leq C_2 \norm{x_0-x^*}^2 \prod_{j = 0}^{k-1} \left(1 - \alpha_j^{2-\frac{1}{c}}A^{1-\frac{1}{c}}\gamma^{\frac{1}{c}}\right) \\
    &+ \left[ \sum_{i = 0}^{k-1} \alpha_i^2 \prod_{j = i+1}^{k-1} \left( 1 - \alpha_j^{2-\frac{1}{c}} A^{1-\frac{1}{c}}\gamma^{\frac{1}{c}} \right) \right]\frac{L(cA + 2B\norm{x^*}^2)}{\mu}
\end{align*}
where $c \geq 1$.
\end{prop}
\begin{proof}
From the smoothness Assumption \ref{assumption: gradient-lipschitz}, we obtain the bound:
\begin{align}
    \label{one-iterate-bound-smooth}
    V(x_{k}) &\leq V(x_{k-1}) + \langle \nabla V(x_{k-1}), x_{k}-x_{k-1} \rangle + \frac{L\norm{x_{k}-x_{k-1}}^2}{2}
\end{align}
Combined with Assumption \ref{assumption: asymptotic-stability}, it follows that:
\begin{align*}
    V(x_{k}) &\leq V(x_{k-1}) - \alpha_{k-1} \gamma V(x_{k-1})^c + \alpha_{k-1}^2L\norm{F(x_{k-1})+w_{k-1}}^2/2 \\
    &\leq V(x_{k-1}) - \alpha_{k-1} \gamma V(x_{k-1})^c + \alpha_{k-1}^2 L ((C^2 + 2B) \norm{x_{k-1}-x^*}^2 + 2B \norm{x^*}^2 + A) \\
    &\leq V(x_{k-1})\left(1 + \frac{\alpha_{k-1}^2 L(C^2 + 2B)}{C_1}\right) - \alpha_{k-1} \gamma V(x_{k-1})^c + \alpha_{k-1}^2L(2B\norm{x^*}^2 + A)
\end{align*}
Now, we consider two separate cases: if $c = 1$ then we obtain:
\begin{align*}
    V(x_{k}) \leq V(x_{k-1})\left(1 - \alpha_{k-1} \gamma + \frac{\alpha_{k-1}^2 L(C^2 + 2B)}{C_1}\right) + \alpha_{k-1}^2L(2B\norm{x^*}^2 + A)
\end{align*}
In this case, we can simply proceed similarly to the proof of \ref{finite-time-corollary-1} to obtain a similar rate. If $c > 1$, we need to bound $V(x_{k-1})^c$ somehow in order to obtain a contraction on $V$ plus some noise. By AM-GM, we have the bound $(c-1)\alpha_{k-1}^2A + \alpha_{k-1} \gamma V(x_{k-1})^c \geq c\alpha_{k-1}^{2-\frac{1}{c}}A^{1-\frac{1}{c}} \gamma^{\frac{1}{c}} V(x_{k-1})$, and thus we obtain:
\begin{align*}
    V(x_{k}) &\leq V(x_{k-1})\left(1 + \frac{\alpha_{k-1}^2L(C^2 + 2B)}{C_1}\right) - c\alpha_{k-1}^{2-\frac{1}{c}}A^{1-\frac{1}{c}}\gamma^{\frac{1}{c}} V(x_{k-1}) + \alpha_{k-1}^2L(2B\norm{x^*}^2 + cA) \\
    &= V(x_{k-1})\left(1 - c\alpha_{k-1}^{2-\frac{1}{c}}A^{1-\frac{1}{c}}\gamma^{\frac{1}{c}} + \frac{\alpha_{k-1}^2L(C^2 + 2B)}{C_1}\right) + \alpha_{k-1}^2L(2B\norm{x^*}^2 + cA) \\
    &\leq V(x_{k-1})\br{1-\alpha_{k-1}^{2-\frac{1}{c}}A^{1-\frac{1}{c}}\gamma^{\frac{1}{c}}} + \alpha_{k-1}^2L(2B\norm{x^*}^2 + cA)
\end{align*}
for all $\alpha_k \leq \frac{C_{1,a}^c(c-1)^c A^{c-1}\gamma}{L^c(C^2 + 2B)^c}, c > 1$. Thus, we have the SA bound:
\begin{align*}
    \mathbb{E}[V(x_{k}) | \mathcal{F}_{k-1}] &\leq C_2 \norm{x_0-x^*}^2 \prod_{j = 0}^{k-1} \left(1 - \alpha_j^{2-\frac{1}{c}}A^{1-\frac{1}{c}}\gamma^{\frac{1}{c}}\right) \\
    &+ \left[ \sum_{i = 0}^{k-1} \alpha_i^2 \prod_{j = i+1}^{k-1} \left( 1 - \alpha_j^{2-\frac{1}{c}} A^{1-\frac{1}{c}}\gamma^{\frac{1}{c}} \right) \right]\frac{L(cA + 2B\norm{x^*}^2)}{\mu}
\end{align*}
\end{proof}
\begin{theorem}
Under Assumptions \ref{assumption: noise}, \ref{assumption: F-lipschitz}, \ref{assumption: gradient-lipschitz}, \ref{assumption: quadratic-growth}, \ref{assumption: asymptotic-stability}, \ref{assumption: bounded-iterates} and the step size $\alpha_k = \frac{\alpha}{(k+K)^\xi}$ where for $\xi = \frac{c}{2c-1}: K = \max\left\lbrace 1, \frac{\alpha (4C^2 + 8B)}{\gamma} \right\rbrace$ and $0 \leq \xi \leq \frac{c}{2c-1}: K = \max \left\lbrace 1, \left( \frac{\alpha (4C^2 + 8B)}{\gamma} \right)^{\frac{1}{\xi}}, \left( \frac{2 \xi}{\alpha \gamma} \right)^{\frac{1}{1-\xi}}\right\rbrace$ where $c > 1$. The finite-time bound of the SA algorithm is:
\\
\textbf{For $\xi = \frac{c}{2c-1}$}:
\begin{align*}
\mathbb{E}[\norm{x_{k}-x^*}^2] &\leq \frac{C_2}{C_1} \norm{x_0-x^*}^2 \left( \frac{K}{k+K} \right)^\phi +
\begin{cases}
    \frac{2^{\frac{2c}{2c-1}}\alpha^2}{(k+K)^\phi}\frac{L (cA + 2B\norm{x^*}^2)}{C_1 \left(\frac{1}{2c-1}-\phi \right)} &\text{ if } 0 < \alpha < \tau \\
    \frac{2^{\frac{2c}{2c-1}}\alpha^2 \log(k+K)}{(k+K)^{\frac{1}{2c-1}}} \frac{L (cA + 2B\norm{x^*}^2)}{C_1} &\text{ if } \alpha = \tau \\
    \frac{e2^{\frac{2c}{2c-1}}\alpha^2}{(k+K)^{\frac{1}{2c-1}}} \frac{L (cA + 2B\norm{x^*}^2)}{C_1 \br{\phi - \frac{1}{2c-1}}} &\text{ if } \alpha > \tau
\end{cases}
\end{align*}
\textbf{For $\xi \in \left(0,\frac{c}{2c-1}\right)$}:
\begin{align*}
    \mathbb{E}[\norm{x_{k}-x^*}^2] \leq \frac{C_2}{C_1} \norm{x_0-x^*}^2 \exp\left[-\frac{\phi \br{(k+K)^{1-\frac{(2c-1)\xi}{c}}-K^{1-\frac{(2c-1)\xi}{c}}}}{\br{1-\frac{(2c-1)\xi}{c}}} \right] + \frac{2\alpha^2}{C_1 \phi (k+K)^{\frac{\xi}{c}}}
\end{align*}
\textbf{For $\xi = 0$}:
\begin{align*}
    \mathbb{E}[\norm{x_{k}-x^*}^2] \leq \frac{C_2}{C_1} \norm{x_0-x^*}^2 \left(1 - \phi\right)^k + \frac{L \alpha^2(cA + 2B\norm{x^*}^2)}{C_1 \phi}
\end{align*}
where $\phi = \alpha^{2-1/c} A^{1-1/c} \gamma^{1/c}, \tau = \br{\frac{1}{(2c-1)A^{1-1/c}\gamma^{1/c}}}^{\frac{c}{2c-1}}$.
\end{theorem}
\begin{proof}
From Proposition \ref{main-theorem-2} and Assumption \ref{assumption: quadratic-growth}, we have:
\begin{align}
    \label{mse-bound-2}
    \nonumber
    \mathbb{E}[\norm{x_{k}-x^*}^2] \leq \frac{C_2}{C_1} \norm{x_0-x^*}^2 \prod_{j = 0}^{k-1} \left(1 - \alpha_j^{2-\frac{1}{c}}A^{1-\frac{1}{c}}\gamma^{\frac{1}{c}}\right) + \\
    \sum_{i = 0}^{k-1} \alpha_i^2 \prod_{j = i+1}^{k-1} \left( 1 - \alpha_{j-1}^{2-\frac{1}{c}} A^{1-\frac{1}{c}}\gamma^{\frac{1}{c}} \right) \frac{L (cA + 2B\norm{x^*}^2)}{\mu C_1}
\end{align}
Let
\begin{align*}
    T_1 &= \frac{C_2}{C_1} \norm{x_0-x^*}^2 \prod_{j = 1}^{k-1} \left(1 - \alpha_j^{2-\frac{1}{c}} A^{1-\frac{1}{c}}\gamma^{\frac{1}{c}}\right) \\
    T_2 &= \sum_{i = 0}^{k-1} \alpha_i^2 \prod_{j = i+1}^{k-1} \left( 1 - \alpha_j^{2-\frac{1}{c}} A^{1-\frac{1}{c}}\gamma^{\frac{1}{c}} \right)
\end{align*}
We first bound the $T_1$ term as follows:
\begin{align*}
    \nonumber
    T_1 = \frac{C_2}{C_1} \norm{x_0-x^*}^2 \prod_{j = 0}^{k-1} \left(1 - \alpha_j^{2-\frac{1}{c}} A^{1-\frac{1}{c}}\gamma^{\frac{1}{c}}\right) &\leq \frac{C_2}{C_1} \norm{x_0-x^*}^2 e^{- A^{1-\frac{1}{c}}\gamma^{\frac{1}{c}} \sum_{i = 0}^{k-1} \alpha_i^{2-\frac{1}{c}}} \\
    &= \frac{C_2}{C_1} \norm{x_0-x^*}^2 e^{-A^{1-\frac{1}{c}}\gamma^{\frac{1}{c}} \sum_{i = 0}^{k-1} \frac{\alpha^{2-\frac{1}{c}}}{(i+K)^{\frac{(2c-1)\xi}{c}}}}
\end{align*}
and since $\int_a^{b+1} h(x) dx \leq \sum_{n = a}^n h(n) \leq \int_{a-1}^b h(x) dx$ for any non-increasing function $h$, we have for $\xi = \frac{c}{2c-1}$:
\begin{align}
    \frac{C_2}{C_1} \norm{x_0-x^*}^2 \prod_{j = 0}^{k-1} \left(1 - \alpha_j^{2-\frac{1}{c}}A^{1-\frac{1}{c}}\gamma^{\frac{1}{c}}\right) \leq \frac{C_2}{C_1} \norm{x_0-x^*}^2 \left( \frac{K}{k+K} \right)^{\alpha^{2-\frac{1}{c}} A^{1-\frac{1}{c}} \gamma^{\frac{1}{c}}}.
\end{align}
and for $\xi \in \br{0,\frac{c}{2c-1}} \Rightarrow \frac{(2c-1)\xi}{c} \in (0,1)$, we have:
\begin{align}
    \frac{C_2}{C_1} \norm{x_0-x^*}^2 \prod_{j = 0}^{k-1} \left(1 - \alpha_j^{2-\frac{1}{c}}A^{1-\frac{1}{c}}\gamma^{\frac{1}{c}}\right) \leq \frac{C_2}{C_1} \norm{x_0-x^*}^2 \exp\left[-\frac{\alpha^{2-\frac{1}{c}} A^{1-\frac{1}{c}} \gamma^{\frac{1}{c}}}{\br{1-\frac{\xi (2c-1)}{c}}}((k+K)^{1-\frac{\xi (2c-1)}{c}}-K^{1-\frac{\xi (2c-1)}{c}}) \right].
\end{align}
Next, we will bound the $T_2$ term. We have:
\begin{align*}
    T_2 = \sum_{i = 0}^{k-1} \alpha_i^2 \prod_{j = i+1}^{k-1} \left( 1 - \alpha_j^{2-\frac{1}{c}} A^{1-\frac{1}{c}}\gamma^{\frac{1}{c}} \right)
\end{align*}
For $\xi = \frac{c}{2c-1}$ and let $d = \frac{2c}{2c-1}$, we have:
\begin{align*}
    \sum_{i = 0}^{k-1} \alpha_i^2 \prod_{j = i+1}^{k-1} \left( 1 - \alpha_j^{2-\frac{1}{c}} A^{1-\frac{1}{c}}\gamma^{\frac{1}{c}} \right) &= \alpha^2 \sum_{i = 0}^{k-1} \frac{1}{(i+K)^d} \prod_{j = i+1}^{k-1} \left( 1 - \alpha_j^{2-\frac{1}{c}} A^{1-\frac{1}{c}}\gamma^{\frac{1}{c}} \right) \\
    &\leq  \alpha^2 \sum_{i = 0}^{k-1} \frac{1}{(i+K)^d} \exp \left( -\alpha^{2-\frac{1}{c}} A^{1-\frac{1}{c}}\gamma^{\frac{1}{c}} \sum_{j = i+1}^{k-1} \frac{1}{j + K} \right) \\
    &\leq \alpha^2 \sum_{i = 0}^{k-1} \frac{1}{(i+K)^d} \exp \left( -\alpha^{2-\frac{1}{c}} A^{1-\frac{1}{c}}\gamma^{\frac{1}{c}} \int_{j = i+1}^{k} \frac{1}{j + K} \right) \\
    &\leq \alpha^2 \sum_{i = 0}^{k-1} \frac{1}{(i+K)^d} \left( \frac{i+1+K}{k+K} \right)^{\alpha^{2-\frac{1}{c}} A^{1-\frac{1}{c}}\gamma^{\frac{1}{c}}}  \\
    &\leq \frac{2^d\alpha^2}{(k+K)^{\alpha^{2-\frac{1}{c}} A^{1-\frac{1}{c}}\gamma^{\frac{1}{c}}}} \sum_{i = 0}^k \frac{1}{(i+1+K)^{d-\alpha^{2-\frac{1}{c}} A^{1-\frac{1}{c}}\gamma^{\frac{1}{c}}}}
\end{align*}
\begin{itemize}
    \item When $\alpha^{2-\frac{1}{c}} A^{1-\frac{1}{c}}\gamma^{\frac{1}{c}} \in (0, d-1)$ then $\sum_{i = 0}^{k-1} \frac{1}{(i+1+K)^{d-\alpha^{2-\frac{1}{c}} A^{1-\frac{1}{c}}\gamma^{\frac{1}{c}}}} \leq \frac{1}{d-1-\alpha^{2-\frac{1}{c}} A^{1-\frac{1}{c}}\gamma^{\frac{1}{c}}}$
    \item When $\alpha^{2-\frac{1}{c}} A^{1-\frac{1}{c}}\gamma^{\frac{1}{c}} = d-1$ then $\sum_{i = 0}^{k-1} \frac{1}{(i+1+K)^{d-\alpha^{2-\frac{1}{c}} A^{1-\frac{1}{c}}\gamma^{\frac{1}{c}}}} \leq \log(k+K)$
    \item When $\alpha^{2-\frac{1}{c}} A^{1-\frac{1}{c}}\gamma^{\frac{1}{c}} \in (d-1,d)$ then $\sum_{i = 0}^{k-1} \frac{1}{(i+1+K)^{d-\alpha^{2-\frac{1}{c}} A^{1-\frac{1}{c}}\gamma^{\frac{1}{c}}}} \leq \frac{(k+K)^{\alpha^{2-\frac{1}{c}} A^{1-\frac{1}{c}}\gamma^{\frac{1}{c}}-(d-1)})}{\alpha^{2-\frac{1}{c}} A^{1-\frac{1}{c}}\gamma^{\frac{1}{c}}-(d-1)}$
    \item When $\alpha^{2-\frac{1}{c}} A^{1-\frac{1}{c}}\gamma^{\frac{1}{c}} = d$ then $\sum_{i = 0}^{k-1} \frac{1}{(i+1+K)^{d-\alpha^{2-\frac{1}{c}} A^{1-\frac{1}{c}}\gamma^{\frac{1}{c}}}} = k$.
    \item When $\alpha^{2-\frac{1}{c}} A^{1-\frac{1}{c}}\gamma^{\frac{1}{c}} > d$ then $\sum_{i = 0}^{k-1} \frac{1}{(i+1+K)^{d-\alpha^{2-\frac{1}{c}} A^{1-\frac{1}{c}}\gamma^{\frac{1}{c}}}} \leq \frac{ (k+K)^{\alpha^{2-\frac{1}{c}} A^{1-\frac{1}{c}} \gamma^{\frac{1}{c}}-(d-1)}}{ A^{1-\frac{1}{c}} \gamma^{\frac{1}{c}} - (d-1)}$
\end{itemize}
In summary, when $\xi = \frac{c}{2c-1}$ we have:
\begin{align*}
    \mathbb{E}[\norm{x_{k}-x^*}^2] &\leq T_1 + T_2 \frac{L (cA + 2B\norm{x^*}^2)}{C_1}
\end{align*}
Thus from the bounds above:
\begin{align*}
\mathbb{E}[\norm{x_{k}-x^*}^2] &\leq \frac{C_2}{C_1} \norm{x_0-x^*}^2 \left( \frac{K}{k+K} \right)^{\alpha^{2-\frac{1}{c}} A^{1-\frac{1}{c}} \gamma^{\frac{1}{c}}} \\
&+
\begin{cases}
    \frac{2^d\alpha^2}{(k+K)^{\alpha_j^{2-\frac{1}{c}} A^{1-\frac{1}{c}}\gamma^{\frac{1}{c}}}} \frac{L (cA + 2B\norm{x^*}^2)}{C_1 \left(d-1-\alpha^{2-\frac{1}{c}} A^{1-\frac{1}{c}}\gamma^{\frac{1}{c}}\right)} &\text{ if } 0 < \alpha < \br{\frac{d-1}{A^{1-\frac{1}{c}}\gamma^{\frac{1}{c}}}}^{\frac{c}{2c-1}}\\
    \frac{2^d\alpha^2 \log(k+K)}{(k+K)^{d-1}} \frac{L (cA + 2B\norm{x^*}^2)}{C_1} &\text{ if } \alpha = \br{\frac{d-1}{A^{1-\frac{1}{c}}\gamma^{\frac{1}{c}}}}^{\frac{c}{2c-1}} \\
    \frac{e2^d\alpha^2}{(k+K)^{d-1}} \frac{L (cA + 2B\norm{x^*}^2)}{C_1 \br{\alpha^{2-\frac{1}{c}} A^{1-\frac{1}{c}}\gamma^{\frac{1}{c}} - (d-1)}} &\text{ if } \alpha > \br{\frac{d-1}{A^{1-\frac{1}{c}}\gamma^{\frac{1}{c}}}}^{\frac{c}{2c-1}}
\end{cases}
\end{align*}
For $\xi \in \br{0,\frac{c}{2c-1}}$ (that is $\xi' =  \frac{(2c-1)\xi}{c} \in (0,1)$), consider the sequence $\{u_k\}_{k \geq 0}:u_0 = 0, u_{k+1} = \br{1-\alpha_j^{2-\frac{1}{c}} A^{1-\frac{1}{c}} \gamma^{\frac{1}{c}}}u_k + \alpha_k^2$, we have that $T_2 \leq u_k$. We will obtain a finite-time bound for $T_2$ by bounding the iterates of the sequence $\{u_k\}$ via induction. Assume that $\frac{\alpha^{\frac{1}{c}} D}{(k+K)^{(d-1)\xi'}} \geq u_k$ for some constant $D > 0$, we have:
\begin{align*}
    \frac{\alpha^{\frac{1}{c}} D}{(k+1+K)^{\frac{\xi'}{2c-1}}} - u_{k+1} &= \frac{\alpha^{\frac{1}{c}} D}{(k+1+K)^{\frac{\xi'}{2c-1}}} - \left((1-\alpha_k^{2-\frac{1}{c}} A^{1-\frac{1}{c}} \gamma^{\frac{1}{c}})u_{k} + \alpha^2\right) \\
    &\geq \frac{\alpha^{\frac{1}{c}} D}{(k+1+K)^{(d-1)\xi'}} - \left(\frac{\alpha^{\frac{1}{c}} D (1-\alpha_k^{2-\frac{1}{c}} A^{1-\frac{1}{c}} \gamma^{\frac{1}{c}})}{(k+K)^{(d-1)\xi'}} + \alpha_k^2\right) \\
    &= \frac{\alpha^{\frac{1}{c}} D}{(k+1+K)^{(d-1)\xi'}} - \left(\frac{\alpha^{\frac{1}{c}} D}{(k+K)^{(d-1)\xi'}} - \frac{\alpha^2 A^{1-\frac{1}{c}} \gamma^{\frac{1}{c}} D}{(k+K)^{d \xi'}} + \frac{\alpha^2}{(k+K)^{d\xi'}} \right) \\
    &= \frac{1}{(k + K)^{(d-1)\xi'}} \times \left( \alpha^{\frac{1}{c}} D\left(\frac{k+K}{k+1+K}\right)^{(d-1)\xi'} -\alpha^{\frac{1}{c}} D - \frac{\alpha^2 - \alpha^2 A^{1-\frac{1}{c}} \gamma^{\frac{1}{c}} D}{(k+K)^{\xi'}} \right) \\
    &\geq \frac{1}{(k + K)^{(d-1)\xi'}} \left( \alpha^{\frac{1}{c}} D\left(1 - \frac{(d-1)\xi'}{k+K} - 1 \right) + \frac{\alpha^2 A^{1-\frac{1}{c}} \gamma^{\frac{1}{c}} D - \alpha^2}{(k+K)^{\xi'}} \right) \\
    &= \frac{1}{(k + K)^{(d-1)\xi'}} \left( -\frac{\alpha^{\frac{1}{c}} D(d-1)\xi'}{k+K} + \frac{\alpha^2 A^{1-\frac{1}{c}} \gamma^{\frac{1}{c}} D - \alpha^2}{(k+K)^{\xi'}} \right) \geq 0
\end{align*}
for $c \geq 1 \Rightarrow d = \frac{2c}{2c-1} \in (1,2)$ and sufficiently large $K$ and some properly chosen constant $D$. Choose $D = \frac{2}{A^{1-\frac{1}{c}}\gamma^{\frac{1}{c}}}, K \geq \left[\frac{2\xi'}{\alpha^{2-\frac{1}{c}} (2c-1)A^{1-\frac{1}{c}}\gamma^{\frac{1}{c}}}\right]^{\frac{1}{1-\xi'}}$, we have that $u_k \leq \frac{2\alpha^{\frac{1}{c}}}{A^{1-\frac{1}{c}}\gamma^{\frac{1}{c}}(k+K+1)^{\frac{\xi'}{2c-1}}}$. The first inequality is obtained from the induction hypothesis and the second inequality is obtained from $\left(\frac{k+K}{k+K+1}\right)^{\frac{\xi'}{2c-1}} \geq \left[ \left( 1+\frac{1}{k+K} \right)^{k+K} \right]^{-\frac{\xi'}{(2c-1)(k+K)}} \geq e^{-\frac{\xi'}{(2c-1)(k+K)}} \geq 1 -\frac{\xi'}{(2c-1)(k+K)}$. Hence, we have that $T_2 \leq \frac{2\alpha^{\frac{1}{c}}}{A^{1-\frac{1}{c}}\gamma^{\frac{1}{c}}(k+K)^{(d-1)\xi'}}$, thus gives:
\begin{align*}
    \mathbb{E}[\norm{x_{k}-x^*}^2] \leq \frac{C_2}{C_1} \norm{x_0-x^*}^2 \exp\left[-\frac{\alpha^{2-\frac{1}{c}} A^{1-\frac{1}{c}} \gamma^{\frac{1}{c}}}{\br{1-\frac{\xi (2c-1)}{c}}}\br{(k+K)^{1-\frac{\xi (2c-1)}{c}}-K^{1-\frac{\xi (2c-1)}{c}}} \right] \\
    + \frac{2\alpha^{\frac{1}{c}}}{C_1 A^{1-\frac{1}{c}}\gamma^{\frac{1}{c}}(k+K)^{\frac{\xi}{c}}}
\end{align*}
For $\xi = 0$, from Theorem \ref{main-theorem-2} we have:
\begin{align*}
    \mathbb{E}[V(x_{k}) | \mathcal{F}_k] &\leq C_2 \norm{x_0-x^*}^2 \left(1 - \alpha^{2-\frac{1}{c}} A^{1-\frac{1}{c}}\gamma^{\frac{1}{c}}\right)^k \\
    &+ \left[ \sum_{i = 0}^{k-1} \alpha^2 \left( 1 - \alpha^{2-\frac{1}{c}} A^{1-\frac{1}{c}}\gamma^{\frac{1}{c}} \right)^{k-1-i} \right]L(cA + 2B\norm{x^*}^2) \\
    &\leq C_2 \norm{x_0-x^*}^2 \left(1 - \alpha^{2-\frac{1}{c}} A^{1-\frac{1}{c}}\gamma^{\frac{1}{c}}\right)^k + \frac{L \alpha^2(cA + 2B\norm{x^*}^2)}{\alpha^{2-\frac{1}{c}} A^{1-\frac{1}{c}}\gamma^{\frac{1}{c}}} \\
    &= C_2 \norm{x_0-x^*}^2 \left(1 - \alpha^{2-\frac{1}{c}} A^{1-\frac{1}{c}}\gamma^{\frac{1}{c}}\right)^k + \frac{L \alpha^{\frac{1}{c}}(cA + 2B\norm{x^*}^2)}{A^{1-\frac{1}{c}}\gamma^{\frac{1}{c}}}
\end{align*}
from $C_1 \norm{x-x^*}^2 \leq V(x)$ in Assumption \ref{assumption: quadratic-growth}, we have
\begin{align*}
    \mathbb{E}[\norm{x_{k}-x^*}^2] &\leq \frac{C_2}{C_1} \norm{x_0-x^*}^2 \left(1 - \alpha^{2-\frac{1}{c}} A^{1-\frac{1}{c}}\gamma^{\frac{1}{c}}\right)^k + \frac{L \alpha^{\frac{1}{c}}(cA + 2B\norm{x^*}^2)}{C_1 A^{1-\frac{1}{c}}\gamma^{\frac{1}{c}}}
\end{align*}
\end{proof}

\subsubsection{Nonsmooth sub-exponentially stable case}
\label{sssec:finite-time-corollary-3-proof}
In this section, we will complement the results of the previous section by considering general non-smooth systems that satisfy sub-exponential stability. First, we obtain the negative drift of the Moreau envelope. Let $R = V^{\frac{2}{a}}$ and $M_\mu = R \square \frac{\norm{x}^2}{2\mu}$, we obtain the following bound on the time-derivative of the Moreau envelope:
\begin{lemma}
    Under Assumptions \ref{assumption: polynomial-growth}, \ref{assumption: generalized-gradient-bound} and \ref{assumption: clarke-asymptotic-stability} and let $R = V^{\frac{2}{a}}$ and $M_{\mu} = R \square \frac{\norm{x}^2}{2\mu}$, we have:
    \begin{align}
        \langle \nabla M_{\mu}(x), F(x) \rangle \leq -\frac{2\gamma\left(\frac{C_1}{C_2}\right)^{c-1+\frac{2}{a}}}{a\left(1 + \mu G_M\right)^2} M_{\mu}(x)^{\frac{a(c-1)}{2}+1} + \mu C G_M \left( \frac{C_{2,a}^{\frac{2}{a}}}{C_{1,a}^{\frac{2}{a}}} + 2C_{2,a}^{\frac{2}{a}}\mu \right)M_{\mu}(x) \forall x \in \mathbb{R}^d,
    \end{align}
    where $G_M = \frac{2GC_{2,a}^{\frac{2}{a}-1}}{a}$.
\end{lemma}
\begin{proof}
Indeed, approaching similarly to the proof of Lemma \ref{lemma: moreau-negative-drift}, we have:
\begin{align*}
    \langle \nabla M_{\mu}(x), F(x) \rangle &\leq \langle \nabla R(u), F(u) \rangle + \langle \nabla M_{\mu}(x), F(x)-F(u) \rangle\\
    &= -\frac{2\gamma}{a} R(u)^{\frac{a(c-1)}{2}+1} + C \norm{x-u}\norm{\nabla M_{\mu}(x)} \\
    &\leq -\frac{2\gamma C_{1,a}^{c-1+\frac{2}{a}}}{a} \norm{u-x^*}^{a(c-1)+2} + \frac{2 \mu C G C_{2,a}^{\frac{2}{a}-1}}{a}  \norm{u-x^*}^2 \\
    &\leq -\frac{2\gamma\left(\frac{C_1}{C_2}\right)^{c-1+\frac{2}{a}}}{a\left(1 + \mu G_M\right)^2} R(x)^{\frac{a(c-1)}{2}+1} + \frac{2 \mu C G}{a} R(x) \\
    &\leq -\frac{2\gamma\left(\frac{C_1}{C_2}\right)^{c-1+\frac{2}{a}}}{a\left(1 + \mu G_M\right)^2} M_{\mu}(x)^{\frac{a(c-1)}{2}+1} + \mu C G_M \left( \frac{C_{2,a}^{\frac{2}{a}}}{C_{1,a}^{\frac{2}{a}}} + 2C_{2,a}^{\frac{2}{a}}\mu \right)M_{\mu}(x)
\end{align*}
where the first inequality is from Lemma \ref{lemma: rescaled-time-derivative-condition} and Assumption \ref{assumption: F-lipschitz}, the second inequality is from Assumption \ref{assumption: polynomial-growth} and Lemma \ref{lemma: moreau-gradient-bound}, the third inequality is from Corollary \ref{u-lower-bound-general} and the nonexpansive property of the prox operator, and Lemma \ref{lemma: envelope-properties}.
\end{proof}
Notice that we don't have a negative semi-definite RHS as we had in Lemma \ref{lemma: moreau-negative-drift}. Thus, naively applying the Moreau envelope here will not give convergence even for diminishing step sizes. Indeed, with constant $\mu$, there is an additional positive term $\frac{2\mu CG}{a}R(x)$ which cannot be canceled out by the other non-negative term for sufficiently small $x$, meaning that we would not be able to obtain convergence for arbitrary accuracy. This motivates us to modify the Moreau envelope by having the parameter $\mu$ to be adaptive so that the additional non-negative term will vanish to $0$ as we slowly drive $\mu \rightarrow 0$.

If we choose $E_{k+1} = \E[M_{\mu_{k+1}}(x_{k+1})|\mathcal{F}_k]$ as our potential function then the one-iterate bound Proposition \ref{prop: one-iterate-bound-adaptive-mu} is similar to \eqref{one-iterate-bound-smooth} which we can follow a similar proof strategy to obtain finite-time convergence. However, the difference is that since we now have a time-varying $\mu_k$ parameter, a careful choice of $\mu_k$ is required in order to achieve competitive rates. Choosing a $\mu_k$ that converges to $0$ too slowly would harm the rate at which $E_k$ converges while if $\mu_k$ converges too fast then the noise term would grow too quickly and it would harm the convergence rate regardless. From Lemma \ref{lemma: moreau-negative-drift-2}, we obtain the following one-iterate bound:
\begin{prop}
\label{prop: one-iterate-bound-adaptive-mu}
Let $R = V^{\frac{2}{a}}$ and $M_{\mu} = R \square \frac{\norm{x}^2}{2\mu}$, there exists constants $\nu_1, \nu_2, \nu_3, \nu_4, A^*, B^* > 0$ such that for all $k \in \mathbb{Z}^+$:
\begin{align}
    \nonumber
    \mathbb{E}[M_{\mu_{k+1}}(x_{k+1}) | \mathcal{F}_k] &\leq \left(1 + \alpha_k \mu_k \nu_1 + (\mu_k-\mu_{k+1})\nu_2 + \frac{\alpha_k^2\nu_3}{\mu_k} \right)M_{\mu_k}(x_k) \\
    &- \alpha_k \nu_4 M_{\mu_k}(x_k)^{\frac{a(c-1)}{2}+1} + \frac{\alpha_k^2}{\mu_k} \frac{\br{\mu_0^2 G_M^2 + 1} N_C}{2}
\end{align}
where
\begin{align*}
    G_M &= \frac{2}{a} G C_{2,a}^{\frac{2}{a}-1}, N_C = (C^2+1)(A+2B\norm{x^*}^2) \\
    \nu_1 &= C G_M \left( \frac{C_{2,a}^{\frac{2}{a}}}{C_{1,a}^{\frac{2}{a}}} + 2C_{2,a}^{\frac{2}{a}}\mu_0 \right), \nu_2 = G_M^2 \left( \frac{1}{C_{1,a}^{\frac{2}{a}}} + 2\mu_0 \right),\\
    \nu_3 &= (C^2+1)(2B+1) \left(\mu_0^2 G_M^2 + 1\right) \left( \frac{1}{2C_{1,a}^{\frac{2}{a}}} + \mu_0 \right), \nu_4 = \frac{2 \gamma\left(\frac{C_1}{C_2}\right)^{c-1+\frac{2}{a}}}{a\left(1 + \mu_0 G_M\right)^2},
\end{align*}
and the smoothness parameter $\mu_k$ is positive and decreasing.
\end{prop}
\begin{proof}
    From $\frac{1}{\mu_{k}}$-smoothness of $M_{\mu_{k}}$, we have:
    \begin{align*}
        M_{\mu_{k}}(x_{k+1}) &\leq M_{\mu_{k}}(x_k) + \langle \nabla M_{\mu_{k}}(x_k), x_{k+1}-x_k \rangle + \frac{\norm{x_{k+1}-x_k}^2}{2\mu_{k}}
    \end{align*}
    From Lemma \ref{mu-derivative-bound}, we have:
    \begin{align*}
        M_{\mu_{k+1}(x_{k+1})} \leq M_{\mu_k}(x_{k+1}) +  (\mu_k-\mu_{k+1})G_M^2 \norm{x_{k+1}-x^*}^2
    \end{align*}
    Thus:
    \begin{align*}
        M_{\mu_{k+1}}(x_{k+1)} &\leq M_{\mu_{k}}(x_k) + \underbrace{\langle \nabla M_{\mu_{k}}(x_k), x_{k+1}-x_k \rangle}_{\text{Drift term}} + \underbrace{(\mu_k-\mu_{k+1})G_M^2 \norm{x_{k+1}-x^*}^2}_{\text{Partial derivative of } M \text{ w.r.t } \mu} + \underbrace{\frac{\norm{x_{k+1}-x_k}^2}{2\mu_{k}}}_{\text{Noise term}}
    \end{align*}
    Taking expectations on both sides and denote $E_{k+1} = \mathbb{E}[M_{\mu_{k+1}}(x_{k+1}) | \mathcal{F}_k]$, we have:
    \begin{align*}
        &E_{k+1} \leq M_{\mu_{k}}(x_k) + \alpha_k \langle \nabla M_{\mu_{k}}(x_k), F(x_k) \rangle + (\mu_k-\mu_{k+1})G_M^2 \norm{x_{k+1}-x^*}^2 + \frac{\E\sqbr{\norm{x_{k+1}-x_k}^2 | \cF_k}}{2\mu_{k}}.
    \end{align*}
    The RHS can be further bounded as:
    \begin{align*}
        E_{k+1}
        &\overset{(a)}{\leq} M_{\mu_{k}}(x_k) + \alpha_k \left[ \frac{-2\gamma\left(\frac{C_1}{C_2}\right)^{c-1+\frac{2}{a}}}{a\left(1 + \mu_{k} G_M\right)^2} M_{\mu_{k}}(x_k)^{\frac{a(c-1)}{2}+1} + \mu_{k} C G_M \left( \frac{C_{2,a}^{\frac{2}{a}}}{C_{1,a}^{\frac{2}{a}}} + 2C_{2,a}^{\frac{2}{a}}\mu_{k} \right)M_{\mu_{k}}(x_k)\right] \\
        &+ (\mu_k-\mu_{k+1}) G_M^2\br{\frac{\norm{x_{k+1}-x_k}^2+\norm{x_k-x^*}^2}{2}} + \frac{\E\sqbr{\norm{x_{k+1}-x_k}^2 | \cF_k}}{2\mu_{k}} \\
        &\overset{(b)}{\leq} M_{\mu_k}(x_k)\left[1 + \alpha_k \mu_k C G_M \left( \frac{C_{2,a}^{\frac{2}{a}}}{C_{1,a}^{\frac{2}{a}}} + 2C_{2,a}^{\frac{2}{a}}\mu_{k} \right) + (\mu_k-\mu_{k+1})G_M^2\left( \frac{1}{C_{1,a}^{\frac{2}{a}}} + 2\mu_k \right)\right] \\
        &- \frac{2\alpha_k \gamma\left(\frac{C_1}{C_2}\right)^{c-1+\frac{2}{a}}}{a\left(1 + \mu_k G_M\right)^2} M_{\mu_{k}}(x_k)^{\frac{a(c-1)}{2}+1} + \left[(\mu_k-\mu_{k+1})G_M^2 + \frac{1}{\mu_k}\right]\frac{\E\sqbr{\norm{x_{k+1}-x_k}^2 | \cF_k}}{2} \\
        &\overset{(c)}{\leq} M_{\mu_k}(x_k)\left[1 + \alpha_k \mu_k C G_M \left( \frac{C_{2,a}^{\frac{2}{a}}}{C_{1,a}^{\frac{2}{a}}} + 2C_{2,a}^{\frac{2}{a}}\mu_{k} \right) + (\mu_k-\mu_{k+1})G_M^2\left( \frac{1}{C_{1,a}^{\frac{2}{a}}} + 2\mu_k \right)\right] \\
        &-\frac{2\alpha_k \gamma\left(\frac{C_1}{C_2}\right)^{c-1+\frac{2}{a}}}{a\left(1 + \mu_k G_M\right)^2} M_{\mu_{k}}(x_k)^{\frac{a(c-1)}{2}+1} + \alpha_k^2\left[(\mu_k-\mu_{k+1})G_M^2 + \frac{1}{\mu_k}\right]\frac{\E\sqbr{\br{C\norm{x_k-x^*}+\norm{w_k}}^2| \cF_k}}{2}
    \end{align*}
    where the inequality (a) follows from Lemma \ref{lemma: moreau-negative-drift-2}, inequality (b) follows from Lemma \ref{lemma: envelope-properties}, inequality (c) follows from applying the triangle inequality, and Assumption \ref{assumption: F-lipschitz} on the noise term. The RHS can be further upper-bounded as:
    \begin{align*}
        E_{k+1}
        &\overset{(d)}{\leq} M_{\mu_k}(x_k)\left[1 + \alpha_k \mu_k C G_M \left( \frac{C_{2,a}^{\frac{2}{a}}}{C_{1,a}^{\frac{2}{a}}} + 2C_{2,a}^{\frac{2}{a}}\mu_{k} \right) + (\mu_k-\mu_{k+1}) G_M^2 \left( \frac{1}{C_{1,a}^{\frac{2}{a}}} + 2\mu_k \right)\right] \\
        &-\frac{2\alpha_k \gamma\left(\frac{C_1}{C_2}\right)^{c-1+\frac{2}{a}}}{a\left(1 + \mu_k G_M\right)^2} M_{\mu_{k}}(x_k)^{\frac{a(c-1)}{2}+1} + \alpha_k^2\left[(\mu_k-\mu_{k+1})G_M^2 + \frac{1}{\mu_k}\right]\frac{(C^2+1)\br{\norm{x_k-x^*}^2+\E\sqbr{\norm{w_k}^2 | \cF_k}}}{2} \\
        &\overset{(e)}{\leq} \left[1 + \alpha_k \mu_k C G_M \left( \frac{C_{2,a}^{\frac{2}{a}}}{C_{1,a}^{\frac{2}{a}}} + 2C_{2,a}^{\frac{2}{a}}\mu_{k} \right) + (\mu_k-\mu_{k+1}) G_M^2 \left( \frac{1}{C_{1,a}^{\frac{2}{a}}} + 2\mu_k \right)\right] M_{\mu_k}(x_k) \\
        &+  \frac{\alpha_k^2(C^2+1)}{2} \left((\mu_k-\mu_{k+1})G_M^2 + \frac{1}{\mu_k}\right) \left( \frac{1}{C_{1,a}^{\frac{2}{a}}} + 2\mu_k \right) M_{\mu_k}(x_k) - \frac{2\alpha_k \gamma\left(\frac{C_1}{C_2}\right)^{c-1+\frac{2}{a}}}{a\left(1 + \mu_k G_M\right)^2} M_{\mu_{k}}(x_k)^{\frac{a(c-1)}{2}+1} \\
        &+ \alpha_k^2 \left[(\mu_k-\mu_{k+1})G_M^2 +  \frac{1}{\mu_k}\right]\frac{(C^2+1)(A+2B\norm{x_k-x^*}^2+2B\norm{x^*}^2)}{2}
    \end{align*}
     where the inequality (d) is obtained from applying Cauchy-Schwarz on the noise term and (e) follows from Assumption \ref{assumption: noise} and Lemma \ref{lemma: envelope-properties}. The RHS can be further upper-bounded as follows:
    \begin{align*}
        E_{k+1}
        &\overset{(f)}{\leq} M_{\mu_k}(x_k)\left[1+ \alpha_k \mu_k C G_M \left( \frac{C_{2,a}^{\frac{2}{a}}}{C_{1,a}^{\frac{2}{a}}} + 2C_{2,a}^{\frac{2}{a}}\mu_{k} \right) + (\mu_k-\mu_{k+1})G_M^2 \left( \frac{1}{C_{1,a}^{\frac{2}{a}}} + 2\mu_k \right)\right] \\
        &+ \frac{\alpha_k^2(C^2+1)(2B+1)}{2} \left((\mu_k-\mu_{k+1})G_M^2 + \frac{1}{\mu_k}\right) \left( \frac{1}{C_{1,a}^{\frac{2}{a}}} + 2\mu_k \right) M_{\mu_k}(x_k) \\
        &- \frac{2\alpha_k \gamma\left(\frac{C_1}{C_2}\right)^{c-1+\frac{2}{a}}}{a\left(1 + \mu_k G_M\right)^2} M_{\mu_{k}}(x_k)^{\frac{a(c-1)}{2}+1} + \frac{\alpha_k^2}{2} \left[(\mu_k-\mu_{k+1})G_M^2 + \frac{1}{\mu_k}\right]\underbrace{(C^2+1)(A+2B\norm{x^*}^2)}_{N_C} \\
        &\overset{(g)}{\leq} M_{\mu_k}(x_k)\left[1+ \alpha_k \mu_k C G_M \left( \frac{C_{2,a}^{\frac{2}{a}}}{C_{1,a}^{\frac{2}{a}}} + 2C_{2,a}^{\frac{2}{a}}\mu_0 \right) + (\mu_k-\mu_{k+1})G_M^2 \left( \frac{1}{C_{1,a}^{\frac{2}{a}}} + 2\mu_0 \right)\right] \\
        &+ \frac{\alpha_k^2(C^2+1)(2B+1) \left(\mu_0^2 G_M^2 + 1\right)}{2\mu_k} \left( \frac{1}{C_{1,a}^{\frac{2}{a}}} + 2\mu_0 \right) M_{\mu_k}(x_k) \\
        &- \frac{2\alpha_k \gamma\left(\frac{C_1}{C_2}\right)^{c-1+\frac{2}{a}}}{a\left(1 + \mu_0 G_M\right)^2} M_{\mu_{k}}(x_k)^{\frac{a(c-1)}{2}+1} + \frac{\alpha_k^2}{\mu_k} \frac{\br{\mu_0^2 G_M^2 + 1} N_C}{2} \\
        &= \left(1 + \alpha_k \mu_k \nu_1 + (\mu_k-\mu_{k+1})\nu_2 + \frac{\alpha_k^2\nu_3}{\mu_k} \right)M_{\mu_k}(x_k) - \alpha_k \nu_4 M_{\mu_k}(x_k)^{\frac{a(c-1)}{2}+1} + \frac{\alpha_k^2}{\mu_k} \frac{\br{\mu_0^2 G_M^2 + 1} N_C}{2},
    \end{align*}
    where the inequality (f) follows from Lemma \ref{lemma: envelope-properties} and the last inequality (g) follows from the observation that $0 <  \mu_k \leq \mu_0$.
\end{proof}
However, a time-varying Moreau envelop as shown above will produce an additional $\mu$-differential term $(\mu_k-\mu_{k+1})\nu_2M_{\mu_k}(x_k)$ in the drift in addition to making the noise term $\frac{\alpha_k^2}{\mu_k} \frac{\br{\mu_0^2 G_M^2 + 1} N_C}{2}$ larger as $k$ gets larger. Thus, we need $\mu_k$ to converge to $0$ slowly enough in order to not make such term explode while ensure a sufficiently good convergence. A rough estimation gives us $\alpha_k \mu_k \approx \frac{\alpha_k^2}{\mu_k} \Leftrightarrow \alpha_k \approx \mu_k^2$ is the choice that offers the best trade-off. From here, we can obtain the finite-time bounds of the algorithm in the non-smooth sub-exponential case as follows.
\begin{theorem}
    \label{finite-time-corollary-3-full}
    Under Assumptions \ref{assumption: noise}, \ref{assumption: F-lipschitz}, \ref{assumption: polynomial-growth}, \ref{assumption: generalized-gradient-bound}, \ref{assumption: bounded-iterates} and \ref{assumption: clarke-asymptotic-stability} hold. With the step size $\alpha_k = \frac{\alpha}{(k+K)^{\xi}}$ and the choice of parameter $\mu_k = \frac{\mu}{(k+K)^{0.5\xi}}$ where:
    \begin{align*}
        c \geq 1, d_{a,c} = a(c-1)/2+1, \xi \leq \frac{2d}{3}, d = \frac{3d_{a,c}}{3d_{a,c}-1} \\
        \Phi = 0.5 d_{a,c} \alpha^{2-\frac{1}{d_{a,c}}} \nu_4^{\frac{1}{d_{a,c}}} \left(\frac{\br{\mu_0^2 G_M^2 + 1} N_C}{2\mu}\right)^{1-\frac{1}{d_{a,c}}}, \omega_* = \frac{d_{a,c} \alpha^2 \br{\mu_0^2 G_M^2 + 1} N_C}{2\mu},
    \end{align*}
    and $N_C, G_M$ are defined as in Proposition \ref{prop: one-iterate-bound-adaptive-mu}, we have:
    \begin{align*}
    &\textbf{For $\xi = \frac{2d}{3}$}: \\
    &E_k \leq E_0 C_{2,a}^{\frac{2}{a}}\br{\frac{1}{C_{1,a}^{\frac{2}{a}}} + 2\mu_0}\left(\frac{K}{k+K}\right)^{\Phi} + 
    \br{\frac{1}{C_{1,a}^{\frac{2}{a}}} + 2\mu_0} \times
    \begin{cases}
        \frac{1}{(k+K)^{\Phi}}\frac{2^d\alpha^{\frac{3}{2}}\omega_*}{d-1-\Phi} &\text{ if } \Phi \in (0,d-1)\\
        \frac{ \log(k+K)}{(k+K)^{d-1}} \times 2^d\alpha^{\frac{3}{2}} \omega_* &\text{ if } \Phi = d-1 \\
        \frac{1}{(k+K)^{d-1}}\frac{e2^d\alpha^{\frac{3}{2}} \omega_*}{\Phi-(d-1)} &\text{ if } \Phi > d-1
    \end{cases}\\
    &\textbf{For $\xi \in \left(0, \frac{2d}{3}\right)$}: \\
    &E_k \leq E_0 C_{2,a}^{\frac{2}{a}}\br{\frac{1}{C_{1,a}^{\frac{2}{a}}} + 2\mu_0} \exp{\left[ -\frac{\Phi ((k+K)^{1-\frac{(3d_{a,c}-1)\xi}{2d_{a,c}}}-K^{1-\frac{(3d_{a,c}-1)\xi}{2 d_{a,c}}}))}{1-\frac{(3d_{a,c}-1)\xi}{2d_{a,c}}} \right]} + \frac{2\alpha^{1.5} \omega_* \br{\frac{1}{C_{1,a}^{\frac{2}{a}}} + 2\mu_0}}{\Phi (k+K)^{\frac{\xi}{2d_{a,c}}}} \qquad\qquad\qquad \\
    &\textbf{For $\xi = 0$}: \\
    &E_k \leq E_0 C_{2,a}^{\frac{2}{a}}\br{\frac{1}{C_{1,a}^{\frac{2}{a}}} + 2\mu_0}(1-\Phi)^k + 
    \frac{2 \omega_* \br{\frac{1}{C_{1,a}^{\frac{2}{a}}} + 2\mu_0}}{\Phi}
    \end{align*}
Here, we denote $E_k = \E[\norm{x_k-x^*}^2] \forall k \in \Z^+$ and $E_0 = \norm{x_0-x^*}^2$.
\end{theorem}
\begin{proof}
Let $E_{k} = \mathbb{E}[M_{\mu_{k}}(x_{k}) | \mathcal{F}_{k-1}]$ and $d_{a,c} = \frac{a(c-1)}{2}+1$. Recall that from Proposition \ref{prop: one-iterate-bound-adaptive-mu}, we have for $k \geq 1$:
\begin{align*}
    E_{k} &\leq \left(1 + \alpha_{k-1} \mu_{k-1} \nu_1 + (\mu_{k-1}-\mu_{k})\nu_2 + \frac{\alpha_{k-1}^2\nu_3}{\mu_{k-1}} \right)M_{\mu_{k-1}}(x_{k-1}) - \alpha_{k-1} \nu_4 M_{\mu_{k-1}}(x_{k-1})^{d_{a,c}} \\
    &+ \frac{\alpha_{k-1}^2}{\mu_{k-1}}\frac{\br{\mu_0^2 G_M^2 + 1} N_C}{2} \\
    &= \left(1 + \frac{\alpha\mu\nu_1 +  \frac{\alpha^2\nu_3}{\mu}}{(k-1+K)^{1.5\xi}} + \frac{\xi \mu \nu_2}{(k-1+K)^{0.5\xi+1}}\right)M_{\mu_{k-1}}(x_{k-1}) - \frac{\alpha\nu_4}{(k-1+K)^{\xi}} M_{\mu_{k-1}}(x_{k-1})^{d_{a,c}} \\
    &+ \frac{\alpha^2 \br{\mu_0^2 G_M^2 + 1} N_C}{2\mu}\frac{1}{(k-1+K)^{1.5\xi}} \\
    &\leq \left(1 + \frac{\alpha\mu\nu_1 + \xi \mu \nu_2 +  \frac{\alpha^2\nu_3}{\mu}}{(k-1+K)^{1.5\xi}}\right)M_{\mu_{k-1}}(x_{k-1}) - \frac{\alpha\nu_4}{(k-1+K)^{\xi}} M_{\mu_{k-1}}(x_{k-1})^{d_{a,c}} \\
    &+ \frac{\alpha^2 \br{\mu_0^2 G_M^2 + 1} N_C}{2\mu} \frac{1}{(k-1+K)^{1.5\xi}}
\end{align*}
since $\xi \leq 1$. Let $\phi = \alpha^{2-\frac{1}{d_{a,c}}} \nu_4^{\frac{1}{d_{a,c}}} \left(\frac{\br{\mu_0^2 G_M^2 + 1} N_C}{2\mu}\right)^{1-\frac{1}{d_{a,c}}} = \frac{Phi}{0.5 d_{a,c}}, \omega = \frac{\alpha^2 \br{\mu_0^2 G_M^2 + 1} N_C}{2\mu} = \frac{\omega_*}{d_{a,c}}$. By AM-GM, we have the bound 
\begin{align*}
    &\br{d_{a,c}-1} \frac{\omega}{(k-1+K)^{1.5\xi}} + \frac{\alpha\nu_4}{(k-1+K)^{\xi}} M_{\mu_{k-1}}(x_{k-1})^{d_{a,c}} \geq d_{a,c} \phi \left( \frac{1}{k-1+K} \right)^{\left(\frac{3}{2}-\frac{1}{2d_{a,c}}\right)\xi} M_{\mu_{k-1}}(x_{k-1})
\end{align*}
and thus we obtain:
\begin{align*}
    E_{k} &\leq \left(1 + \frac{\alpha\mu\nu_1 + \xi \mu \nu_2 +  \frac{\alpha^2\nu_3}{\mu}}{(k-1+K)^{1.5\xi}}\right)M_{\mu_{k-1}}(x_{k-1}) + d_{a,c} \frac{\omega}{(k-1+K)^{1.5\xi}} \\
    &- \frac{\alpha\nu_4}{(k-1+K)^{\xi}}M_{\mu_{k-1}}(x_{k-1})^{d_{a,c}} - (d_{a,c}-1) \frac{\omega}{(k-1+K)^{1.5\xi}} \\
    &\leq \left(1 - \frac{d_{a,c} \phi}{\left(k-1+K \right)^{\left(\frac{3}{2}-\frac{1}{2d_{a,c}}\right)\xi}}  + \frac{\alpha\mu\nu_1 + \xi \mu \nu_2 +  \frac{\alpha^2\nu_3}{\mu}}{(k-1+K)^{1.5\xi}}\right)M_{\mu_{k-1}}(x_{k-1}) + \frac{d_{a,c} \omega}{(k-1+K)^{1.5\xi}} \\
    &\leq \left( 1 - \frac{d_{a,c} \phi}{2\left(k-1+K \right)^{\left(\frac{3}{2}-\frac{1}{2d_{a,c}}\right)\xi}} \right)M_{\mu_{k-1}}(x_{k-1}) + \frac{d_{a,c} \omega}{(k-1+K)^{1.5\xi}}.
\end{align*}
From here, we expand this one iterate bound and obtain the inequality:
\begin{align*}
    E_{k} &\leq \underbrace{\prod_{i = 0}^{k-1} \left( 1 - \frac{d_{a,c} \phi}{2\br{i+K}^{\left(\frac{3}{2}-\frac{1}{2d_{a,c}}\right)\xi}} \right)}_{T_1} M_{\mu_0}(x_0) + d_{a,c} \underbrace{\sum_{i = 0}^{k-1} \frac{\omega}{(i+K)^{1.5\xi}} \prod_{j = i+1}^{k-1} \left( 1 - \frac{d_{a,c} \phi}{2\br{j+K}^{\left(\frac{3}{2}-\frac{1}{2d_{a,c}}\right)\xi}} \right)}_{T_2}.
\end{align*}
We bound the terms $T_1, T_2$ similarly to other cases. First, we deal with $T_1$. If $\xi = \frac{2d}{3}$ then $T_1 \leq \left(\frac{K}{k+K}\right)^{\phi}$. Otherwise, if $\xi \in \br{0, \frac{2d}{3}}$ (that is $\xi' = \frac{(3d_{a,c}-1)\xi}{2d_{a,c}} \in (0,1)$) then: 
\begin{align*}
    &T_1 \leq \exp{\left[ -\frac{ d_{a,c}\phi((k+K)^{1-\xi'}-K^{1-\xi'}))}{2(1-\xi')} \right]} = \exp{\left[ -\frac{d_{a,c} \phi\br{(k+K+1)^{1-\frac{(3d_{a,c}-1)\xi}{2d_{a,c}}}-K^{1-\frac{(3d_{a,c}-1)\xi}{2d_{a,c}}})}}{2-\frac{(3d_{a,c}-1)\xi}{d_{a,c}}} \right]}.
\end{align*}
For $T_2$, if $\xi = \frac{2d}{3}$ where $d = \frac{3d_{a,c}}{3d_{a,c}-1}$ then we have:
\begin{align*}
    \sum_{i = 0}^{k-1} \frac{1}{(i+K)^{1.5\xi}} \prod_{j = i+1}^{k-1} \left( 1 - \frac{d_{a,c} \phi}{2\br{j+K}^{\left(\frac{3d_{a,c}-1}{2d_{a,c}}\right)\xi}} \right) &= \alpha^{\frac{3}{2}} \sum_{i = 0}^{k-1} \frac{1}{(i+K)^d} \prod_{j = i+1}^{k-1} \left( 1 - \frac{d_{a,c} \phi}{2(j+K)} \right) \\
    &\leq \alpha^{\frac{3}{2}} \sum_{i = 0}^{k-1} \frac{1}{(i+K)^d} \exp \left( -\frac{d_{a,c}\phi}{2} \sum_{j = i+1}^{k-1} \frac{1}{j + K} \right) \\
    &\leq \alpha^{\frac{3}{2}} \sum_{i = 0}^{k-1} \frac{1}{(i+K)^d} \exp \left( -\frac{d_{a,c}\phi}{2} \int_{j = i+1}^{k} \frac{1}{j + K} \right) \\
    &\leq \alpha^{\frac{3}{2}} \sum_{i = 0}^{k-1} \frac{1}{(i+K)^d} \left( \frac{i+1+K}{k+K} \right)^{\frac{d_{a,c}\phi}{2}}  \\
    &\leq \frac{2^d\alpha^{\frac{3}{2}}}{(k+K)^{\frac{d_{a,c}\phi}{2}}} \sum_{i = 0}^{k-1} \frac{1}{(i+1+K)^{d-\frac{d_{a,c}\phi}{2}}}.
\end{align*}
In summary, when $\xi = \frac{2d}{3}$, we have the following bounds that correspond to specific cases of $\phi$ as follows
\begin{itemize}
    \item When $\phi \in \br{0, \frac{2(d-1)}{d_{a,c}}}$ then $\sum_{i = 0}^{k-1} \frac{1}{(i+1+K)^{d-\frac{d_{a,c}\phi}{2}}} \leq \frac{1}{d-1-\frac{d_{a,c}\phi}{2}}$
    \item When $\phi = \frac{2(d-1)}{d_{a,c}}$ then $\sum_{i = 0}^{k-1} \frac{1}{(i+1+K)^{d-\frac{d_{a,c}\phi}{2}}} \leq \log(k+K)$
    \item When $\phi \in \br{\frac{2(d-1)}{d_{a,c}}, \frac{2d}{d_{a,c}}}$ then $\sum_{i = 0}^{k-1} \frac{1}{(i+1+K)^{d-\frac{d_{a,c}\phi}{2}}} \leq \frac{(k+K)^{\phi-(d-1)}}{\frac{d_{a,c}\phi}{2}-(d-1)}$
    \item When $\phi = \frac{2d}{d_{a,c}}$ then $\sum_{i = 0}^{k-1} \frac{1}{(i+1+K)^{d-\frac{d_{a,c}\phi}{2}}} = k$.
    \item When $\phi > \frac{2d}{d_{a,c}}$ then $\sum_{i = 0}^{k-1} \frac{1}{(i+1+K)^{d-\frac{d_{a,c}\phi}{2}}} \leq \frac{e\br{k+K}^{\frac{d_{a,c}\phi}{2}-(d-1)}}{\frac{d_{a,c}\phi}{2}-(d-1)}$
\end{itemize}
If $\xi \in \br{0, \frac{2d}{3}}$ (that is $\xi' = \frac{(3d_{a,c}-1)\xi}{2d_{a,c}} = \frac{3\xi}{2d} \in (0,1)$) then by induction we have $T_2 \leq \frac{4 \alpha^{\frac{3}{2}} \omega}{\phi \br{k+K}^{\frac{\xi}{2d_{a,c}}}}$, which gives
\begin{align*}
    E_k \leq M_{\mu_0}(x_0) \exp{\left[ -\frac{\phi((k+K)^{1-\frac{3\xi}{2d}}-K^{1-\frac{3\xi}{2d}}))}{1-\frac{3\xi}{2d}} \right]} + \frac{4 \alpha^{\frac{3}{2}} \omega}{\phi \br{k+K}^{\frac{\xi}{2d_{a,c}}}}.
\end{align*}
Plugging back $\Phi = 0.5 d_{a,c} \phi, \omega_* = d_{a,c} \omega$ and proceed similarly to the proof of Theorem \ref{finite-time-corollary-2}, we have the finite time bounds:
\begin{align*}
    E_{k} \leq 
    \begin{cases}
        M_{\mu_0}(x_0)\left(\frac{K}{k+K}\right)^{\Phi} + \frac{2^d\alpha^{\frac{3}{2}}}{(k+K)^{\Phi}}\frac{\omega_*}{d-1-\Phi} &\text{ if } \xi =  \frac{2d}{3}, \Phi \in \br{0, d-1}\\
        M_{\mu_0}(x_0)\left(\frac{K}{k+K}\right)^{d-1} + \frac{2^d\alpha^{\frac{3}{2}} \omega_* \log(k+K)}{(k+K)^{d-1}} &\text{ if } \xi = \frac{2d}{3}, \Phi = d-1 \\
        M_{\mu_0}(x_0)\left(\frac{K}{k+K}\right)^{\Phi} + \frac{2^d\alpha^{\frac{3}{2}}}{(k+K)^{d-1}}\frac{e \omega_*}{\Phi-(d-1)} &\text{ if } \xi = \frac{2d}{3}, \Phi \in \br{d-1, \infty} \\
        M_{\mu_0}(x_0) \exp{\left[ -\frac{\phi((k+K)^{1-\frac{(3d_{a,c}-1)\xi}{2d_{a,c}}}-K^{1-\frac{(3d_{a,c}-1)\xi}{2c}}))}{1-\frac{(3d_{a,c}-1)\xi}{2d_{a,c}}} \right]} + \frac{2 \alpha^{\frac{3}{2}} \omega_*}{\Phi \br{k+K}^{\frac{\xi}{2d_{a,c}}}} &\text{ if } \xi \in \left(0, \frac{2d}{3}\right) \\
        M_{\mu_0}(x_0)(1-\Phi)^k + \frac{\omega_*}{\Phi} &\text{ if } \xi = 0, \Phi \in \br{0, 1}
    \end{cases}
\end{align*}
Lastly, from Lemma \ref{lemma: envelope-properties}, we have that $M_{\mu_0}(x_0) \leq C_{2,a}^{\frac{2}{a}}\norm{x_0-x^*}^2$ and $M_{\mu_k}(x_k) \geq \frac{\norm{x_k-x^*}^2}{\frac{1}{C_{1,a}^{\frac{2}{a}}} + 2\mu_k} \geq \frac{\norm{x_k-x^*}^2}{\frac{1}{C_{1,a}^{\frac{2}{a}}} + 2\mu_0}$. Furthermore, we have that $0 < \mu_k \leq \mu_0$ by our choice of $\mu_k$. Hence, we have:
\begin{align*}
    &\textbf{For $\xi = \frac{2d}{3}$}: \\
    &\E\sqbr{\norm{x_k-x^*}^2} \leq C_{2,a}^{\frac{2}{a}}\br{\frac{1}{C_{1,a}^{\frac{2}{a}}} + 2\mu_0}\norm{x_0-x^*}^2\left(\frac{K}{k+K}\right)^{\Phi} \\
    &\qquad + 
    \br{\frac{1}{C_{1,a}^{\frac{2}{a}}} + 2\mu_0} \times
    \begin{cases}
        \frac{1}{(k+K)^{\Phi}}\frac{2^d\alpha^{\frac{3}{2}}\omega_*}{d-1-\Phi} &\text{ if } \xi =  \frac{2d}{3}, \Phi \in (0,d-1)\\
        \frac{ \log(k+K)}{(k+K)^{d-1}} \times 2^d\alpha^{\frac{3}{2}} \omega_* &\text{ if } \xi = \frac{2d}{3}, \Phi = d-1 \\
        \frac{1}{(k+K)^{d-1}}\frac{e2^d\alpha^{\frac{3}{2}} \omega_*}{\Phi-(d-1)} &\text{ if } \xi = \frac{2d}{3}, \Phi > d-1
    \end{cases}\\
    &\textbf{For $\xi \in \left(0, \frac{2d}{3}\right)$}: \\
    &\E\sqbr{\norm{x_k-x^*}^2} \leq C_{2,a}^{\frac{2}{a}}\br{\frac{1}{C_{1,a}^{\frac{2}{a}}} + 2\mu_0}\norm{x_0-x^*}^2 \exp{\left[ -\frac{\Phi\br{(k+K)^{1-\frac{3\xi}{2d}}-K^{1-\frac{3\xi}{2d}})}}{1-\frac{3\xi}{2d}} \right]} \qquad\qquad\qquad\qquad \\
    &+ \frac{2\alpha^{1.5} \omega_* \br{\frac{1}{C_{1,a}^{\frac{2}{a}}} + 2\mu_0}}{\Phi (k+K)^{\frac{\xi}{2d_{a,c}}}} \qquad\\
    &\textbf{For $\xi = 0$}: \\
    &\E\sqbr{\norm{x_k-x^*}^2} \leq C_{2,a}^{\frac{2}{a}}\br{\frac{1}{C_{1,a}^{\frac{2}{a}}} + 2\mu_0}\norm{x_0-x^*}^2(1-\Phi)^k + 
    \frac{\omega_*\br{\frac{1}{C_{1,a}^{\frac{2}{a}}} + 2\mu_0}}{\Phi}
\end{align*}
where $\Phi = \frac{d_{a,c} \alpha^{2-\frac{1}{d_{a,c}}} \nu_4^{\frac{1}{d_{a,c}}} \left(\frac{\br{\mu_0^2 G_M^2 + 1} N_C}{2\mu}\right)^{1-\frac{1}{d_{a,c}}}}{2} = 0.5 d_{a,c} \phi, \omega_* = \frac{d_{a,c} \alpha^2 \br{\mu_0^2 G_M^2 + 1} N_C}{2\mu} = d_{a,c} \omega$. This completes the proof.
\end{proof}

\subsection{Proof for almost surely convergence results}
\label{ssec: proof-almost-sure}
In this subsection, we present the proofs for the almost surely convergence results, which follow the finite-time bounds as corollaries. Besides having different step size conditions, the proof of these corollaries are mostly similar. Here, we apply the Supermartingale Convergence Theorem from \cite{neurodynamic} on the one-iterate bound in each setting, and the argument follows quickly from \cite{zaiwei-envelope}. We present the proofs for all settings for completeness.


\subsubsection{Proof of \ref{almost-sure-convergence-1}}
\label{sssec:almost-sure-convergence-1-proof}
Recall that from Theorem \ref{prop: one-iterate-bound-exponential}, we have:
    \begin{align*}
        \mathbb{E}[M_\mu(x_{k+1}) | \mathcal{F}_k] \leq \left( 1-\frac{\alpha_k \gamma_M}{2} \right)M_\mu(x_k) + \frac{\alpha_k^2(A+2B \norm{x^*}^2)}{\mu}
    \end{align*}
Using the Supermartingale Convergence Theorem \cite{neurodynamic} and fix the sample path, we have that there exists a non-negative random variable $M_\infty$ such that $\lim_{k \rightarrow \infty} M_\mu(x_k) = M_\infty$ almost surely and $\sum_{k = 1}^{\infty} \alpha_k \gamma_M M_\mu(x_k) < +\infty$ almost surely. Thus, what is left to us is to prove that $M_\infty = 0$ almost surely. Suppose the contrary that $M_\infty > 0$, then there exists $\delta > 0$ and $k_0 \in \mathbb{Z}^+$ such that $M_\mu(x_k) \geq \delta/2 \forall k \geq k_0$. 
However, since we have:
\begin{align*}
    \sum_{k = 1}^{\infty} \alpha_k \gamma_M M_{\mu}(x_k) \geq \delta' \gamma_M \sum_{k = 1}^{\infty} \alpha_k = \infty
\end{align*}
which is a contradiction to the fact that $\sum_{k = 1}^{\infty} \alpha_k \gamma_M M_\mu(x_k) < +\infty$ holds almost surely. Hence, we must have that $M_\mu(x_k)$ converges to $0$ almost surely. From Lemma \ref{lemma: envelope-properties} and \ref{assumption: polynomial-growth}, we must have that $x_k$ converges to $x^*$ almost surely.

\subsubsection{Proof of \ref{almost-sure-convergence-2}}
\label{sssec:almost-sure-convergence-2-proof}
Similarly, we obtain the one iterate bounds from Theorem \ref{main-theorem-2} as follows:
\begin{align*}
    V(x_{k+1}) \leq V(x_k)-\alpha_k^{2-\frac{1}{c}}A^{1-\frac{1}{c}}\gamma^{\frac{1}{c}} V(x_k) + \alpha_k^2L(2B\norm{x^*}^2 + cA)
\end{align*}
From the assumptions that $\sum_{k=1}^{\infty}\alpha_k^{2-\frac{1}{c}} = \infty, \alpha_k^2 < +\infty$, we have that 
$L(2B\norm{x^*}^2 + cA) \sum_{i=1}^{\infty} \alpha_k^2 < \infty$. Thus, by the Supermatingale Convergence Theorem, we have that $V(x_k)$ converges to some non-negative random variable $V_\infty$ almost surely. Using arguments similar to the proof of \ref{almost-sure-convergence-1} and from the assumption $\sum_{k=1}^{\infty}\alpha_k^{2-\frac{1}{c}} = \infty$, we have that $V_\infty = 0$. From Lemma \ref{assumption: quadratic-growth}, we must have that $x_k$ converges to $x^*$ almost surely. Hence proved.

\subsubsection{Proof of \ref{almost-sure-convergence-3}}
\label{sssec:almost-sure-convergence-3-proof}
From Proposition \ref{prop: one-iterate-bound-adaptive-mu} and let $d_{a,c} = \frac{a(c-1)}{2}+1$, we have the following bound:
\begin{align*}
    \mathbb{E}[M_{\mu_{k+1}}(x_{k+1}) | \mathcal{F}_k] &\leq \left(1 + \alpha_k \mu_k \nu_1 + (\mu_k-\mu_{k+1})\nu_2 + \frac{\alpha_k^2\nu_3}{\mu_k} \right)M_{\mu_k}(x_k) \\
        &- \alpha_k \nu_4 M_{\mu_k}(x_k)^{\frac{a(c-1)}{2}+1} + \frac{\alpha_k^2}{\mu_k} \frac{\br{\mu_0^2 G_M^2 + 1} N_C}{2}.
\end{align*}
By selecting $\{\alpha_k\}, \{\mu_k\}$ such that $\alpha_k = \mu_k^2$ and $\mu_k = \frac{\mu}{(k+K)^{0.5\xi}}$ where $\xi \in \left(\frac{2}{3},\frac{2d}{3}\right]$, this implies $\sum_{k=1}^{\infty}\alpha_k^{\frac{3d_{a,c}-1}{2d_{a,c}}} = \infty, \sum_{k=1}^{\infty} \alpha_k^{\frac{3}{2}} < +\infty$. Using the same analysis as the proof of Corollary \ref{finite-time-corollary-3}, the one iterate bound now becomes:
\begin{align*}
    \mathbb{E}[M_{\mu_{k+1}}(x_{k+1}) | \mathcal{F}_k] \leq 
    \left( 1 - \frac{d_{a,c} \alpha^{2-\frac{1}{d_{a,c}}} \nu_4^{\frac{1}{d_{a,c}}} \left(\frac{\br{\mu_0^2 G_M^2 + 1} N_C}{2\mu}\right)^{1-\frac{1}{d_{a,c}}}}{2\left(k-1+K \right)^{\left(\frac{3}{2}-\frac{1}{2d_{a,c}}\right)\xi}} \right)M_{\mu_{k-1}}(x_{k-1}) + \frac{d_{a,c} \alpha^2 \br{\mu_0^2 G_M^2 + 1} N_C}{2\mu (k-1+K)^{1.5\xi}}.
\end{align*}
for $0 \leq \xi \leq \frac{2d}{3}$. From the Supermartingale Convergence Theorem and from:
\begin{align*}
    \sum_{k=1}^{\infty}\alpha_k^{\frac{3d_{a,c}-1}{2d_{a,c}}} = \infty,\sum_{k=1}^{\infty} \alpha_k^{\frac{3}{2}} < +\infty,
\end{align*}
we have that the iterates $x_k$ also converge almost surely to $0$ by similar arguments as above. Hence proved.

\subsection{Proof of Central Limit Theorem \ref{thm: clt-exponential}}
\label{ssec: clt-proof}
\textcolor{black}{
\begin{proof}
    To prove this result, we shall apply Theorem 5 from \cite{kontoyiannisborkar2024odemethodasymptoticstatistics} and we shall proceed by verifying the assumptions of the theorem. Note that Assumption \textbf{(DV3)} holds from Assumption \ref{assumption: dv3} and Assumption \textbf{(A1)} in \cite{kontoyiannisborkar2024odemethodasymptoticstatistics} holds with step size $\alpha_k = \frac{\alpha}{(k+K)^\xi}$ for $\xi \in (1/2,1]$, Assumption \textbf{(A2)} in \cite{kontoyiannisborkar2024odemethodasymptoticstatistics} is a generalization of Assumption \ref{assumption: F-lipschitz} where $L(x) = \max\{C, \norm{F(0)}\} \, \forall x \in \R^d$. Next, Assumption \textbf{(A3)} in \cite{kontoyiannisborkar2024odemethodasymptoticstatistics} is satisfied from Assumption \ref{assumption: ode-vector-field}. Note that from the chain rule, Assumption \ref{assumption: clarke-time-derivative} and Lemma \ref{lemma: rescaled-time-derivative-condition} gives us
    \begin{align}
        \label{eqn: scaled-drift}
        \langle g_x, r^{-1} F(r x) \rangle = \langle g_{rx}, F(r x) \rangle \leq -\frac{2\gamma}{a} R(r x)
    \end{align}   
    for $g_x \in \partial R(x)$ and $r > 0$. This implies that the ODE $\dot{x} = F_\infty(x)$ is globally asymptotically stable, and so Assumption \textbf{(A3)} is satisfied. For Assumption \textbf{(A4b)}, note that we have $\lim_{r \rightarrow \infty} \sup_{x \in \R^d} \frac{L(x)}{\max\{r, Q(x)\}} = \lim_{r \rightarrow \infty} \sup_{x \in \R^d} \frac{\max\{C, \norm{F(0)}\}}{\max\{r, Q(x)\}} = 0$. Moreover, \eqref{eqn: scaled-drift} implies that the ODE $\dot{x} = F_\infty(x)$ admits exponential stability with rate $\frac{2\gamma}{a}$, which gives $C_{1,a}^{\frac{2}{a}} \norm{x(t)-x^*}^2 \leq R(x(t)) \leq e^{-\frac{2\gamma t}{a}} R(x(0)) \leq C_{2,a}^{\frac{2}{a}} e^{-\frac{2\gamma t}{a}} \norm{x(0)-x^*}^2$. This implies the relaxation time $T_r \leq \frac{a\log 2\br{\frac{C_{2,a}}{C_{1,a}}}^{a^{-1}}}{\gamma} \leq \frac{1}{4\max\{C, \norm{F(0)}\}}$ from our assumption and so Assumption \textbf{(A4a)} is satisfied. Finally, we have \textbf{(A5b)} in \cite{kontoyiannisborkar2024odemethodasymptoticstatistics} is satisfied since we have chosen $\xi = 1$. Since all assumptions of Theorem 5 from \cite{kontoyiannisborkar2024odemethodasymptoticstatistics} are satisfied, the desired CLT result follows.
\end{proof}}

\subsection{Remark on the projection step}
\label{ssec:projection-justification}
In this subsection, we provide an analysis of the projection step in order to guarantee that Assumption \ref{assumption: bounded-iterates} holds for all iterates. For this task, the idea is to show that the one-step contraction bound on the Lyapunov function $V$ still holds under the projection step onto the ball $K_1 r$. To do this, we choose $r$ sufficiently large such that $K_* r \geq \norm{x^*}$ and $\alpha_k \leq \alpha_0 \leq \min\left\lbrace \frac{K_\gamma}{\gamma}, \frac{K_\alpha r}{\br{7Cr/6 + \sqrt{A}+D\sqrt{B}}}
 \right\rbrace$ where $K_1, K_*, K_\alpha, K_\gamma$ are appropriately chosen constants. If $\norm{x_k+\alpha_k (F(x_k)+w_k)} \leq r$ then we obtain the one-step bound
\begin{align*}
    V(x_{k+1}) \leq \br{1-\frac{\alpha_k^{c'} \gamma}{2}} V(x_k) + O(\alpha_k^2).
\end{align*}

Let $y_{k+1} = x_k+\alpha_k (F(x_k)+w_k)$, assuming that $\norm{y_{k+1}} > r$, we project this onto the ball radius $K_1 r$. Let this projected point be $x_{k+1}$, we have
\begin{align*}
    \norm{x_{k+1}-x^*} \leq \norm{x_{k+1}} + \norm{x^*} \leq \frac{r (K_1 + K_*)\br{1-\frac{\alpha_k \gamma}{2}}}{\br{1 - \frac{K_\gamma}{2}}}
\end{align*}
since we have chosen $\alpha_k \leq \frac{K_\gamma}{\gamma}$. From Assumption \ref{assumption: polynomial-growth}, this gives
\begin{align}
    \label{eqn: v-bound-projection-1}
    V(x_{k+1}) \leq C_2^{\frac{2}{a}} \norm{x_{k+1}-x^*}^2 \leq C_2^{\frac{2}{a}} \frac{r^2 (K_1 + K_*)^2 \br{1-\frac{\alpha_k \gamma}{2}}^2}{\br{1 - \frac{K_\gamma}{2}}^2}.
\end{align}
On the other hand, we have:
\begin{align*}
    r &\leq \E\sqbr{\norm{x_k+\alpha_k (F(x_k)+w_k)} | \cF_k} \\
    &\leq \norm{x_k} + \alpha_k \br{\norm{F(x_k)} + \E\sqbr{\norm{w_k} | \cF_k}} \\
    &\leq \norm{x_k} + \alpha_k \br{C\norm{x_k-x^*} + \E\sqbr{\norm{w_k} | \cF_k}} \\
    &\leq \norm{x_k} + \alpha_k \br{C\norm{x_k}+C\norm{x^*} + \sqrt{A+B\norm{x_k}^2}} \\
    &\leq \norm{x_k} + \alpha_k \br{7Cr/6 + \sqrt{A}+r\sqrt{B}}
\end{align*}
From the choice of $\alpha_k$, we have $\alpha_k \leq \frac{K_\alpha r}{\br{7Cr/6 + \sqrt{A}+r\sqrt{B}}}$ which implies $\norm{x_k} \geq \br{1-K_\alpha} r$. And so, we have $\norm{x_k-x^*} \geq \norm{x_k}-\norm{x^*} \geq \br{1-K_\alpha-K_*} r$.
From \ref{eqn: v-bound-projection-1} and Assumption \ref{assumption: polynomial-growth}, we have 
\begin{align}
    \label{eqn: one-step-v-bound-with-projection}
    V(x_{k+1}) \leq C_K \br{1-\frac{\alpha_k \gamma}{2}}^2 V(x_k) \leq \br{1-\frac{\alpha_k^{c'} \gamma}{2}} V(x_k).
\end{align}
where we chose $K_1, K_*, K_\alpha, K_\gamma$ such that $C_K =  \br{\frac{C_2}{C_1}}^{\frac{2}{a}} \br{\frac{K_1+K_*}{\br{1-\frac{K_\gamma}{2}} (1-K_\alpha-K_*)}}^2 < 1$. Combine these together, we have
\begin{align*}
    \E[V(x_{k+1})] &= \E\sqbr{V(x_k) | \norm{y_{k+1}-x^*} \leq r} + \E\sqbr{V(x_k) | \norm{y_{k+1}-x^*} > r} \\
    &\leq \sqbr{\br{1-\frac{\alpha_k^{c'} \gamma}{2}} \norm{x_k-x^*}^2 + O(\alpha_k^2)} \P\sqbr{\norm{y_{k+1}-x^*} \leq r} \\
    &+ \br{1-\frac{\alpha_k \gamma}{2}} \norm{x_k-x^*}^2 \P\sqbr{\norm{y_{k+1}-x^*} > r}\\
    &\leq \br{1-\frac{\alpha_k^{c'} \gamma}{2}} \norm{x_k-x^*}^2 + O(\alpha_k^2)
\end{align*}
for $c' \geq 1$. And so, this projection step ensures the one-iterate bound holds for all iterations $k$. For the non-smooth subexponential case, we can apply the analysis similarly.

Now, to choose $K_1, K_*, K_\alpha, K_r$ such that $C_K < 1$, we choose $K_1 = K_\alpha = K_* = \frac{1}{8}\br{\frac{C_1}{C_2}}^{\frac{1}{a}}, K_\gamma = 1$, which gives
\begin{align}
    C_K = \br{\frac{C_2}{C_1}}^{\frac{2}{a}} \br{\frac{\frac{1}{4} \br{\frac{C_1}{C_2}}^{\frac{1}{a}}}{\frac{1}{2} \br{1 - \frac{1}{4} \br{\frac{C_1}{C_2}}^{\frac{1}{a}}}}}^2 < \frac{4}{9} < 1
\end{align}
where $C_1 \leq C_2$ from Assumption \ref{assumption: polynomial-growth}.

\textbf{Discussion on the projection radius}: In our projection routine, we require the knowledge of $\norm{x^*}$, which is usually unknown to us and there is no obvious way to estimate this quantity theoretically. On the other hand, we can intuitively decide whether our projection region covers $x^*$ by observing whether we have to invoke the projection multiple times or observing the drift: If the iterate is close to the boundary but we still have a strong drift then it suggests that $x^*$ is not covered in the projection radius. Moreover, from our almost sure convergence results, it is extremely unlikely to invoke the projection routine multiple times if $x^*$ is contained in the projection radius. And so, one can iteratively enlarge the projection radius and projection threshold $r$ if the projection routine is invoked multiple times.
\subsection{Verification of the negative drift condition}
\label{ssec: negative-drift-verification}
In this section, we shall prove the negative drift condition in the non-smooth sub-exponential convergence setting in Subsection \ref{ssec: nonsmooth-subexponential-experimental-setting}.

Let $x = (x_1,x_2) \in \R^2$. Consider the first case where $x_1 \neq 0$, we have $V(x_1,x_2) = \sqrt{4x_1^2+3x_2^2} - |x_1|$ and 
\begin{align*}
    \nabla V(x_1,x_2) = \sqbr{\frac{4x_1}{\sqrt{4x_1^2+3x_2^2}} - \sign(x_1), \frac{3x_2}{\sqrt{4x_1^2+3x_2^2}}}.
\end{align*}
From here, we can verify the negative drift condition as follows:
\begin{align*}
    \dot{V}(x_1,x_2) &= \langle \nabla V(x_1,x_2), \dot{x} \rangle \\
    &= \left\langle \sqbr{\frac{4x_1}{\sqrt{4x_1^2+3x_2^2}} - \sign(x_1), \frac{3x_2}{\sqrt{4x_1^2+3x_2^2}}}, \sqbr{(x_1^2-x_2^2)u(x_1,x_2), 2x_1 x_2 u(x_1,x_2)} \right\rangle \\
    &= \frac{4x_1(x_1^2-x_2^2)u(x_1,x_2) + 6x_1x_2^2u(x_1,x_2)}{\sqrt{4x_1^2+3x_2^2}} - (x_1^2-x_2^2)u(x_1,x_2) \sign(x_1) \\
    &= -\frac{2|x_1|(2x_1^2+x_2^2)}{\sqrt{4x_1^2+3x_2^2}} + (x_1^2-x_2^2) \text{ since } u(x_1,x_2)\sign(x_1) = -1.
\end{align*}
We shall prove that $\dot{V}(x_1,x_2) \leq -\frac{1}{15} V(x_1,x_2)^2$ for all $(x_1,x_2)$ such that $x_1 \neq 0$. Indeed, we have
\begin{align*}
    \dot{V}(x_1,x_2) &\leq -\frac{1}{15} V(x_1,x_2)^2 \\
    \Leftrightarrow -\frac{2|x_1|(2x_1^2+x_2^2)}{\sqrt{4x_1^2+3x_2^2}} + (x_1^2-x_2^2) &\leq - \frac{(\sqrt{4x_1^2+3x_2^2} - |x_1|)^2}{15} \\
    \Leftrightarrow -\frac{2|x_1|(2x_1^2+x_2^2)}{\sqrt{4x_1^2+3x_2^2}} + (x_1^2-x_2^2) &\leq - \frac{5x_1^2+3x_2^2 - 2|x_1|\sqrt{4x_1^2+3x_2^2}}{15} \\
    \Leftrightarrow 5x_1^2 &\leq \frac{|x_1|(17x_1^2+9x_2^2)}{\sqrt{4x_1^2+3x_2^2}} + 3x_2^2 \\
    \Leftrightarrow 9x_1^2 &\leq \frac{|x_1|(17x_1^2+9x_2^2)}{\sqrt{4x_1^2+3x_2^2}} + (4x_1^2+3x_2^2).
\end{align*}
The last inequality is true since
\begin{align*}
    \frac{|x_1|(17x_1^2+9x_2^2)}{\sqrt{4x_1^2+3x_2^2}} + (4x_1^2+3x_2^2) &= \frac{|x_1|(17x_1^2+9x_2^2)}{2\sqrt{4x_1^2+3x_2^2}} + \frac{|x_1|(17x_1^2+9x_2^2)}{2\sqrt{4x_1^2+3x_2^2}} + (4x_1^2+3x_2^2) \\
    &\geq 3 \sqrt[3]{\frac{|x_1|(17x_1^2+9x_2^2)}{2\sqrt{4x_1^2+3x_2^2}} \times \frac{|x_1|(17x_1^2+9x_2^2)}{2\sqrt{4x_1^2+3x_2^2}} \times (4x_1^2+3x_2^2)} \text{ from AM-GM}\\
    &= 3 \sqrt[3]{\frac{x_1^2 (17x_1^2+9x_2^2)^2}{4}} \\
    &\geq 3 \sqrt[3]{\frac{289 x_1^6}{4}} \\
    &\geq 9x_1^2.
\end{align*}
The last inequality is true because $3 \sqrt[3]{\frac{289}{4}} > 9$. Thus, we have Assumption \ref{assumption: clarke-asymptotic-stability} holds for $c = 2, \gamma = \frac{1}{15}$ for all $(x_1,x_2)$ such that $x_1 \neq 0$.

When $x_1 = 0$, the gradient of $V(x_1,x_2)$ does not exist but instead, we can compute the Clarke generalized gradient. In this case, we have the sub-exponential stability Assumption \ref{assumption: clarke-asymptotic-stability} holds as well. Thus, we have Assumption \ref{assumption: clarke-asymptotic-stability} holds for $c = 2, \gamma = \frac{1}{15}$ for all $(x_1,x_2)$.
\end{document}